\documentclass[a4paper, english, 11pt]{amsart}
\usepackage[top = 3cm, bottom = 3cm, left = 3cm, right = 3cm]{geometry}
\usepackage[english]{babel}
\usepackage{lmodern}

\usepackage{amssymb}
\usepackage{amsmath}
\usepackage{array}
\usepackage{amsthm}
\usepackage{bbm}
\usepackage{hyperref}
\usepackage{tikz-cd}
\usepackage{upgreek}
\usetikzlibrary{babel}
\usepackage{tcolorbox}
\usepackage{enumitem}

\usepackage{import}
\usepackage{xifthen}
\usepackage{pdfpages}
\usepackage{color,transparent}
\usepackage{graphicx}
\usepackage{calc}
\usepackage{subcaption}

\newtheorem{thm}{Theorem}[section]

\newtheorem*{thmasterisco}{Theorem}
\newtheorem*{propasterisco}{Proposition}
\newtheorem*{lemmaasterisco}{Lemma}

\newtheorem{manualtheoreminner}{Theorem}
\newenvironment{thmsinnum}[1]{%
	\renewcommand\themanualtheoreminner{#1}%
	\manualtheoreminner
}{\endmanualtheoreminner}

\newtheorem{prop}[thm]{Proposition}
\newtheorem{lemma}[thm]{Lemma}
\newtheorem{cor}[thm]{Corollary}

\newtheorem*{claimsinnum}{Claim}
\theoremstyle{definition}

\newtheorem{rmk}{Remark}[section]

\newcommand{\cl}{\mathrm{Cl}}

\title{Annular Chaos for non-wandering homeomorphisms}
\author{Alejandro Passeggi}
\address{Facultad de Ciencias, UdelaR, Montevideo Uruguay; apasseggi@cmat.edu.uy.}
\author{Favio Pir\'an}
\address{Facultad de Ingenier\'ia, UdelaR, Montevideo Uruguay; fpiran@fing.edu.uy.}

\begin{document}

\begin{abstract}
We study topological conditions ensuring the presence of rotational chaos for non-wandering or area-preserving annular homeomorphisms. Compared to previous criteria, our main result provides a simpler alternative that avoids the need to locate periodic points, requiring only knowledge of the behavior of certain open sets. This feature is crucial for enabling concrete applications to Poincaré return maps arising in Hamiltonian systems with two degrees of freedom.

The resulting topological criterion admits straightforward numerical implementation: a computer can verify all the required conditions using a simple algorithm that relies solely on basic data from the map. We illustrate this approach with the so-called \emph{driven pendulum}.
\end{abstract}

\maketitle

\section{Introduction}

The study of chaos is one of the central concerns in dynamical systems: how to define it, establish its existence, and describe the structure underlying this phenomenon. The theory of surface dynamics has a strong historical connection with this goal, dating back to Poincaré's study of the three-body problem and the Van der Pol equation, where surfaces naturally appear as return sections of the associated flow\footnote{Extended flow in the non-autonomous case.}. As in certain restricted versions of the three-body problem and other renowned problems in mechanics, we focus on the study of conservative dynamics on the annulus.

\medskip

Despite the large and well-developed body of models for chaotic systems, classical techniques for proving chaos have often proved difficult to implement in concrete systems. As a consequence, many physically motivated examples still lack a rigorous proof of chaotic behavior---despite a vast literature in which numerical simulations suggest, or even non-rigorously claim, its presence. This issue goes beyond the matter of definition: it reveals how limited our understanding remains of the dynamics in key situations that originally inspired the field. Indeed, if we cannot even rigorously establish the existence of chaos in such systems, we are even further from identifying a geometric structure that supports it.
See the detailed discussion in the introduction of \cite{passeggi2023weak}.

\medskip

Recently, the aforementioned article established results linking the existence of a rotational horseshoe with finitary conditions based on simple properties of the given maps. This development brings the state of the art in surface dynamics closer to practical applications. Furthermore, building upon these criteria, \cite{condannchaoscap} presents a list of computer-assisted proofs of chaos for well-known families of maps, across large regions in parameter space, demonstrating the effectiveness of the new method. As emphasized in these works, a key motivation is to find applications of the largely developed topological theory of surface dynamics to concrete systems.

\medskip

Nevertheless, a key limitation arises in many relevant cases in these previous results: the need to locate two fixed points of the map with different rotation numbers. While this requirement poses no difficulty for maps defined by explicit formulas, such as those considered in \cite{condannchaoscap}, it becomes a major obstacle when dealing with Poincaré return maps. To illustrate this, one may consider one of the simplest physical systems suspected to exhibit chaotic behavior: the \emph{driven pendulum}, governed by the equation
$$
\ddot{q} = -\frac{g}{l} \sin(q) + A \sin(\omega t).
$$
This system can be studied via its Poincaré return map, corresponding to the time-$\frac{2\pi}{\omega}$ map of the extended flow. However, finding a pair of fixed points for such a map remains a significant challenge.

\medskip

In this article, we present a generalization of Theorem~A from \cite{passeggi2023weak}, providing conditions for annular chaos that do not require the existence of periodic points. These new conditions can be implemented in a completely straightforward manner, once the map is known, using basic computational tools.

\smallskip

As a first meaningful application, we revisit the driven pendulum and show, by a methodology involving a computer-assisted proof developed in \cite{capllavpass}, that our theorem detects topological chaos for physically relevant parameters. We refer to Section~4 of \cite{capllavpass} for a more detailed exposition.

\subsection{Main theorem}

Let us consider, as usual, the topological model of the annulus given by $\mathbb{A}=S^1\times \mathbb{R}$. Denote by $\mathrm{Homeo}_0(\mathbb{A})$ the set of homeomorphisms in the homotopy class of the identity. We now define the \emph{rotational difference} of a pair of closed topological disks $U_0, U_1$ for $f\in\mathrm{Homeo}_0(\mathbb{A})$, generalizing the corresponding notion for fixed points of annular maps. To do so, we require that

\begin{itemize}[leftmargin=2cm]
	\item $\mathrm{Int}(U_0) \cap \mathrm{Int}(f(U_0)) \neq \emptyset$, $\mathrm{Int}(U_1) \cap \mathrm{Int}(f(U_1)) \neq \emptyset$,
	\item $U_0\cup f(U_0)$ and $U_1\cup f(U_1)$ are disjoint and both inessential sets.
\end{itemize}

We call such a configuration a \emph{disjoint pair of disks}. Define the \emph{rotational difference} for this pair as
\[ \rho(U_0,U_1)=k_1-k_0, \]
where $k_0, k_1$ are the integers, for a given lift $F$ of $f$, such that
\[ F(\tilde{U}_0)\cap \left(\tilde{U}_0+k_0\right)\neq \emptyset \quad \text{and} \quad F(\tilde{U}_1)\cap \left(\tilde{U}_1+k_1\right)\neq \emptyset \]
for any lifts $\tilde{U}_0, \tilde{U}_1$ of $U_0, U_1$, respectively. Although each of these integers depends on $F$, their difference does not; we therefore denote it simply by $\rho(U_0,U_1)$.

The last ingredient needed to state our main result is the following. Given two closed disks $U_0, U_1\subset\mathbb{A}$, we say that $U_0$ \emph{visits} $U_1$ whenever some forward orbit of a point in the interior of $U_0$ meets the interior of $U_1$. Note that for non-wandering maps, if $U_0$ visits $U_1$, then $U_1$ also visits $U_0$.

\smallskip

We now state a simplified version of our main result in the non-wandering setting.  We denote by $\mathrm{Homeo}_{\,0,nw}(\mathbb{A})$ the class of non-wandering homeomorphisms in $\mathrm{Homeo}_0(\mathbb{A})$.

\begin{thmasterisco}
	
	Let $f \in \mathrm{Homeo}_{\,0,nw}(\mathbb{A})$ and assume that for some disjoint pair of disks $U_0, U_1$ one has $\rho:= \rho(U_0,U_1)\geq 3$ and that $U_0$ visits $U_1$. Then $f$ exhibits rotational chaos.
	
	Moreover, there exists an instability region $\mathcal{I}$ supporting the rotational horseshoe and intersecting both $U_0$ and $U_1$. This region satisfies that its rotation set, for some lift $F$, contains the interval $(1,\rho -1)$. In particular, every rational point in this interval is realized by some periodic orbit contained in the instability region. 
	
\end{thmasterisco}

Although the objects appearing in the previous statement are well known in the field, we highlight the following terminology. We say that a homeomorphism of the annulus $f$ exhibits \emph{rotational chaos} if some iterate of $f$ admits a topological rotational horseshoe, which in particular implies positive topological entropy and the existence of periodic orbits of arbitrarily large period. We call a \emph{Birkhoff instability region}, or simply an \emph{instability region}, any $f$-invariant open subannulus containing two periodic points with different rotation numbers whose ends are Birkhoff-related (see Section~\ref{s.pre}).

\begin{rmk}

The result is purely topological.

\end{rmk}

\begin{rmk}
An inspection of the literature rigorously addressing the concept of chaos reveals that, for the first time, a finitary criterion is provided whose application does not require the prior localization of periodic orbits. This marks a significant departure from classical approaches, where chaotic dynamics are typically inferred through the detection of homoclinic orbits or braid configurations involving periodic orbits. On the other hand, once the introduced conditions are fulfilled, one obtains periodic orbits exhibiting nontrivial rotational behavior, together with exponential growth of periodic orbits with respect to the period.
\end{rmk}

\begin{rmk}\label{r.noperturb}
At first glance, one might attempt to extend Theorem~A from \cite{passeggi2023weak} to obtain our main result by constructing a suitable perturbation $g$ of the given map $f$, supported inside the sets $U_0 \cup f(U_0)$ and $U_1 \cup f(U_1)$. Unfortunately, this strategy does not work, as such a perturbation could itself create the desired rotational horseshoe. We will return to this point later in the paper.
\end{rmk}

\begin{rmk}\label{r.compuntingU0U1}
As we will show in the example of a periodically forced pendulum, the existence of the pair of sets $U_0,U_1$ can be reliably verified using numerical methods. It is worth noting that the periodically forced pendulum was also the example used by Franks in his celebrated article \cite{Franks(PoincBirkh)} to illustrate the applicability of his techniques for establishing the existence of infinitely many periodic orbits. Here, we build upon these techniques, together with recent developments in the topological theory of surface dynamics, to establish the existence of chaos.

\end{rmk}

\begin{rmk}\label{r.chaoticsea}
Under the hypotheses of the main theorem, there exists an instability region intersecting both open sets $U_0$ and $U_1$. This region can be characterized as an open invariant subannulus with nontrivial rotation behavior and containing no invariant essential curves or \emph{circloids}. In particular, the existence of a rotational horseshoe follows from the rotational properties of this instability region. The localization of an instability region intersecting both $U_0$ and $U_1$ was not available in the previous result of \cite{passeggi2023weak}.
\end{rmk}

\begin{rmk}\label{r.poinbirk}
The statement provides a version of the classical Poincaré--Birkhoff theorem in which it is not necessary to first identify a pair of periodic orbits with different rotation numbers. The classical twist condition is replaced by the existence of a disjoint pair of disks with nontrivial rotational difference. Following the ideas of \cite{Franks(PoincBirkh)}, the hypotheses on the disks $U_0$ and $U_1$ allow us to choose suitable lifts of the map so as to obtain both positively and negatively returning disks. This guarantees the existence of fixed points in the universal cover for different lifts of the map.
\end{rmk}

\begin{rmk}\label{r.endsrelatedver}
As in the driven pendulum case, one sometimes encounters dynamical systems preserving an area form that, although invariant, is not finite; hence the dynamics need not be non-wandering. Fortunately, this issue has been addressed in the theory of surface dynamics, and alternative hypotheses have been developed for settings that fall outside the non-wandering framework (see \cite{fabiopatricehorseshoes}). One such setting concerns an annular homeomorphism $f$ that fixes the ends, where the ends are said to be \emph{Birkhoff-related}. This means that, given arbitrary neighborhoods of the two ends, there exist positive orbits from one neighborhood to the other. Below, we state the previous result also for this type of map.
\end{rmk}

\begin{rmk}
Once some regularity of the map is assumed, such as smoothness, as is the case for usual Poincar\'e return maps, Pesin theory allows us to derive rich hyperbolic structures from topological rotational horseshoes. 
\end{rmk}

\subsection{Detailed version of the main theorem}

We now present a more detailed statement of our main result, which also considers the cases where the rotational difference between the open sets $U_0$ and $U_1$ equals $1$ or $2$. To this end, we extend the notion of a disjoint pair of disks as follows.

Given $f \in \mathrm{Homeo}_0(\mathbb{A})$, we say that a pair of closed topological disks $U_0$ and $U_1$ is an $n$-\emph{disjoint pair of disks} for $n \in \mathbb{N}$ (an $n$-\emph{dpd}) whenever

\begin{itemize}[leftmargin=2cm]
	\item $\mathrm{Int}(U_0) \cap \mathrm{Int}(f(U_0)) \neq \emptyset$, $ \mathrm{Int}(U_1)\cap \mathrm{Int}(f(U_1)) \neq \emptyset$;
	\item ${f^i(U_0)} \cap {f^j(U_1)} = \emptyset$ for $i,j \in [1,n] \cap \mathbb{N}$;
	\item $\bigcup_{i=0}^{n} f^i(U_0)$ and $\bigcup_{i=0}^{n} f^i(U_1)$ are inessential.
\end{itemize}

The version of the main theorem in the non-wandering case follows.

\begin{thmsinnum}{\textbf{A}}
	\label{maintheorem}
	Let $f \in \mathrm{Homeo}_{0,nw}(\mathbb{A})$ and assume that, for some $n \ge 3$, there exists an $n$-dpd $U_0, U_1$ with nonvanishing rotational difference $\rho \in \mathbb{N}^*$ such that $U_0$ visits $U_1$.
	
	Then $f$ exhibits rotational chaos and the rotational horseshoe for a power of $f$ can be considered included in a Birkhoff instability region $\mathcal{I}$ meeting both $U_0$ and $U_1$. Moreover, for some lift $F$ of $f$,
	\[
	[1/n, \rho - 1/n] \subseteq  \rho_{\overline{\mathcal{I}}}(F)
	\]
	and every rational number in $(1/n,\rho-1/n)$ is realized by a periodic point in $\mathcal{I}$.
	
	If $\rho \ge 2$, the result remains valid for $n \ge 2$; and if $\rho \ge 3$, it remains valid for $n \ge 1$.
\end{thmsinnum}

\begin{figure}[h!]
	
	\centering
	\includegraphics[width=390 pt]{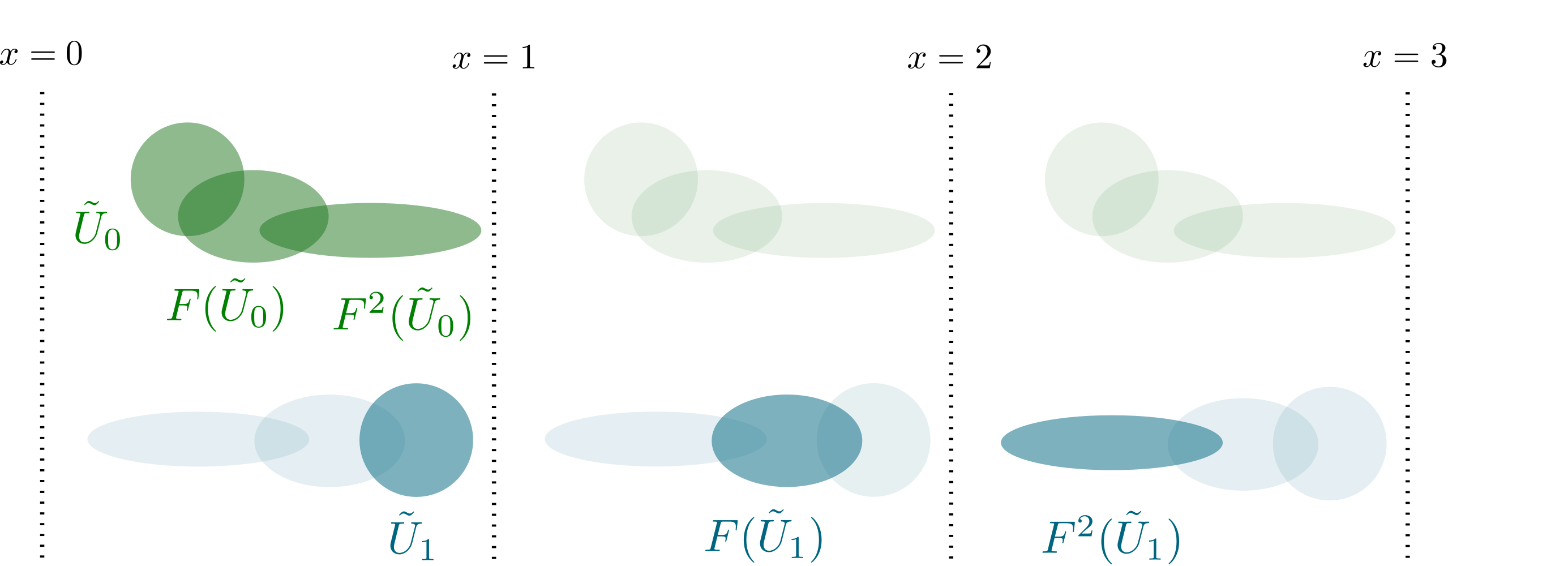}
	
	\caption{\centering Illustration in the universal cover of a $3$-dpd with rotational difference equal to $1$}
\end{figure}

As explained in Remark~\ref{r.endsrelatedver}, certain applications require working with systems preserving an infinite measure $\lambda$. In this setting, an additional assumption on the relation between the ends of the annulus is needed. Given $f\in\mathrm{Homeo}_0(\mathbb{A})$, we say that the ends of $\mathbb{A}$ are \emph{strongly non-related} if there exist disjoint neighborhoods $V^+$ and $V^-$ of $+\infty$ and $-\infty$, respectively, such that the forward orbit of each neighborhood is disjoint from the other. On the other hand, the ends are Birkhoff-related if, as explained in the aforementioned remark, every pair of disjoint neighborhoods of the two ends has the property that the forward orbit of each neighborhood meets the other.

These hypotheses arise naturally in the presence of certain symmetries, as detailed in the driven pendulum example presented immediately after the formal statement of our second main result. Given a measure $\lambda$ that is absolutely continuous with respect to the Lebesgue measure, we denote by $\mathrm{Homeo}_{\,0,\lambda}(\mathbb{A})$ the class of identity-isotopic homeomorphisms preserving $\lambda$.

\begin{thmsinnum}{\textbf{B}}
\label{maintheorembr}	
	Let $f \in \mathrm{Homeo}_{\,0,\lambda}(\mathbb{A})$, and let $U_0, U_1$ form an $n$--dpd for some $n \geq 3$, where $U_0$ visits $U_1$, $U_1$ visits $U_0$, and $\rho_f(U_0, U_1) = \rho \in \mathbb{N}^*$. Assume that either
	\begin{enumerate}
		\item the ends of $\mathbb{A}$ are Birkhoff-related, or
		\item the ends of $\mathbb{A}$ are strongly non-related.
	\end{enumerate}
	
	Then $f$ exhibits rotational chaos and the rotational horseshoe for a power of $f$ can be considered included in an instability region $\mathcal I$ meeting both $U_0$ and $U_1$. Moreover, for some lift $F$ of $f$, the rotation set of $\mathcal{I}$ satisfies
	\[
	\bigl[1/n,\, \rho - 1/n\bigr]\subseteq  \rho_{\overline{\mathcal{I}}}(F),
	\]
	and every rational number in $(1/n,\rho-1/n)$ is realized by a periodic point in $\mathcal{I}$.
	
	Furthermore, if $\rho \geq 2$, the result remains valid for $n \geq 2$, and if $\rho \geq 3$, it holds for every $n \in \mathbb{N}$.
\end{thmsinnum}

\subsection{Chaos in the driven Pendulum}

As a main application of this result, we show how the tool developed here---grounded in the vast body of topological theory of surface dynamics---allows the implementation of very simple computer-assisted proofs of the existence of rotational horseshoes for the so-called \emph{driven pendulum}. This application can be viewed as a continuation of the study of this type of differential equation presented in the renowned article by John Franks~\cite{Franks(PoincBirkh)}, see Section~5.

Recall that the driven pendulum is given by
$$\ddot{q} = -\frac{g}{l} \sin(q) + A \sin(\omega t).$$
Here we introduce a sample of ongoing work (\cite{capllavpass}), based on the CAPD library and built upon interval arithmetic (see~\cite{condannchaoscap}).

\smallskip

In light of Theorem~\ref{maintheorembr} and in order to ensure topological chaos, we need to detect a $1$-dpd, namely $U_0, U_1$ with $\rho(U_0,U_1) \ge 3$, for the Poincaré map $f \in \mathrm{Homeo}_0(\mathbb{A})$ associated with the differential equation, given by the time-$T$ map of the family of solutions starting at time~$0$, where $T = \frac{2\pi}{\omega}$. Figure~\ref{imagen1sofia} shows, in blue and pink, such a pair $U_0, U_1$ in the annulus and in the universal cover~$\mathbb{R}^2$. Their respective images are drawn in the same colors, where $F$ denotes a lift of $f$, from which one can observe that $\rho(U_0,U_1) = 4$. After a few iterates, one observes that the forward orbit of each disk intersects the other.

\begin{figure}[h!]
	
	\centering
	\includegraphics[width=390 pt]{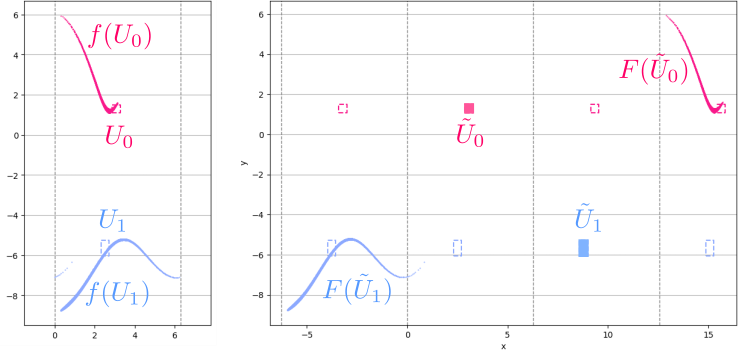}
	
	\caption{On the left, a representation in the annulus; on the right, its lift to the universal cover. The images of the simulation were produced by Sofía Llavayol, and the parameters used are \( g = 9.81 \), \( L = 1.0 \), \( A = 3.0 \), and \( T = 2.5 \).
	}
	\label{imagen1sofia}
\end{figure}

Then, if every orbit has uniformly bounded vertical displacement, we can apply case~(2) of Theorem~\ref{maintheorembr} and deduce the existence of a rotational horseshoe. On the other hand, if there exist orbits with arbitrarily large vertical displacement, we apply case~(1) of Theorem~\ref{maintheorembr}. Indeed, the symmetry of the differential equation ensures that every positive vertical displacement is accompanied by a corresponding negative one, yielding the hypothesis required in case~(1).

\smallskip

The rigorous existence of these boxes and of the orbit dynamically connecting them is established in the aforementioned work in progress. Altogether, we obtain the following.

\begin{thmasterisco}[Theorem 1.1 of \cite{capllavpass} ]	
Consider the differential equation above with
$$g = 9.8,\quad l = 1,\quad A = 3,\quad T = 2.5$$
Then, the associated Poincaré map
admits a rotational horseshoe.

Moreover, it is contained
inside an instability region $\mathcal{I}$ which meets the boxes\footnote{In fact, the side lengths of the boxes are specified with a precision of 12 decimal digits.}
\begin{gather*}
	U_0=[2.8710, 3.2711]\times[1.0928, 1.4929], \\
	U_1=[8.5937, 8.9938]\times[-6.0503, -5.2704]
\end{gather*}

and the length of its rotation interval $\rho_{\mathcal{I}}(f)$ is at least 2.

In particular, for a lift $F$ of $f$ such
that $\rho_{\mathcal{I}}(F)\supset [0,2]$, every rational point $\frac{p}{q}$ in the rotation set is realized by a periodic 
orbit of $f$ belonging to $\mathcal{I}$.
	
\end{thmasterisco}

If one disregards historical connections, physics, topological dynamics, and numerical methods might \emph{a priori} appear to be unrelated fields. Nevertheless, the results presented here reveal a nontrivial and meaningful intersection among these three areas.

\smallskip
\section{Outline of the proofs of Theorems~A and~B} 

The main conclusions of both theorems are the existence of \emph{rotational horseshoes} and \emph{instability regions}. The statements concerning the realization of rotation vectors and the estimates for the rotation set of the instability region will not be discussed in this summary. Recall that a \emph{circloid} is an essential continuum whose complement consists of two subannuli, each containing one end of the annulus, and which is minimal among continua with this property. When the circloid is invariant under some $f\in\mathrm{Homeo}_{\,0}(\mathbb{A})$, the prime-ends rotation number associated with one of the boundary components of the complementary subannuli and a lift $F$ of $f$ is defined as the rotation number of the corresponding boundary component after a suitable compactification, called the \emph{prime-ends compactification}, in which the boundary becomes a topological circle.

\medskip

We highlight three key ingredients in the proofs.

\medskip

The first ingredient is that, under the hypotheses of both theorems, every $f$-invariant circloid intersecting $U_0$ or $U_1$ satisfies strong rotational constraints. More precisely, for a suitable lift $F$ of $f$, the assumptions on the disks imply that any $f$-invariant circloid intersecting $U_0$ has a prime-ends rotation number close to $0$, while any circloid intersecting $U_1$ has a prime-ends rotation number close to $\rho$, where $\rho$ denotes the rotational difference between $U_0$ and $U_1$. The techniques developed to prove this also show that every $f$-invariant circloid necessarily has a singleton rotation set. The non-wandering or area-preserving hypothesis is essential at this stage. See Section~\ref{seccionRotationsets} for further details and a counterexample showing that the conclusion fails without these assumptions.

\medskip

The second ingredient, proved in the final section, states that the hypotheses of both theorems imply that every $f$-invariant open subannulus meeting both $U_0$ and $U_1$ contains two periodic orbits with different rotation numbers. The proof relies on performing small perturbations that create fixed points inside $U_0$ and $U_1$, even though these perturbations may destroy the non-wandering or area-preserving property. Inspired by ideas of Franks from~\cite{Franks(PoincBirkh)}, we use the perturbed map to construct suitable periodic chains of disks for appropriate lifts, yielding periodic points with different rotation numbers. Since these periodic points cannot lie in the support of the perturbation, they must also exist for the original map $f$. Consequently, the rotation set of the corresponding invariant subannulus is non-trivial.

\medskip

The third ingredient is a result of Le Calvez and Tal \cite{fabiopatricehorseshoes}, which asserts that every Birkhoff instability region in either the non-wandering or the area-preserving setting admits a rotational horseshoe for some iterate of $f$. Thus, in both theorems, the existence of rotational chaos is reduced to proving the existence of an instability region intersecting both $U_0$ and $U_1$.

\medskip
For the non-wandering case, namely Theorem~\ref{maintheorem}, the instability region must in general be constructed as follows. Since $U_0$ and $U_1$ are Birkhoff related, no invariant essential continuum can separate them. Moreover, by the rotational constraints discussed above, no $f$-invariant circloid can intersect both disks simultaneously.

Assume that $F$ is a lift of $f$ such that, for every lift $\widetilde{U}_0$ of $U_0$,
\[
F(\widetilde{U}_0)\cap\widetilde{U}_0\neq\emptyset,
\]
and that the union of $U_0$ with its first $n$ iterates remains inessential.

Consider the family of all $f$-invariant circloids. If this family is empty, then, by the realization result mentioned above, the whole annulus is an instability region and the conclusion follows. Otherwise, assume that no invariant circloid meets either $U_0$ or $U_1$. One can show that, since one disk visits the other, both disks are contained in an $f$-invariant open subannulus containing no invariant circloids. By the rotational constraints established above for invariant subannuli meeting $U_0$ and $U_1$, this subannulus is necessarily an instability region.

Finally, suppose that some invariant circloid intersects $U_0$ or $U_1$, say $U_0$. Notice that the rotation set of this circloid must be close to $0$ because of the rotational constraints discussed above. One first shows that every invariant circloid meeting $U_0$ leaves $U_1$ on the same side; assume it is the lower side. By introducing a natural total order on the family of invariant circloids meeting $U_0$, a limiting argument produces a minimal such circloid $\mathcal{C}_0$. Repeating the argument with invariant circloids lying below $\mathcal{C}_0$ but disjoint from $U_0$, another limiting argument yields an open annulus bounded above by $\mathcal{C}_0$ and containing no invariant circloid in its interior. This annulus meets both $U_0$ and $U_1$, and is therefore an instability region.

\bigskip

For the area-preserving case, namely Theorem~\ref{maintheorembr}, we must study separately the two situations appearing in the statement. If the ends of the annulus are Birkhoff related, then the whole annulus is itself an instability region, and the conclusion follows immediately from the result of Le Calvez and Tal mentioned above. If there exists a bounded $f$-invariant subannulus intersecting both $U_0$ and $U_1$, then the restriction of $f$ to this subannulus is non-wandering, so the ideas presented for Theorem~\ref{maintheorem} also hold.

It therefore remains to consider the case where the ends are strongly non-related and no bounded invariant subannulus intersects both $U_0$ and $U_1$. In this case, the assumption on the ends implies that $U_0$ and $U_1$ are contained in an invariant subannulus bounded from below or above. A similar argument of the non-wandering case, considering all $f$-invariant circloids, shows that one can extract such a subannulus whose interior contains no invariant circloid. We then prove that this subannulus is a Birkhoff instability region, where the area-preserving hypothesis plays a crucial role.

\bigskip 

We would like to thank Fabio Armando Tal and Pierre--Antoine Guihéneuf for insightful technical discussions related to several proofs, as well as for valuable comments that improved this article.

\section{Preliminaries}\label{s.pre}

\subsection{Topological preliminaries}

Throughout this article, we model the annulus by
\[
\mathbb{A}:=S^1\times\mathbb{R},
\]
where $S^1=\mathbb{R}/\sim$, with $x\sim y$ if and only if $x-y\in\mathbb{Z}$.

We write $\operatorname{pr}_i:\mathbb{R}^2\to\mathbb{R}$ for the projection onto the $i$th coordinate, $i=1,2$. The universal cover of $\mathbb{A}$ is $\mathbb{R}^2$, with covering projection $\pi:\mathbb{R}^2\to\mathbb{A}$ defined by $\pi(x,y)=([x],y)$. We denote the two ends of $\mathbb{A}$ by $+\infty$ and $-\infty$, according to the second coordinate. We now fix some terminology for subsets of $\mathbb{A}$.

\begin{itemize}
	\item An \emph{essential set} is any connected set that cannot be contained in any open topological disk of $\mathbb{A}$.
	\item A \emph{continuum} is a compact connected subset of $\mathbb{A}$.
	\item An \emph{essential continuum} is a continuum that separates the two ends of $\mathbb{A}$.
	\item An \emph{annular continuum} is an essential continuum whose complement consists of two connected components, each being a neighborhood of one end of the annulus.
	\item A \emph{circloid} is an annular continuum that contains no proper essential annular continuum.
	\item A \emph{cofrontier} is a circloid with empty interior.
	\item A \emph{regular annulus} is an open essential subannulus whose boundary components  in $\mathbb{A}\cup \{+\infty\} \cup \{-\infty\}$ are either circloids in $\mathbb{A}$ or an end of $\mathbb{A}$.
\end{itemize}

Note that the entire annulus $\mathbb{A}$ or $S^1\times \mathbb{R}^{>0}$ are regular annuli. On the other hand, $\mathbb{A}\setminus \left( \{ [0] \} \times\left[0,+\infty \right)\right)$ or $S^1\times (0,+\infty) \setminus \left( \{[0]\}\times [0,1] \right)$ are not.

Circloids are objects that have proved to be fundamental in several developments in the field. In particular, they have been intensively studied over the past few decades. See, for instance, \cite{JagerToralHomeo, passpotrsamb, passeggi2023weak, JagerKoro} and the references therein. In general, the topology of circloids and co-frontiers can be very complicated. In fact, in several relevant dynamical situations, the latter class often appears in the form of indecomposable continua. As an example, the reader can have in mind the so-called \emph{pseudo-circle}.

Given an annular continuum $\mathcal{A}$, we denote by $\mathcal{U}^+(\mathcal{A})$ and $\mathcal{U}^-(\mathcal{A})$ the connected components of the complement of $\mathcal{A}$ accumulating at the ends $+\infty$ and $-\infty$, respectively.

We recall some elementary properties of complementary components that will be used throughout the article. In particular, if $\mathcal{A}$ and $\mathcal{B}$ are annular continua with $\mathcal{A}\subseteq \mathcal{B}$, then
\[
\mathcal{U}^-(\mathcal{B})\subseteq\mathcal{U}^-(\mathcal{A}),
\qquad
\mathcal{U}^+(\mathcal{B})\subseteq\mathcal{U}^+(\mathcal{A}).
\]
Corollary~3.3 of \cite{JagerToralHomeo} states that, for every circloid $\mathcal{C}$,
\[
\mathcal{U}^+\bigl(\mathcal{U}^-(\mathcal{C})\bigr)=\mathcal{U}^+(\mathcal{C}),
\qquad
\mathcal{U}^-\bigl(\mathcal{U}^+(\mathcal{C})\bigr)=\mathcal{U}^-(\mathcal{C}).
\]

There exists a natural partial order $\preccurlyeq$ on circloids. Given two circloids $\mathcal{C}_1$ and $\mathcal{C}_2$, we write
\[
\mathcal{C}_1 \preccurlyeq \mathcal{C}_2
\quad \text{iff} \quad
\mathcal{C}_1 \subseteq \overline{\mathcal{U}^-(\mathcal{C}_2)}.
\]

The following standard fact follows easily from Lemma~2.4 of \cite{JagerToralHomeo}, after adapting the notation to our setting.

\begin{lemma}
	\label{lemaunicocircloidborde}
	Let $\mathcal{A}$ be an annular continuum. Then there exists a unique maximum circloid $\mathcal{C}^+_{\mathcal{A}}\subseteq \mathcal{A}$ and a unique minimum circloid $\mathcal{C}^-_{\mathcal{A}}\subseteq \mathcal{A}$, called the upper and lower circloids of $\mathcal{A}$, respectively. Specifically, every circloid $\mathcal{C}\subseteq \mathcal{A}$ satisfies
		\[
		\mathcal{C}^-_{\mathcal{A}} \preccurlyeq \mathcal{C} \preccurlyeq \mathcal{C}^+_{\mathcal{A}}.
		\]
	
\end{lemma}

Whenever a circloid $\mathcal{C}$ has empty interior, it is a cofrontier and, moreover,
\[
\mathcal{C}=\partial\mathcal{U}^+(\mathcal{C})=\partial\mathcal{U}^-(\mathcal{C}).
\]
This also follows from Lemma~3.4 of \cite{JagerToralHomeo}.

	Given two different circloids $\mathcal{C}^-$ and $\mathcal{C}^+$ with $\mathcal{C}^-\preccurlyeq\mathcal{C}^+$, we define the \emph{region between $\mathcal{C}^-$ and $\mathcal{C}^+$} as the open set
\[
\mathcal{R}(\mathcal{C}^-,\mathcal{C}^+)
:=
\mathcal{U}^+(\mathcal{C}^-)
\cap
\mathcal{U}^-(\mathcal{C}^+),
\]
and the \emph{annular continuum delimited by $\mathcal{C}^-$ and $\mathcal{C}^+$} as 
\[
\mathcal{A}(\mathcal{C}^-,\mathcal{C}^+)
:=
\overline{\mathcal{U}^+(\mathcal{C}^-)}
\cap
\overline{\mathcal{U}^-(\mathcal{C}^+)}.
\]
In particular,
\[
\mathcal{R}(\mathcal{C}^-,\mathcal{C}^+)
\subseteq
\operatorname{Int}\bigl(\mathcal{A}(\mathcal{C}^-,\mathcal{C}^+)\bigr)
\subseteq
\mathcal{A}(\mathcal{C}^-,\mathcal{C}^+).
\]

The region $\mathcal{R}(\mathcal{C}^-,\mathcal{C}^+)$ is nonempty. Indeed, the proof of Proposition~3.7 in \cite{JagerToralHomeo} shows, using Corollary~3.3, that if $\mathcal{U}^+(\mathcal{C}^-)\cap\mathcal{U}^-(\mathcal{C}^+)
=
\mathcal{U}^-(\mathcal{C}^-)\cap\mathcal{U}^+(\mathcal{C}^+)
=
\emptyset$, 
then $\mathcal{C}^-=\mathcal{C}^+$. Since $\mathcal{C}^-\preccurlyeq\mathcal{C}^+$ implies that $\mathcal{U}^-(\mathcal{C}^-)\cap\mathcal{U}^+(\mathcal{C}^+)=\emptyset$, it follows that $\mathcal{U}^+(\mathcal{C}^-)\cap\mathcal{U}^-(\mathcal{C}^+)\neq\emptyset$ whenever the circloids are different.

\medskip

When $\mathcal{C}^-$ and $\mathcal{C}^+$ meet, $\mathcal{R}(\mathcal{C}^-,\mathcal{C}^+)$ is a union of open disks, while when they are disjoint, $\mathcal{R}(\mathcal{C}^-,\mathcal{C}^+)$ is a regular annulus.

\medskip

The following lemma describes the monotonicity of these regions. Its proof, although simple, is deferred to the final section.

\begin{lemma}
	\label{lemmapreliminaries-regionentrecircloids}

		Consider two different circloids $\mathcal{C}^-$ and $\mathcal{C}^+$ with $\mathcal{C}^-\preccurlyeq\mathcal{C}^+$. Let $\mathcal{E}$ be a circloid such that $\mathcal{C}^+\preccurlyeq\mathcal{E}$ and $\mathcal{C}^+\neq\mathcal{E}$. Then
		\[
		\mathcal{R}(\mathcal{C}^-,\mathcal{C}^+)
		\subsetneq
		\mathcal{R}(\mathcal{C}^-,\mathcal{E}).
		\]
\end{lemma}

\medskip

To finish this subsection on topological preliminaries, let us introduce the final objects used throughout the article.

For a closed connected set $V\subset\mathbb{A}$, we define $\operatorname{Fill}(V)$ as the union of $V$ with all the bounded connected components of its complement.

\smallskip

We call a \emph{vertical line} the image of a proper embedding
$v:\mathbb{R}\to\mathbb{A}$ such that
\[
\lim_{t\to\pm\infty}v(t)=\pm\infty.
\]

We call an \emph{ends-connector} any connected set accumulating at both ends of the annulus $\mathbb{A}$ and whose complement consists of an open topological disk which admits a vertical line.

\subsection{Rotation theory}

We denote by $\mathrm{Homeo}_0(\mathbb{A})$ the set of homeomorphisms isotopic to the identity. If $f\in \mathrm{Homeo}_0(\mathbb{A})$ and $F:\mathbb{R}^2\to\mathbb{R}^2$ is a lift of $f$, then $F$ preserves orientation and satisfies
\[
F(x+(1,0))=F(x)+(1,0)
\]
for every $x\in\mathbb{R}^2$. Every other lift of $f$ differs from $F$ by an integer translation. For each lift $F$, its \emph{rotation set} $\rho(F)$ is defined by
\[
\rho(F)=\left\{\lim_i\frac{\mathrm{pr}_1\left(F^{n_i}(x_i)-x_i\right)}{n_i}: \ n_i\nearrow+\infty,\ x_i\in\mathbb{R}^2\right\}.
\]
The rotation set is always a closed interval, possibly trivial or unbounded. The rotation set associated with any other lift of $f$ is obtained from $\rho(F)$ by an integer translation.

If $\mathcal{X}\subseteq\mathbb{A}$ is an $f$-invariant set, one can define the rotation set relative to $\mathcal{X}$ by restricting the definition above to $\mathcal{X}$, and denote it by $\rho_{\mathcal{X}}(F)$. Whenever $\mathcal{X}$ is compact, so is $\rho_{\mathcal{X}}(F)$. When $\mathcal{X}$ is a singleton, say $\mathcal{X}=\{x\}$, necessarily $x$ is a fixed point. In this case, we call $\rho_x(F)\in\mathbb{Z}$ the \emph{rotation number} of $x$. Conversely, we say that $x$ realizes the rotation number $\rho$ if $\rho=\rho_x(F)$. The same terminology applies to periodic points.

Although the rotation number of a periodic point depends on the choice of the lift $F$, the difference between the rotation numbers of any two periodic points is independent of the chosen lift.

\subsection{Prime-ends compactification}

The prime-ends compactification provides an abstract compactification of an open topological disk in which the added boundary is a circle. The construction extends to arbitrary open connected sets (see \cite{Koropecki_2015}). In this article, we only need the case where the disk is obtained by adding an end of the annulus to a neighborhood of that end. We now describe this construction in our setting

An \emph{arc} in $\mathbb{A}$ is a continuous and injective map $\gamma:[0,1]\to \mathbb{A}$.
If $K$ is an essential annular continuum and $A =\mathcal{U}^+(K) \subseteq \mathbb{A}$, a \emph{crosscut} in $A$ is an arc $\gamma$ in $\mathbb{A}$ such that $\gamma((0,1)) \subseteq U$ and $\gamma(0), \gamma(1) \in \partial A$. Every crosscut divides $A$ into two connected components, called \emph{cross-sections}, one bounded and the other unbounded.

A \emph{fundamental chain} in $A$ is a decreasing sequence $\{R_n\}$ of cross-sections such that, if $\partial R_n \cap A= A_n$ denotes the corresponding sequence of crosscuts, then $A_n \cap A_m = \emptyset$ for $n \neq m$, and $\operatorname{diam}(A_n) \to 0$ as $n \to \infty$. We say that two fundamental chains $\{R_n\}$ and $\{R_n'\}$ are \emph{equivalent} if every $R_n$ is contained in some $R_j'$ and vice versa. Each equivalence class $\mathcal{E}$ of fundamental chains is called a \emph{prime end}.

Adding $+\infty$ to $A$, we obtain an open topological disk
\[
D := A \sqcup \{+\infty\}.
\]
The set
\[
\widehat{D} = D \sqcup \{ \mathcal{E} \mid \mathcal{E} \text{ is a prime end} \},
\]
equipped with a suitable topology, is a compactification of $D$ (and hence of $A$), called the \emph{prime-ends compactification}. This compactification is homeomorphic to a closed disk.

We endow $D$ with the complex structure inherited from $S^2$, regarding $D$ as a subset of
\[
\mathbb{A}\sqcup \{+\infty,-\infty\} \simeq S^2.
\]
By the Riemann mapping theorem, if $\mathbb{D}$ denotes the open unit disk, there exists a conformal homeomorphism
\[
\varphi:\mathbb{D}\to D
\]
such that $\varphi(0)=+\infty$, and $\varphi$ extends to a homeomorphism
\[
\overline{\varphi}:\overline{\mathbb{D}}\to\widehat{D}.
\]

Under this identification, each prime end corresponds to a point of $\partial\mathbb{D}$. If $f$ is any map in $\mathrm{Homeo}_0(\mathbb{A})$ and $K$ is $f$-invariant, we can extend $f|_A$ to $D=\mathcal{U}^+(K) \sqcup \{+\infty\}$. As $f$ sends crosscuts to crosscuts and cross sections to cross sections, $f|_D$ extends to a homeomorphism $\widehat{f}: \widehat{D} \to \widehat{D}$. 
More details can be found in \cite{milnor1990dynamicscomplexvariableintroductory}, \cite{Carathéodory1913a}, and \cite{Carathéodory1913b}.

A \emph{ray} is a continuous and injective map $r:(0,+\infty)\to A$ such that
\[\lim_{t\to+\infty}r(t)\]
exists and belongs to $\partial A$. A point $x\in\partial A$ is called \emph{accessible} if there exists a ray $r$ satisfying
\[\lim_{t\to+\infty}r(t)=x.\]
In particular, every crosscut joins two accessible points. Moreover, the set of accessible points is dense in $\partial A$.

Viewing $A$ as a subset of $\widehat{D}$, every ray in $A$ can also be regarded as a ray in $\widehat{D}$ (or, equivalently, via the homeomorphism $\varphi$, as the ray $\ell=\varphi^{-1}(r)$ in $\mathbb{D}$). This point of view gives rise to the following observations:
\begin{itemize}
	\item If $r$ is a ray in $A$ landing at $x\in\partial A$, then the same ray, viewed in $\widehat{D}$, converges to a prime end $\mathcal{E}_x$. In this case, we say that $\mathcal{E}_x$ is an \emph{accessible prime end}.
	\item If $r$ and $r'$ are rays in $A$ landing at different points $x,x'\in\partial A$, then $\mathcal{E}_x\neq\mathcal{E}_{x'}$.
\end{itemize}

The set of accessible prime ends is dense in $\partial\widehat{D}$ (see Subsection~3.1 of \cite{Koropecki_2015}). 

\bigskip

As the prime-ends boundary is a circle, every homeomorphism isotopic to the identity extends to this compactification with a well-defined rotation number on the boundary circle, depending on the chosen lift of the map. To relate this prime-ends rotation number with the rotation set of $\partial A$, we first describe the lift involved in its definition.

Consider $f\in\mathrm{Homeo}_0(\mathbb{A})$, let $K$ be an $f$-invariant essential annular continuum, and set $A=\mathcal{U}^+(K)$, which is also $f$-invariant. Define $D=A\sqcup\{+\infty\}$. Let $\varphi:\mathbb{D}\to D$ be a homeomorphism sending $0$ to $+\infty$. Notice that $\mathbb{D}^*=\mathbb{D}\setminus\{(0,0)\}\simeq S^1\times(0,1)$ is identified with $A$ through $\varphi$.

Let $\mathrm{pr}:\mathbb{H}\to\mathbb{D}^*$ be the universal covering map of $\mathbb{D}^*$, for which the deck transformations correspond to integer horizontal translations, and let $\pi_A:\pi^{-1}(A)\to A$ denote the universal covering map of $A$. Consider a lift
\[
\phi:\mathbb{H}\to\pi^{-1}(A)
\]
of $\varphi|_{\mathbb{D}^*}$. Since $\phi$ preserves orientation and induces an isomorphism between the groups of deck transformations, the generator corresponding to the unit horizontal translation is mapped to the generator corresponding to the unit horizontal translation. Hence,
\[
\phi(x+n,y)=\phi(x,y)+(n,0)
\]
for every $n\in\mathbb{Z}$.

If
\[
g:=\varphi^{-1}\circ f\circ\varphi:\mathbb{D}^*\to\mathbb{D}^*
\]
and $F$ is a lift of $f$, then
\[
G:=\phi^{-1}\circ F\circ\phi
\]
is the lift of $g$ corresponding to the lift $F$. The other lifts of $g$ are of the form
\[
G':=\phi^{-1}\circ F\circ\psi,
\]
where $\psi(x,y)=\phi(x,y)+(m,0)$ for some $m\in\mathbb{Z}$.

Since $g$ extends continuously to the prime-ends compactification of $\mathbb{D}^*$, the lift $G$ also extends continuously to $\overline{\mathbb{H}}$.

\[
\begin{tikzcd}
	\arrow[loop left, looseness=10, "G"] \mathbb{H} \arrow[r, "\phi"] \arrow[d, "\mathrm{pr}"']
	& \pi^{-1}(A) \arrow[d, "\pi_A"] \arrow[loop right, looseness=4, "F"] \\
	\arrow[loop left, looseness=8, "g"] \mathbb{D}^* \arrow[r, "\varphi"]
	& A \arrow[loop right, looseness=12, "f"]
\end{tikzcd}
\]

The restriction of $G$ to $\partial\mathbb{H}$ is a lift of the restriction of the prime-ends extension of $g$ to $\partial\mathbb{D}$. Therefore, the rotation number
\[\rho_{\partial\mathbb{D}}(G)=\lim_{n\to+\infty}\frac{\mathrm{pr}_1\bigl(G^n(x,1)-(x,1)\bigr)}{n}\]
is well defined and does not depend on the choice of $x\in\partial\mathbb{H}$.

We define the \emph{upper prime-ends rotation number} of $K$ associated with the lift $F$ by
\[\rho_K^+(F):=\rho_{\partial\mathbb{D}}(G).\]

Similarly, considering $\mathcal{U}^-(K)$, one defines the \emph{lower prime-ends rotation number} of $K$.

Using the density of accessible points together with the fact that the dynamics of rays is compatible with the prime-ends compactification, one obtains
\[\rho_K^\pm(F)\in\rho_K(F).\]
A simple proof of this fact can be found in \cite{HERNÁNDEZ-CORBATO_2017}.

\subsection{Rotational horseshoes}

In this article, we focus on annulus maps exhibiting \emph{rotational horseshoes} in the sense of \cite{fabiopatricehorseshoes}, \cite{PAMiliton}, and \cite{passpotrsamb}, which are closely tied to the topology of the annulus.

Let us recall a standard definition, inspired by Smale's classical visual presentation: it involves a topological rectangle whose image under the map exhibits two or more Markovian intersections with itself. Since we are working in the $C^0$ setting, we do not expect hyperbolicity to be present. Given the central role these objects play in our work, we now proceed to carefully introduce their precise formulation.

A \emph{rectangle} is a topological closed disk in $\mathbb{R}^2$ or $\mathbb{A}$ together with four arcs
	$R_u$, $R_r$, $R_d$, and $R_l$ (up, right, down, left) whose concatenation 
	in this order gives a simple loop. Note that the image of such loop must coincide
	with the boundary of $R$. In this work, we usually abuse notation by calling the rectangle by the same name as the topological disk. Further, given any homeomorphism $f$ on the
	space where the rectangle lies, we can naturally consider the new rectangle $f(R)$ as given
	by the image of the topological disk, and the composition of the arcs of $R$ with $f$.

Given two closed rectangles $R$ and $R'$ in $\mathbb{R}^2$, we say that $R\cap R'$ defines a \emph{Markovian intersection} if there exists a connected component $H\subset R\cap R'$ such that
\begin{enumerate}
	\item $\partial H$ contains subarcs $H_u\subset R'_u$ and $H_d\subset R'_d$, both joining $R_l$ to $R_r$.
	\item $R'_l$ and $R'_r$ do not intersect the interior of the subrectangle of $R$ determined by $H_u$ and $H_d$.
\end{enumerate}

Observe that the sides of a rectangle may be relabeled by exchanging the roles of the upper and lower sides with those of the left and right sides. Although this is topologically irrelevant, the definition of a Markovian intersection depends on the chosen labeling.

Let $f\in\mathrm{Homeo}_{\,0}(\mathbb{A})$, let $F$ be a lift of $f$, and let $k\ge2$ be an integer. We say that $f$ presents a \emph{rotational horseshoe} with $k$ symbols and period $n>1$ if, for some fixed $\ell\in\mathbb{Z}$, there exists a rectangle $\tilde{R}$ in $\mathbb{R}^2$ such that, for every $i\in\{1,\dots,k\}$,
\[
F^n(\tilde{R})\cap\bigl(\tilde{R}+(i+\ell,0)\bigr)
\]
contains a Markovian intersection.

If $\tilde{R}$ projects onto a rectangle $R\subset\mathbb{A}$, we say that the horseshoe is \emph{supported on} $R$.

\begin{figure}[h!]
	
	\centering
	\includegraphics[width=230 pt]{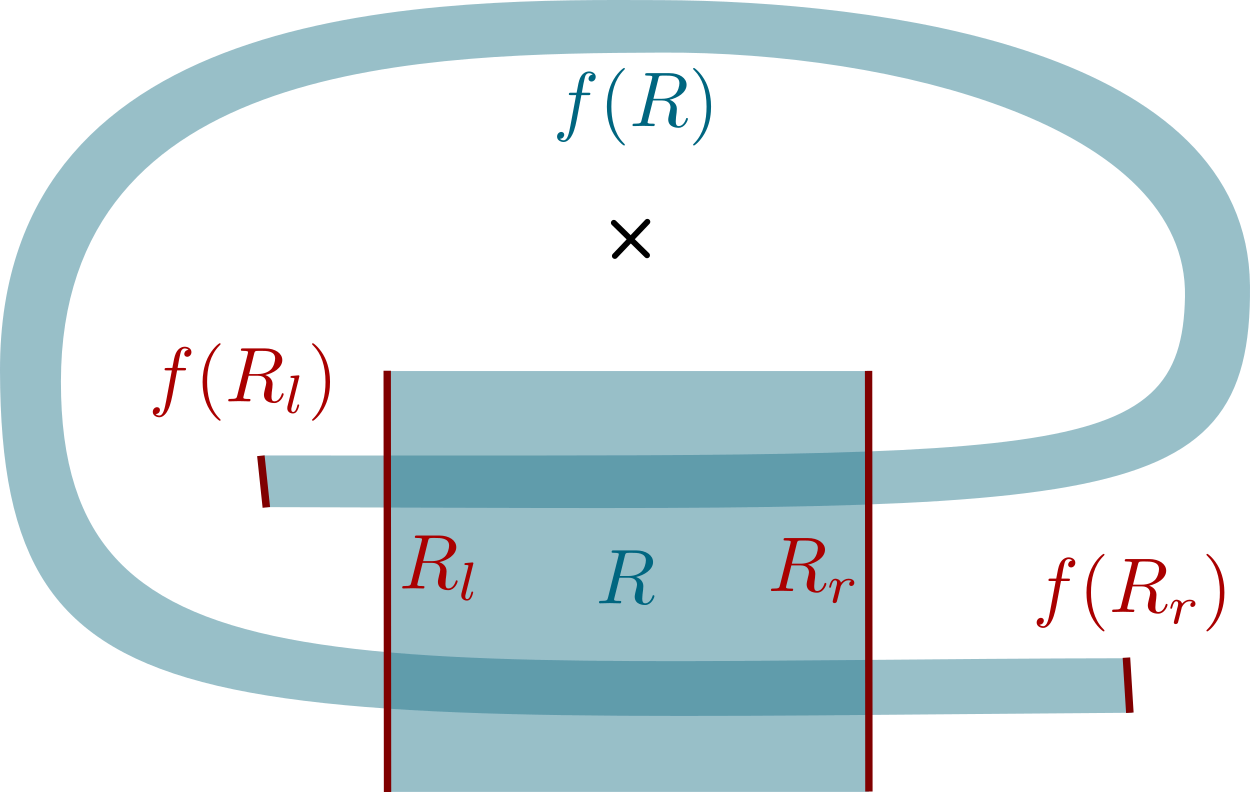}
	
	\caption{\small{The sets $R$ and $f(R)$ producing a rotational horseshoe with two symbols are shown in blue. If $H_0, H_1$ are the upper and lower components of $R\cap f(R)$, respectively, notice that $f^{-1}(H_0)$ and $f^{-1}(H_1)$ are two vertical rectangles in $R$. For a point $x$ whose full orbit is contained in the intersection of these four subrectangles, we can associate to $x$ the sequence $\sigma_x \in \{0,1\}^{\mathbb{Z}}$ given by its itinerary, i.e., $\sigma_x(n)=i$ if $f^n(x)\in H_i$.}}
	
\end{figure}

The following proposition is an adaptation of Proposition~9.16 in \cite{PAMiliton} to the annulus setting.

\begin{prop}
	Suppose that $f$ carries a rotational horseshoe with $k$ symbols and period $n$. Then there exist an $f$-invariant compact set $K\subset\mathbb{A}$, an $F^n$-invariant compact lift $\tilde{K}\subset\mathbb{R}^2$, and a surjective continuous map
	\[
	h\colon\tilde{K}\to\{1,\dots,k\}^{\mathbb{Z}}
	\]
	such that, denoting by $\tilde g=F^n|_{\tilde K}$, the following diagram commutes:
	\[
	\begin{tikzcd}
		\{1,\dots,k\}^{\mathbb{Z}} \arrow[r, "\sigma"]
		& \{1,\dots,k\}^{\mathbb{Z}} \\
		\arrow[d, "\pi"] \arrow[u, "h"] \tilde{K} \arrow[r, "\tilde{g}"]
		& \arrow[d, "\pi"] \arrow[u, "h"] \tilde{K} \\
		K \arrow[r, "f^n"]
		& K
	\end{tikzcd}
	\]
	where $\pi$ is the covering projection and $\sigma$ is the Bernoulli shift. Moreover, the preimage under $h$ of every $p$-periodic sequence for $\sigma$ contains a point that projects to a $p$-periodic point for $f^n$.
\end{prop}

\medskip

Observe that the compact invariant set $K$ in the previous proposition necessarily has a nontrivial rotation set. The symbolic itinerary of a point in $K$ determines its rotation around the annulus; in particular, the rotation number of every periodic orbit in $K$ can be computed directly from its symbolic itinerary. Another important remark is that if $f$ has a rotational horseshoe of period $n$ with $k$ symbols, then the topological entropy of $f$ is positive. Indeed,
\[
h_{\mathrm{top}}(f)\ge\frac{\log(k)}{n}.
\]
It is also possible to construct examples of rotational horseshoes for which arbitrarily large iterates of $f$ are required before the Markovian intersections appear. In other words, the topological entropy can be made arbitrarily small if one lacks information about \emph{when} these Markovian intersections occur (see, for instance, \cite{passpotrsamb}).

As in the case of the classical Smale horseshoe, if $\tilde{R}$ projects injectively onto a rectangle $R$ carrying a horseshoe for $f$, then for every negative semi-infinite sequence $\theta^-$ we can consider the local unstable set consisting of points sharing this semi-infinite sequence. This defines inside $R$ a local unstable set $X_{\theta^-}$ consisting of continua connecting $R_l$ to $R_r$. Moreover, each $X_{\theta^-}$ is contained in a larger unstable set $W^u(X_{\theta^-})$.

Since our maps need not be hyperbolic---we work in a purely topological setting---they may contain invariant disks, often referred to as \emph{elliptic islands}.

Finally, note that $\tilde{R}$ does not always project injectively onto a rectangle of $\mathbb{A}$. However, there always exists a finite intermediate cover $\check{\mathbb{A}}$ of $\mathbb{A}$ in which $\tilde{R}$ projects onto a rectangle $\check{R}$. If $\check{f}$ is the lift of $f$ to this finite cover, then the topological entropies of $\check{f}$ and $f$ coincide. In this case, we work in $\check{\mathbb{A}}$, which contains the projection of $\tilde{K}$ inside $\check{R}$, and the sets $X_{\theta^-}$ are defined as the projections of the corresponding sets $\check{X}_{\theta^-}$, which are continua connecting the left and right sides of $\check{R}$.

\subsection{Instability regions}

We use the term \emph{Birkhoff instability region}, or simply \emph{instability region} throughout this article, to refer to an $f$-invariant open subannulus $A \subseteq \mathbb{A}$ containing two periodics points realizing different rotation numbers and whose ends are \emph{Birkhoff-related}, i.e., for every pair of disjoint neighborhoods of the ends of $A$, the forward orbit of one intersects the other. As proved later in Proposition~\ref{equivinstabilityregion}, in the non-wandering setting, the condition that the ends are Birkhoff-related is equivalent to requiring that $A$ contains no $f$-invariant circloid.

 A crucial result establishes that rotational horseshoes and instability regions are closely related: every Birkhoff instability region contains a rotational horseshoe by results of Le Calvez and Tal in \cite{fabiopatricehorseshoes} (see Section~\ref{sec.Instability} for more details).

\section{Rotation sets of annular boundaries}
\label{seccionRotationsets}

In this section, we develop the key result that allows us to prove our main theorem. It concerns how the rotation set of an invariant annular continuum is influenced by the dynamics of topological disks intersecting it. More precisely, we establish the following.

\begin{thm}
	\label{primeendsrotationthmconservative}
	Let $f\in \mathrm{Homeo}_{\,0,nw}(\mathbb{A}) \cup \mathrm{Homeo}_{\,0,\lambda}(\mathbb{A})$.
	Let $A\subseteq \mathbb{A}$ be an $f$-invariant open annulus bounded below, and let $\mathcal{X}$ be the lower boundary component of $A$. Let $U$ be a closed disk meeting $\mathcal{X}$ such that
	\begin{itemize}
		\item $\operatorname{Int}(U \cap f(U)) \cap A \neq \emptyset$,
		\item $\bigcup_{i=0}^n f^i(U)$ is inessential for some $n\in \mathbb{N}$.
	\end{itemize}
	If $\tilde{U}$ is a lift of $U$ and $F$ is a lift of $f$ for which $F(\tilde{U}) \cap \tilde{U} \neq \emptyset$, then the corresponding prime ends rotation number from $A$ satisfies
	$$\rho^+_{\mathcal{X}}(F) \in [-1/n,1/n].$$
\end{thm}

The main ingredient in the proof of this result is to show that every lifted connected component of $U\cap A$ satisfies the property that its first $n$ iterates remain within the same \emph{fundamental domain} of $A$. To achieve this, we establish the main technical result of this article, Proposition~\ref{proppancitasincluidas} in Subsection~\ref{subsec.maximalcross}. In fact, we show that it suffices to study the behavior of the \emph{maximal} connected components of $U\cap A$.

A crucial remark here (see Remark~\ref{remarkcontraejemplo}) is that this type of result fails without the non-wandering or area-preserving hypotheses. The mechanism behind this failure also shows that our main theorem cannot be obtained as a direct corollary of Theorem~A in \cite{passeggi2023weak} by means of a naive perturbative argument.

The section concludes with two corollaries of Theorem~\ref{primeendsrotationthmconservative} concerning $f$-invariant circloids. In particular, we prove that every $f$-invariant circloid for a non-wandering or area-preserving map has a trivial rotation set. This generalizes the corresponding result for area-preserving maps established in Section~2.7.2 of \cite{Conejeros2019Applications} as a consequence of the main theorem of \cite{FRANKSLECALVEZ2003} and Theorem~2.8 of \cite{KoropeckiRealizing}.

\subsection{Constraints on Prime-Ends Rotation Numbers}

A fundamental result in \cite{passeggi2023weak} asserts, roughly speaking, that if a fixed point lies ``above'' and sufficiently close to a bounded annular continuum, then the upper prime-ends rotation number of the continuum is close to the rotation number of the fixed point. More precisely, one assumes the existence of a neighborhood of the fixed point such that the union of its forward iterates is inessential and intersects the continuum. Theorem~\ref{primeendsrotationthmconservative} considerably extends this result.

As mentioned above, it is not always possible to derive the results of \cite{passeggi2023weak} by means of a purely perturbative argument. Nevertheless, under suitable additional hypotheses, such an approach can indeed be carried out. We begin by considering a situation in which the desired conclusion can be obtained by adapting the strategy of \cite{passeggi2023weak}. However, as illustrated in Figure~\ref{ejemplopancitas}, it is not always possible to perturb the system without altering the invariant continuum.

The following proposition deals with the case in which a closed disk plays the role of a neighborhood of a fixed point near an $f$-invariant annular continuum.

\begin{prop}
	\label{primeendsrotationproposition}
	Let $f\in\mathrm{Homeo}_0(\mathbb{A})$, let $A\subseteq\mathbb{A}$ be an $f$-invariant open essential annulus bounded below and unbounded above, and let $U$ be a closed disk whose interior intersects $\partial A$ such that
	\begin{itemize}
		\item $f(U)\cap U\cap A\neq\emptyset$,
		\item $\displaystyle\bigcup_{i=0}^n f^i(U)$ is inessential for some $n\in\mathbb{N}$.
	\end{itemize}
	Consider a ray $r\subset A$ from $+\infty$ landing at an accessible point $r_0\in\partial A$ and disjoint from $\bigcup_{i=0}^n f^i(U)$. Choose a lift $\tilde U$ of $U$ meeting the region $\tilde A_r$ between two consecutive lifts $\tilde r$ and $\tilde r+(1,0)$ of $r$. Let $F$ be the lift of $f$ satisfying $F(\tilde U)\cap\tilde U\neq\emptyset$, and let $\tilde V$ be a connected component of $\tilde U\cap\tilde A_r$.
	
	If $F(\tilde V)\subseteq\tilde A_r$, then the upper prime-ends rotation number of $A$ satisfies
	\[
	\rho_A^+(F)\in[-1/n,\,1/n].
	\]
\end{prop}

\bigskip

To illustrate the idea behind Proposition~\ref{primeendsrotationproposition}, consider the following simplified setting. Let $\pi$ denote the covering projection of our annular model, and assume that $\partial A=\pi(\mathbb{R}\times\{0\})$ is a simple closed curve with rotation number $\rho$ for a lift $F$ as in the proposition. Let $V=\pi(\widetilde V)$ be the connected component given by the proposition, assume that $n=1$, and suppose that
\[
F(\widetilde V)\cap\widetilde V\neq\emptyset.
\]

Now consider the set $F(\widetilde V)\cap(\mathbb{R}\times\{0\})$. Since $F(\widetilde V)$ intersects $\widetilde V$, this set is contained in an interval of length less than one and meets $[0,1]\times\{0\}$. Iterating this observation, the image of $[0,1]\times\{0\}$ can drift by at most one fundamental domain at each iterate. Consequently,
\[
\frac{1}{i}\,\mathrm{pr}_1\!\left(F^i([0,1]\times\{0\})\right)
\subset
\left[-1,1+\frac1i\right],
\]
which implies that $\rho\in[-1,1]$.

In general, however, one cannot guarantee the existence of a connected set $V$ satisfying the hypotheses of Proposition~\ref{primeendsrotationproposition}. If one only assumes that $U$ intersects one of its iterates, without imposing any condition on the connected components of $U\cap A$, then the rotation of the continuum can no longer be controlled. This is precisely why Proposition~\ref{primeendsrotationproposition} requires the component $V$ to remain in the same region $\widetilde A_r$. The following remark presents an example showing that the conclusion of the proposition fails in this more general setting.

\bigskip

\begin{rmk}
\label{remarkcontraejemplo}	

Let us introduce the following example. Consider
\[
A=\pi\bigl(\mathbb{R}\times\mathbb{R}^{>0}\bigr), \qquad
\mathcal{X}=\partial A.
\]
Let $U$ be a disk that ``winds around'' $\mathcal{X}$ three times, generating three connected components of $U\cap A$, one of which is contained in the filling of another. Let $f$ be the identity composed with a perturbation supported on a closed disk, such that $U\cup f(U)$ is inessential and the image under $f$ of the ``larger'' component intersects the ``smaller'' one, as shown in Figure~\ref{ejemplopancitas} in the universal covering.

\begin{figure}[h!]
	
	\centering
	\includegraphics[width=360 pt]{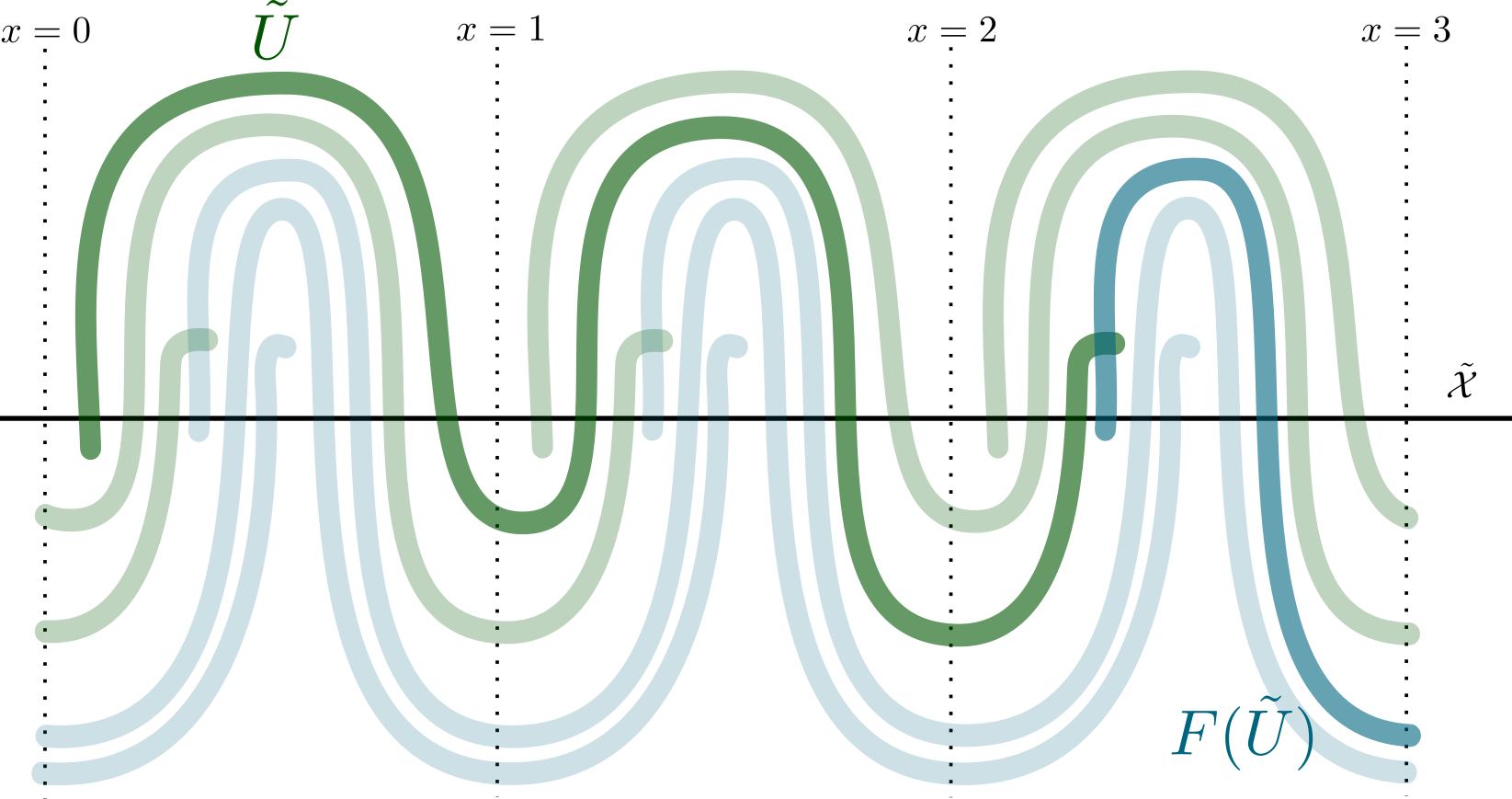}
	
	\caption{\small{The set $\widetilde{U}$ is shown in dark green, while the other lifts of $U$ are shown in light green. The image $F(\widetilde{U})$, shown in dark blue, intersects $\widetilde{U}$; however, the rotation number of $\mathcal{X}$ is equal to $2$. The other lifts of $f(U)$ are shown in light blue.}}
	\label{ejemplopancitas}
\end{figure}

Choose the lift $F$ so that, for every lift $\widetilde{U}$ of $U$, one has
$$
F(\widetilde{U})\cap\widetilde{U}\neq\emptyset.
$$
Then $F$ must have a fixed point on $\mathcal{X}$ with rotation number equal to $2$. Indeed, the larger component of $U\cap A$ determines a disk in its complement, namely its filling. The segment $I$ of $\mathcal{X}$ contained in the boundary of this disk is mapped by $f$ strictly inside itself, forcing the existence of a fixed point $x$ with $\rho_x(F)=2$.

In particular, this example shows that we are not capturing a certain notion of ``rotation'' of the disk $U$ around $\mathcal{X}$. Indeed, it is possible that $U\cup f(U)$ is inessential while, for some lift $F$, we simultaneously have
$$
F(\widetilde{U})\cap\widetilde{U}\neq\emptyset
\quad\text{and}\quad
\rho_{\mathcal{X}}(F)=2.
$$
Moreover, one can construct examples with ``more laps'', obtaining $\rho_{\mathcal{X}}(F)=k$ for arbitrarily large values of $k$. We also make the following remarks:

\begin{itemize}
	\item The hypothesis that there exists a connected component $V$ such that $F(\widetilde{V})$ remains in the same fundamental domain of the annulus is necessary in Proposition~\ref{primeendsrotationproposition}. In this example, every connected component is translated by a positive amount. Moreover, the figure shows that if one attempts to perturb $U$ in order to create a fixed point, then either $U\cup f(U)$ becomes essential or the perturbation must cross $\mathcal{X}$.
	
	\item As mentioned above, one can find disks in $A$ that are mapped strictly inside themselves by $f$, forcing the existence of wandering points. This shows that the non-wandering hypothesis in our main theorem cannot be removed.
\end{itemize}

\end{rmk}

\medskip

The proof of Proposition~\ref{primeendsrotationproposition} now follows.

\begin{proof}
	Denote by $\pi:\mathbb{R}^2\to\mathbb{A}$ the covering projection of the annulus, and by $\pi_A:\pi^{-1}(A)\to A$ its restriction. We first fix some notation and recall the prime-ends constructions described in the Preliminaries.
	
	\begin{itemize}
		\item By adding the end $+\infty$ to $A$, we identify $A\cup\{+\infty\}$ with an open disk and consider its prime-ends compactification. We then consider the prime-ends compactification of $A$. More precisely, there exists a homeomorphism $\varphi:\mathbb{D}\to A\cup\{+\infty\}$ satisfying $\varphi(0)=+\infty$, such that $\varphi^{-1}\circ f\circ\varphi$ extends continuously to a homeomorphism $g:\overline{\mathbb{D}}\to\overline{\mathbb{D}}$. Let $\ell=\varphi^{-1}(r)$ be the ray from $0$ landing at a point $\ell_0\in\partial\mathbb{D}$. We also write $\mathbb{D}^*=\mathbb{D}\setminus\{(0,0)\}$ so that $\varphi:\mathbb{D}^*\to A$.
		
		\item Let $\mathrm{pr}:\mathbb{H}\to\mathbb{D}^*$ be the universal covering projection. Choose lifts $\tilde r$ and $\tilde\ell$ of the rays $r$ and $\ell$, respectively. Then there exists a unique lift $\Phi:\mathbb{H}\to\pi^{-1}(A)$ of $\varphi$ satisfying $\Phi(\tilde\ell)=\tilde r$. Recall that, once $\Phi$ and the lift $F$ of $f$ are fixed, the lift $G$ of $g$ is uniquely determined so as to preserve the rotation information. More precisely,
		\[
		\Phi\circ G=F\circ\Phi.
		\]
		From now on, we identify every subset of $A$ with its image under $\varphi^{-1}$ and, similarly, every lift in $\pi^{-1}(A)$ with its image under $\Phi^{-1}$. In particular, we continue to denote by $V$ the corresponding subset of $\mathbb{D}$ and by $\widetilde V$ its lift in $\mathbb{H}$ between $\tilde\ell$ and $\tilde\ell+(1,0)$.
	\end{itemize}
	
	We therefore obtain the following commutative diagram
	
	\[
	\begin{tikzcd}
		\arrow[loop left, looseness=10, "G"] \mathbb{H}
		\arrow[r, "\Phi"]
		\arrow[d, "\mathrm{pr}"']
		&
		\pi^{-1}(A)
		\arrow[d, "\pi_A"]
		\arrow[loop right, looseness=4, "F"]
		\\
		\mathbb{D}^*
		\arrow[r, "\varphi"]
		\arrow[loop left, looseness=8, "g"]
		&
		A
		\arrow[loop right, looseness=12, "f"]
	\end{tikzcd}
	\]
	
	Since $\bigcup_{i=0}^n f^i(U)$ is inessential and disjoint from $r$, the lift containing $\widetilde V\cup F(\widetilde V)$ lies in $\widetilde A_r$. Hence $\bigcup_{i=0}^n F^i(\widetilde V)\subseteq \widetilde A_r$.
	Equivalently,
	\[
	\bigcup_{i=0}^n G^i(\widetilde V)
	\]
	is contained in the region of $\mathbb{H}$ between $\tilde\ell$ and $\tilde\ell+(1,0)$.
	
	Recall that the restriction of $g$ to $\partial\mathbb{D}$ defines the upper prime-ends rotation number $\rho_{\partial\mathbb{D}}(G)=\rho_A^+(F)$. It is well known (see the Preliminaries) that $\rho_A^+(F)\in\rho_A(F)$.

\begin{figure}[h!]
	
	\centering
	\hspace*{-15 pt}
	\includegraphics[width=500 pt]{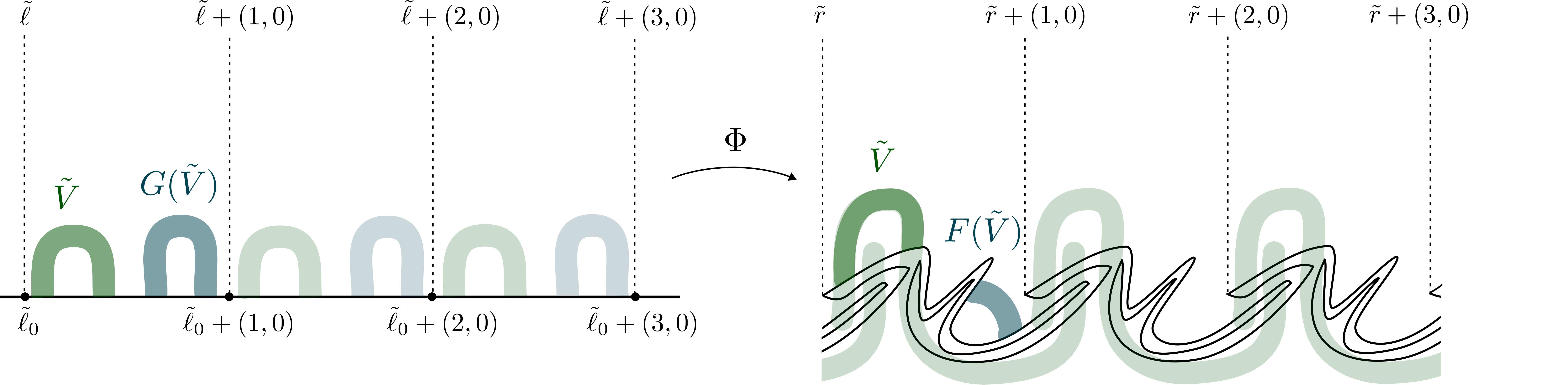}
	
	\caption{\small{Illustration of the behaviour in the universal cover for $n=1$}}
	\label{}
\end{figure}

If $\widetilde{\ell}_0=\Phi^{-1}(\tilde r_0)$ denotes the endpoint of $\widetilde{\ell}$, then $\mathcal{X}_V=\overline{\widetilde V}\cap\partial\mathbb{H}$ is contained in the interval $[\widetilde{\ell}_0,\widetilde{\ell}_0+(1,0)]\times\{0\}$. Since $\bigcup_{i=0}^n G^i(\widetilde V)$ remains between $\widetilde{\ell}$ and $\widetilde{\ell}+(1,0)$, we obtain
\[
G^n(\mathcal{X}_V)\subseteq[\widetilde{\ell}_0,\widetilde{\ell}_0+(1,0)]\times\{0\}.
\]
Hence
\[
G^n\!\left([\widetilde{\ell}_0,\widetilde{\ell}_0+(1,0)]\right)
\cap
[\widetilde{\ell}_0,\widetilde{\ell}_0+(1,0)]\times\{0\}
\neq\emptyset,
\]
and therefore $|\rho_{\partial\mathbb{D}}(G^n)|\le1$. Since
\[
G^n\!\left([\widetilde{\ell}_0,\widetilde{\ell}_0+(1,0)] \times\{0\} \right)
\subset
[\widetilde{\ell}_0-(n+1,0),\,\widetilde{\ell}_0+(n+1,0)]\times\{0\},
\]
we conclude that
\[
|\rho_{\partial\mathbb{D}}(G)|\le\frac1n.
\]

\end{proof}

\subsection{The set $\mathcal{W}$ of maximal cross-section}
\label{subsec.maximalcross}

As observed in the previous example, the map $f$ has a wandering point, since one can find a disk that is mapped strictly inside itself by $f$. We now adopt a more precise approach to understand the behavior of these disks---formally called cross-sections---which will allow us to prove that, for area-preserving or non-wandering maps, Proposition~\ref{primeendsrotationproposition} can indeed be applied to establish the main theorem of this section.

\medskip

Let $A$ be an $f$-invariant subannulus bounded below, and let $\mathcal{X}=\partial A$ denote its boundary. Let $U$ be a closed disk such that $\operatorname{Int}(U)$ intersects $\mathcal{X}$. The set $\partial U\cap A$ consists of a collection of open arcs, called \emph{crosscuts of $U$ in $A$}, which are pairwise disjoint in their interiors and land at two different points of $\partial A$.

Each crosscut determines a bounded disk in $A$, called a \emph{cross-section}. This allows us to define a partial order on the family of crosscuts by inclusion of their corresponding cross-sections. We denote by $\{C_i\}_{i\in I}$ the family of maximal crosscuts with respect to this order, and by $\{W_i\}_{i\in I}$ the corresponding family of cross-sections. We denote by
\[
\mathcal{W}_U=\bigcup_{i\in I}\overline{W_i}
\]
the union of the maximal cross-sections.

We record the following elementary observations.

\begin{itemize}
	
	\item Each cross-section satisfies $\partial W_i=\partial_{\mathcal X}W_i\cup C_i$,	where $\partial_{\mathcal X}W_i=\partial W_i\cap\mathcal X$.
	
	\item $W_i\cap W_j=\emptyset$ whenever $i\neq j$.
	
	\item The set $A\setminus\mathcal W_U$ is an open annulus and contains no point of $U$.
	
	\item A lift $\widetilde U$ of $U$ naturally determines the corresponding lifts $\{\widetilde C_i\}_{i\in I}$, $\{\widetilde W_i\}_{i\in I}$, and $\widetilde{\mathcal W}_U$.
	
\end{itemize}

\medskip

The aim of this subsection is to show that, whenever Proposition~\ref{primeendsrotationproposition} cannot be applied, one has
\[
f(\mathcal W_U)\subsetneq\mathcal W_U
\quad\text{or}\quad
\mathcal W_U\subsetneq f(\mathcal W_U),
\]
neither of which is possible for non-wandering or area-preserving maps.

We need to construct a suitable fundamental domain containing the maximal cross-sections. In many cases, it is possible to consider a vertical line from $+\infty$ to $-\infty$, disjoint from $U$, such that this line does not separate the maximal crosscuts in $A\setminus\mathcal{W}_U$. For instance, when $\mathcal{X}$ is a simple closed curve, one can always find such a vertical line. This allows us to localize the lift $\widetilde{\mathcal{W}}_U$ between two lifted rays determined by the vertical line.

However, in the general case, such a vertical line may not exist, as the boundary $\mathcal{X}$ can be arbitrarily complicated. An example of this behavior is given in Section~\ref{sec.toplem}. Nevertheless, it is possible to consider an ends-connector with analogous properties (see Section~\ref{s.pre} for the relevant definitions).

Given an ends-connector $K$, we say that $K$ does not separate two maximal crosscuts in $A$ if they can be joined by an arc contained in $(A\setminus\mathcal W_U)\setminus K$. The following proposition formalizes this construction by providing a continuum that plays the role of the vertical line from the simpler setting. Its proof is postponed to the final section. 

\begin{prop}
	\label{lematecnicoseparacionpancitas}
	Let $A$ be a subannulus bounded below, and let $U$ be a closed disk whose interior intersects $\mathcal{X}=\partial A$. Then there exists an ends-connector $K$ satisfying the following properties:
	\begin{itemize}
		\item $K$ is disjoint from $U$;
		\item $K$ does not separate maximal crosscuts of $U\cap A$ in $A\setminus\mathcal{W}_U$;
		\item $K\cap A$ is connected.
	\end{itemize}
\end{prop}

\medskip

For any other closed disk $V$ whose interior intersects $\mathcal{X}$, we define $\mathcal{W}_V$ analogously. In particular, we may consider the families of maximal crosscuts and cross-sections determined by either $U$ or $f(U)$ whenever their interiors intersect $\mathcal{X}$.

The following lemma provides the correspondence between the maximal crosscuts and cross-sections determined by $U$ and those determined by $f(U)$. Its proof is an immediate verification.

\begin{lemma}
	\label{lematecnicoimagenpancitas}
	It holds that
	\[
	f(\mathcal{W}_U)=\mathcal{W}_{f(U)}.
	\]
	Moreover, $f$ maps maximal crosscuts to maximal crosscuts, and maximal cross-sections to maximal cross-sections.
\end{lemma}

\medskip

Suppose that $\bigcup_{i=0}^n f^i(U)$ is inessential, and define
\[
V=\operatorname{Fill}\left(\bigcup_{i=0}^n f^i(U)\right).
\]

\begin{cor}
	\label{corodisjointK}
	One can choose an ends-connector $K$, disjoint from $V$, that does not separate the maximal crosscuts of $V\cap A$.
\end{cor}

\begin{proof}
	It suffices to apply Proposition~\ref{lematecnicoseparacionpancitas} to the disk $V$.
\end{proof}

The following lemma is straightforward but will play an important role in the proof of Proposition~\ref{proppancitasincluidas}. We omit its proof.

\begin{lemma}
	\label{crosscutsdisjuntos}
	Assume that no crosscut of $U\cap A$ intersects a crosscut of $f(U)\cap A$. Let $V=\operatorname{Fill}(U\cup f(U))$.
	Then every maximal crosscut (respectively, cross-section) of $V\cap A$ is a maximal crosscut (respectively, cross-section) of either $U\cap A$ or $f(U)\cap A$.
\end{lemma}

This ends-connector allows us to localize the set of maximal cross-sections $\mathcal{W}_U$ in the universal covering, as stated in the following corollary.

\begin{cor}
	\label{coropancitasencerradas}
	Let $K$ be the ends-connector given by Proposition~\ref{lematecnicoseparacionpancitas}, let $K^+$ be the connected component of $K\cap A$ accumulating at $+\infty$, which is disjoint from $\mathcal{W}_U$, and let $\widetilde A$ be the lift of $A$. Fix a lift $\widetilde U$ of $U$, let $\widetilde{\mathcal W}_U$ be the associated lift of $\mathcal W_U$, and let $\widetilde K^+$ be a lift of $K^+$. Denote by $\widetilde A_{\widetilde K}$ the connected component of $\widetilde A\setminus\left(\widetilde K^+\cup(\widetilde K^++(1,0))\right)$ lying between $\widetilde K^+$ and $\widetilde K^++(1,0)$. Then there exists a unique $k\in\mathbb Z$ such that
	\[
	\widetilde{\mathcal W}_U\subseteq \widetilde A_{\widetilde K}+(k,0).
	\]
\end{cor}

\begin{proof}
	By Proposition~\ref{lematecnicoseparacionpancitas}, any two maximal crosscuts $C_1$ and $C_2$ can be joined by an arc contained in $A\setminus(\mathcal W_U\cup K)$. Consequently, all the lifts of the maximal crosscuts (and therefore of the maximal cross-sections) determined by $\widetilde U$ belong to the same connected component of the complement of the lifts of $K^+$. Since the translates
	\[
	\widetilde A_{\widetilde K}+(m,0), \qquad m\in\mathbb Z,
	\]
	are pairwise disjoint, this proves the uniqueness of $k$.
\end{proof}

\begin{figure}[h!]
    \centering
     \includegraphics[width=210 pt]{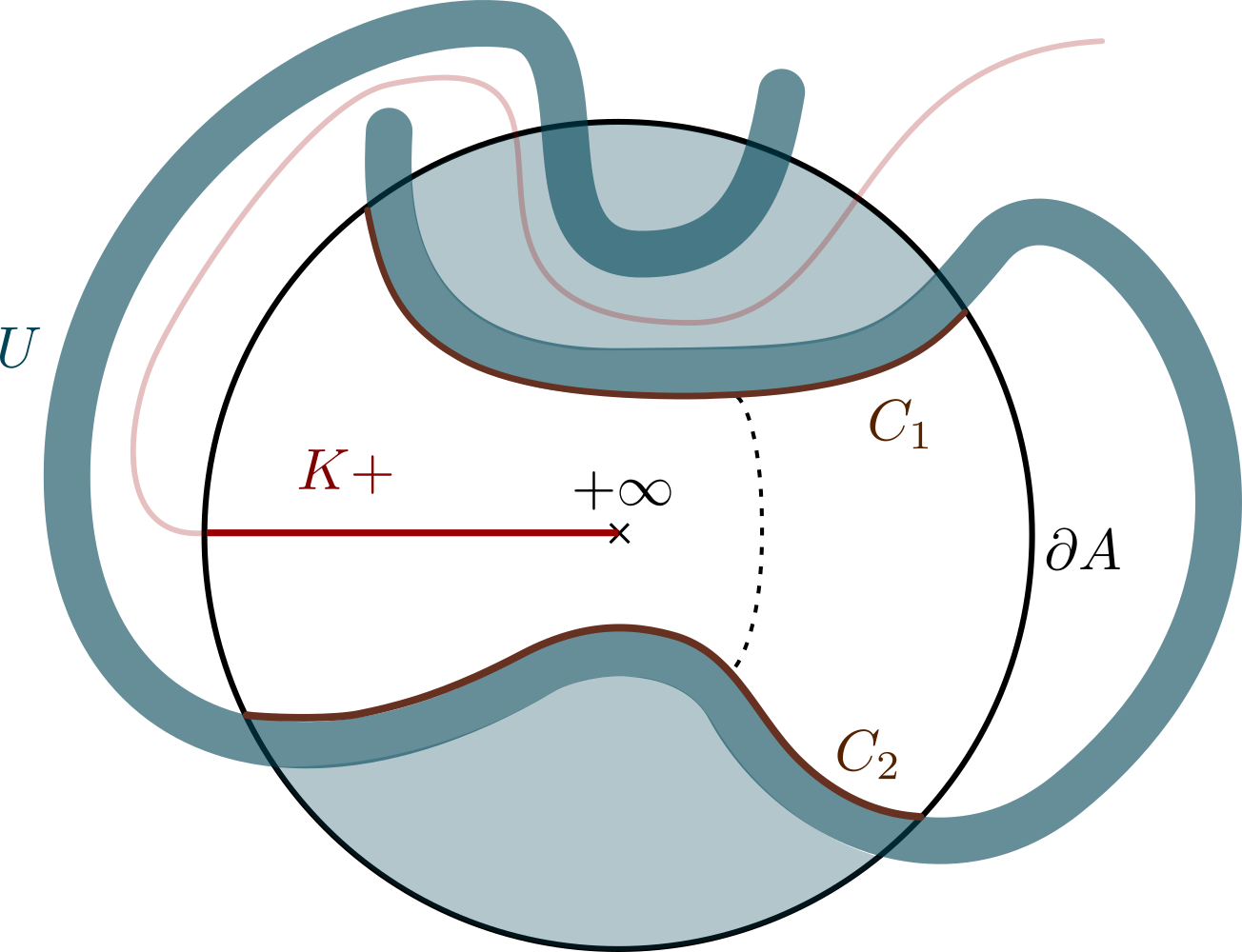}
    
    \caption{\small{The maximal crosscuts $C_1,C_2$ can be connected in $A\setminus (\mathcal{W}\cup K^+)$}}   
\end{figure}

The fact that $\widetilde{W}$ intersects only one lift of $\tilde{A}_{\tilde{K}}$ does not imply that $\widetilde{U}$ does as well. See for instance the case illustrated in Figure \ref{proppancitas}. In this situation, the disk $\widetilde{U}$ intersects two different lifts of $\tilde{A}_{\tilde{K}}$. 

\begin{figure}[h!]
    \centering
     \includegraphics[width=320 pt]{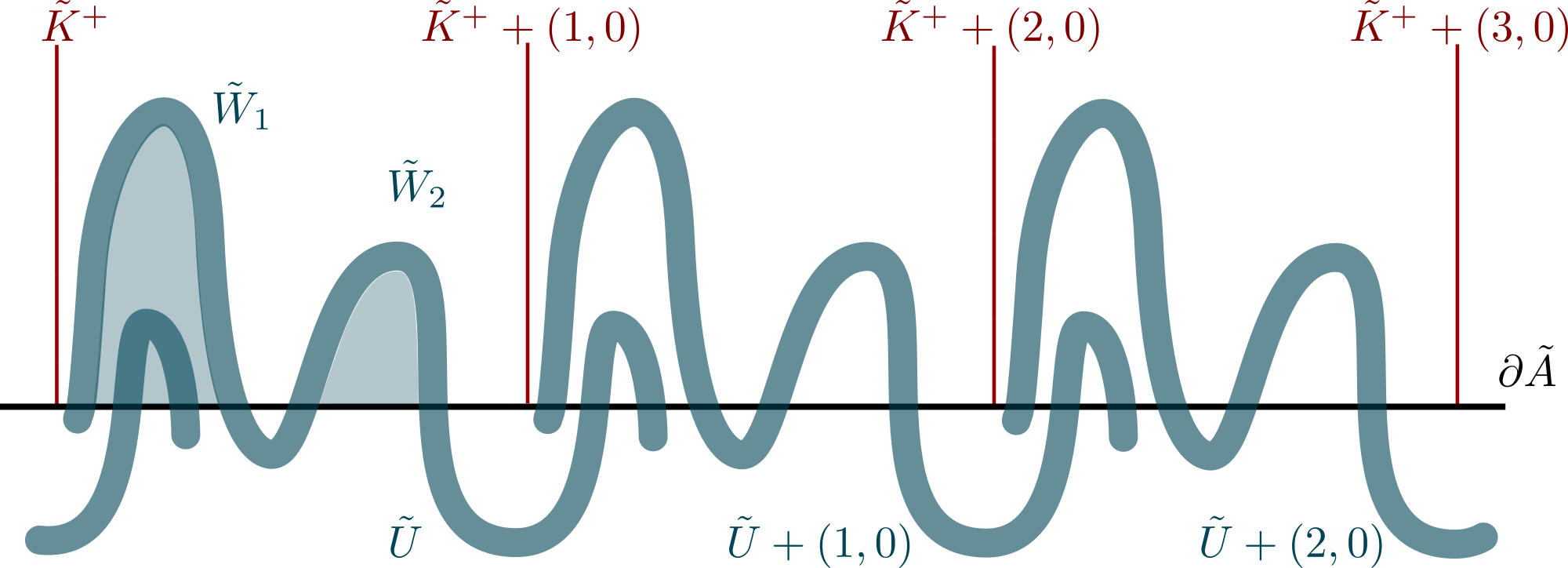}
    
    \caption{\small{The set $\tilde{\mathcal{W}}$ is between two lifts of $K^+$}}
    \label{proppancitas}
\end{figure}

\medskip

The following proposition formalizes the behavior exhibited by $f$ in Example~\ref{ejemplopancitas}. It is the main tool in the proof of Theorem~\ref{primeendsrotationthmconservative} and highlights the role of the non-wandering and area-preserving hypotheses.

\begin{prop}
	\label{proppancitasincluidas}
	Consider $f\in\mathrm{Homeo}_0(\mathbb{A})$ and an $f$-invariant open subannulus $A\subseteq\mathbb{A}$ bounded below. Assume that $U\subseteq\mathbb{A}$ is a closed disk satisfying $\mathrm{Int}(U\cap f(U))\cap A\neq\emptyset$ and $U\cup f(U)$ is inessential.
	Given a lift $\widetilde U$ of $U$, write $\widetilde{\mathcal W}_U$ for the corresponding lift of $\mathcal W_U$. If $F$ is a lift of $f$ such that $F(\widetilde U)\cap\widetilde U\neq\emptyset$ and $F(\widetilde{\mathcal W}_U)\cap\widetilde{\mathcal W}_U=\emptyset$,	then
	\[
	f(\mathcal W_U\cap A)\subset\operatorname{Int}(\mathcal W_U\cap A)
	\quad\text{or}\quad
	\mathcal W_U\cap A\subset\operatorname{Int}\bigl(f(\mathcal W_U\cap A)\bigr).
	\]
	Consequently, $f$ is neither area-preserving nor non-wandering.
\end{prop}

\begin{proof}
	
	Assume first that there is only one maximal cross-section $W$ associated with the maximal crosscut $C$ of $U$. The hypothesis implies that $F(\tilde{W})$ meets $\tilde{W}+(k,0)$ for some $k\neq0$. However,
	\[
	F(\tilde{C})\cap\bigl(\tilde{C}+(k,0)\bigr)=\emptyset,
	\]
	which implies that
	\[
	f(W\cap A)\subset\operatorname{Int}(W\cap A)
	\quad\text{or}\quad
	W\cap A\subset\operatorname{Int}\bigl(f(W)\cap A\bigr).
	\]
	We now assume that there is more than one maximal crosscut.
	
	The set $V=\operatorname{Fill}(U\cup f(U))$ is a closed disk. By Proposition~\ref{lematecnicoseparacionpancitas}, we can consider an ends-connector $K$ disjoint from $V$, with the property that the component $K^+$ does not intersect $\mathcal{W}_V$.	Note that every cross-section of $U$ or $f(U)$ in $A$ is contained in one of the cross-sections of $V$ in $A$. 
	
	 By Corollary~\ref{corodisjointK}, the set $K$ does not intersect either $U$ or $f(U)$. Moreover, it does not separate the maximal crosscuts of $V$, then, it is easy to see that $K$ does not separate the maximal crosscuts of $U$ in $A\setminus\mathcal{W}_U$, and does not separate those of $f(U)$ in $A\setminus\mathcal{W}_{f(U)}$.
	
	Consider the lift $\widetilde A$ of $A$. By Corollary~\ref{coropancitasencerradas}, we can choose a lift $\widetilde K^+$ of $K^+$ such that $\widetilde{\mathcal{W}}_U$ is contained in the region $\widetilde A_{\widetilde K}\subset\widetilde A$ between $\widetilde K^+$ and $\widetilde K^++(1,0)$. By the same corollary and Proposition~\ref{lematecnicoimagenpancitas}, there exists some $k\in\mathbb{Z}$ such that
	\[
	\widetilde{\mathcal{W}}_{f(U)}
	=
	F(\widetilde{\mathcal{W}}_U)
	\subset
	\widetilde A_{\widetilde K}+(k,0).
	\]
	
	If $\{C_i\}_{i\in I}$ and $\{W_i\}_{i\in I}$ denote the collections of maximal crosscuts and cross-sections of $\mathcal{W}_U$, respectively, with $C_i=\partial W_i\cap A$, then we have the following properties.
	
	\begin{enumerate}[label=(\roman*)]
		\item Since, by hypothesis $F(\widetilde{\mathcal{W}}_U)\cap\widetilde{\mathcal{W}}_U=\emptyset$,
		it follows that for all $i,j\in I$,
		\[
		F(\widetilde W_i)\cap\widetilde W_j=\emptyset.
		\]
		
		\item Since $F(\widetilde{\mathcal{W}}_U)\subset\widetilde A_{\widetilde K}+(k,0)$ and $F(\widetilde U)\cap\widetilde U\cap\widetilde A\neq\emptyset$, there exist $i,j\in I$ such that
		\[
		F(\widetilde W_i)\cap\bigl(\widetilde W_j+(k,0)\bigr)\neq\emptyset.
		\]
	\end{enumerate}
	
	Combining these two points, we obtain $k\neq0$. 
	Therefore, since $U\cup f(U)$ is inessential and $F(\tilde{U})\cap \tilde{U}\neq \emptyset$, we have $F(\partial\widetilde U)\cap\bigl(\partial\widetilde U+(k,0)\bigr)=\emptyset.$	Consequently, if for some $i,j\in I$ one has
	\[
	F(\widetilde W_i)\cap\bigl(\widetilde W_j+(k,0)\bigr)\neq\emptyset,
	\]
	then
	\[
	F(\widetilde C_i)\cap\bigl(\widetilde C_j+(k,0)\bigr)=\emptyset,
	\]
	leading to two possibilities:
	\[
	F(\widetilde W_i\cap\widetilde A)
	\subset
	\operatorname{Int}\bigl((\widetilde W_j\cap\widetilde A)+(k,0)\bigr)
	\quad\text{or}\quad
	(\widetilde W_j\cap\widetilde A)+(k,0)
	\subset
	\operatorname{Int}\bigl(F(\widetilde W_i)\cap\widetilde A\bigr).
	\]
	
	In other words, whenever a maximal cross-section of $F(\widetilde U)$ meets one of $\widetilde U+(k,0)$, since their boundaries in $\widetilde A_{\widetilde K}$ are disjoint, one is strictly contained in the interior of the cross-section determined by the other. We claim that all crosscuts have the same behavior: they are either all mapped strictly inside or all mapped strictly outside the corresponding family of crosscuts of $f(U)$.
	
	Whenever a crosscut is mapped strictly inside, so is its associated cross-section. Since the crosscuts of $U$ and $f(U)$ are disjoint in $A$, being strictly inside or outside admits the following equivalent characterization: a crosscut $C$ of $\mathcal{W}_U$ is strictly outside $\mathcal{W}_{f(U)}$ if it contains a point accessible from $+\infty$ in $A\setminus\mathcal{W}_{f(U)}$. On the contrary, it is strictly inside $\mathcal{W}_{f(U)}$ if it contains no point accessible from $+\infty$ in $A\setminus\mathcal{W}_{f(U)}$.
	
	Assume, by contradiction, that the conclusion of the proposition does not hold. Then, the previous discussion implies that there exist two crosscuts $C_1$ and $C_2$ such that $C_1$ is strictly outside $\mathcal{W}_{f(U)}$ and $f(C_2)$ is strictly outside $\mathcal{W}_U$. By Lemma~\ref{crosscutsdisjuntos}, both $C_1$ and $f(C_2)$ contain points accessible from $+\infty$ for the subannulus $A\setminus\mathcal{W}_V$. Therefore, it is possible to consider an arc
	\[
	\alpha\subset A\setminus(K\cup\mathcal{W}_V)
	\]
	joining $C_1$ to $f(C_2)$.
	
	Since $\alpha$ does not meet $K$, the set $U\cup\alpha\cup f(U)$ is inessential. Therefore, we can choose a lift $\widetilde\alpha$ such that
	\[
	\widetilde U\cup\widetilde\alpha\cup F(\widetilde U)
	\]
	projects onto the connected set $U\cup\alpha\cup f(U)$. In particular, $\widetilde\alpha$ is an arc connecting $\widetilde C_1$ to $F(\widetilde C_2)$. Since $\widetilde C_1$ is contained in $\widetilde A_{\widetilde K}$ and $\widetilde C_1\cup\widetilde\alpha\cup F(\widetilde C_2)$ is the lift of a connected set in $A\setminus K$, it must also be contained in the same region $\widetilde A_{\widetilde K}$. This contradicts the inclusion
	\[
	F(\widetilde C_2)\subset\widetilde A_{\widetilde K}+(k,0),
	\]
	where $k\neq0$.
	
	Finally, notice that a crosscut is mapped strictly outside if and only if it is mapped strictly inside by $f^{-1}$. Since the crosscuts of $f(U)$ are disjoint from those of $U$ in $A$, all crosscuts have the same behavior: they are either all mapped strictly inside or all mapped strictly outside the corresponding family of crosscuts of $f(U)$. In the first case, the associated cross-sections are also mapped strictly inside, yielding
	\[
	f(\mathcal{W}_U\cap A)\subset\operatorname{Int}(\mathcal{W}_U\cap A).
	\]
	In the second case, the previous observation applied to $f^{-1}$ gives
	\[
	\mathcal{W}_U\cap A\subset\operatorname{Int}\bigl(f(\mathcal{W}_U\cap A)\bigr).
	\]
	This completes the proof.
\end{proof}

\begin{rmk}
	The previous proposition describes the behavior of $\mathcal{W}_U$ under the assumption that $F(\widetilde{\mathcal{W}}_U)\cap\widetilde{\mathcal{W}}_U=\emptyset$. However, this is not the only possible configuration. For instance, the image of a maximal cross-section may intersect several maximal cross-sections. Nevertheless, if $F(\widetilde{\mathcal{W}}_U)\cap\widetilde{\mathcal{W}}_U\neq\emptyset$, then $F(\widetilde{\mathcal{W}}_U)$ is still contained in the same region between the lifted ends-connectors as $\widetilde{\mathcal{W}}_U$.
\end{rmk}

 \medskip

\subsection{Proof of Theorem~\ref{primeendsrotationthmconservative}}

In the non-wandering or area-preserving setting, the previous results allow us to estimate the prime ends rotation number of the boundary of an invariant essential annulus, provided that a closed disk $U$ intersects the boundary and remains inessential under its initial iterates.

We are now ready to prove Theorem~\ref{primeendsrotationthmconservative}, which is a key step toward the proof of the main theorem of this article.

\begin{proof}
	
	We may assume that $A=\mathcal{U}^+(\mathcal{X})$.
	
	Let $\widetilde{\mathcal W}_U$ be the lift of $\mathcal W_U$ determined by $\widetilde U$.
	
	By Corollaries~\ref{corodisjointK} and~\ref{coropancitasencerradas}, there exists an ends-connector $K$, disjoint from $\bigcup_{i=0}^n f^i(U)$, such that $\widetilde{\mathcal W}_U$ is contained in the region $\widetilde A_{\widetilde K}$ between two lifts of $K^+$, where $K^+$ denotes the connected component of $K\cap A$ accumulating at $+\infty$.
	
	Since $f$ is non-wandering or area-preserving Proposition~\ref{proppancitasincluidas} yields
	\[
	F(\widetilde{\mathcal W}_U)\cap\widetilde{\mathcal W}_U\neq\emptyset.
	\]
	Moreover, for each $i=0,\ldots,n$, Corollary~\ref{coropancitasencerradas} implies that $F^i(\widetilde{\mathcal W}_U)$ is contained in a unique translate of $\widetilde A_{\widetilde K}$. Since $f^i(U)\cap f^{i+1}(U)\cap A\neq\emptyset$,
	consecutive sets $F^i(\widetilde{\mathcal W}_U)$ and $F^{i+1}(\widetilde{\mathcal W}_U)$ intersect. Hence, all these translates coincide, and therefore
	\[
	F^i(\widetilde{\mathcal W}_U)\subset\widetilde A_{\widetilde K},
	\qquad i=0,\dots,n.
	\]
	
	Choose a maximal cross-section $W\subset\mathcal W_U$ such that
	\[
	F(\widetilde W)\subset\widetilde A_{\widetilde K},
	\]
	and let $V$ be the connected component of $U\cap A$ associated with $W$, namely the one satisfying $\partial W\cap A=\partial V\cap A$.	It follows that
	\[
	F(\widetilde V)\subset\widetilde A_{\widetilde K}.
	\]
	
	Since $K^+$ is disjoint from $\bigcup_{i=0}^n f^i(U)$, one can choose a ray $r\subset A$, contained in a sufficiently small neighborhood of $K^+$, disjoint from $\bigcup_{i=0}^n f^i(U)$, and landing on $\mathcal X$. Let $\widetilde r$ be the lift of $r$ close to $\widetilde K^+$, and let $\widetilde A_r$ denote the region between $\widetilde r$ and $\widetilde r+(1,0)$. Since $r$ can be chosen arbitrarily close to $K^+$,
	\[
	\tilde{V}\subset \widetilde A_r \quad \text{and} \quad F(\widetilde V)\subset\widetilde A_r.
	\]
	
	Proposition~\ref{primeendsrotationproposition} now completes the proof.
	
\end{proof}

\subsection{Consequences for invariant circloids}

We conclude this section with three consequences of Theorem~\ref{primeendsrotationthmconservative}. The first extends a result known for area-preserving maps (see \cite{Conejeros2019Applications}, \cite{KoropeckiRealizing}), showing that every invariant circloid of a non-wandering or area-preserving homeomorphism has trivial rotation set. The second and the third will be relevant in Section~\ref{sec.Instability}. 

\begin{prop}
	\label{propcircloidsrotaciontrivial}
	Let $f\in\mathrm{Homeo}_{\,0,nw}(\mathbb{A})\cup\mathrm{Homeo}_{\,0,\lambda}(\mathbb{A})$, and let $\mathcal X$ be an $f$-invariant circloid. Then $\rho_{\mathcal X}(F)=\{\rho\}$, where $\rho=\rho_{\mathcal X}^+(F)=\rho_{\mathcal X}^-(F)$.
\end{prop}

\begin{proof}
	Assume, by contradiction, that $\rho_{\mathcal X}(F)$ is not a singleton. Since $\rho_{\mathcal X}(F)$ is an interval, there exist three rational numbers in $\rho_{\mathcal X}(F)$. It is well known that every rational number in $\rho_{\mathcal X}(F)$ is realized by a periodic point. Moreover, as proved in \cite{koropeckipasseggi}, these periodic points may be chosen in $\partial\mathcal X$.
	
	Choose periodic points $x_0,x_1,x_2\in\partial\mathcal X$ satisfying
	\[
	\rho_{x_0}(F)<\rho_{x_1}(F)<\rho_{x_2}(F).
	\]
	Replacing $f$ by a suitable iterate $g=f^p$, we may assume that all three points are fixed. Writing $G=F^p$, their rotation numbers become integers satisfying
	\[
	\rho_{x_0}(G)<\rho_{x_1}(G)<\rho_{x_2}(G).
	\]
	
	Choose pairwise disjoint neighborhoods $U_0,U_1,U_2$ of $x_0,x_1,x_2$, respectively, such that
	\[
	\bigcup_{i=0}^3g^i(U_j)
	\]
	is inessential for each $j=0,1,2$, and these three unions are pairwise disjoint.
	
	Every neighborhood of a point in $\partial\mathcal X$ intersects at least one of the components $\mathcal U^+(\mathcal X)$ and $\mathcal U^-(\mathcal X)$. Hence, at least two of the neighborhoods $U_0,U_1,U_2$ intersect the same component. Without loss of generality, assume that both $U_0$ and $U_1$ intersect $\mathcal U^+(\mathcal X)$. Since $x_0$ and $x_1$ are fixed by $g$,
	\[
	U_i\cap g(U_i)\cap\mathcal U^+(\mathcal X)\neq\emptyset,
	\qquad i=0,1.
	\]
	Therefore, Theorem~\ref{primeendsrotationthmconservative} implies that
	\[
	\rho_{\mathcal X}^+(G)\in
	[\rho_{x_i}(G)-1/3,\rho_{x_i}(G)+1/3],
	\qquad i=0,1.
	\]
	This is impossible because these intervals are disjoint, as $\rho_{x_0}(G)$ and $\rho_{x_1}(G)$ are different integers. Hence, $\rho_{\mathcal X}(F)$ is a singleton.
	
	Therefore, by~\cite{HERNÁNDEZ-CORBATO_2017}, $\rho_{\mathcal X}^+(F),\, \rho_{\mathcal X}^-(F)\in\rho_{\mathcal X}(F)$. 
	Hence,
	\[
	\rho_{\mathcal X}^+(F)=
	\rho_{\mathcal X}^-(F),
	\]
	and both coincide with the unique element of $\rho_{\mathcal X}(F)$.
\end{proof}

Recall that, given two different circloids $\mathcal{C}^-\preccurlyeq\mathcal{C}^+$, we define
\[
\mathcal{R}(\mathcal{C}^-,\mathcal{C}^+)
:=
\mathcal{U}^+(\mathcal{C}^-)\cap\mathcal{U}^-(\mathcal{C}^+),
\quad
\mathcal{A}(\mathcal{C}^-,\mathcal{C}^+)
:=
\overline{\mathcal{U}^+(\mathcal{C}^-)}
\cap
\overline{\mathcal{U}^-(\mathcal{C}^+)},
\]
where $\mathcal{R}(\mathcal{C}^-,\mathcal{C}^+)$ is the region between the circloids and $\mathcal{A}(\mathcal{C}^-,\mathcal{C}^+)$ is the annular continuum delimited by them.

\begin{cor}
	\label{corolariodiscoentrecircloids}
	Let $f\in\mathrm{Homeo}_{0,nw}(\mathbb{A})\cup\mathrm{Homeo}_{0,\lambda}(\mathbb{A})$, and let $\mathcal{X}^-$ and $\mathcal{X}^+$ be two different $f$-invariant circloids such that
	$\mathcal{X}^-\preccurlyeq\mathcal{X}^+$ and
	$\mathcal{X}^-\cap\mathcal{X}^+\neq\emptyset$.
	Then, for every lift $F$ of $f$, the rotation set
	\[
	\rho_{\mathcal{A}(\mathcal{X}^-,\mathcal{X}^+)}(F)
	\]
	is a singleton consisting of a rational number. In particular, the rotation numbers of $\mathcal{X}^-$ and $\mathcal{X}^+$ coincide.
\end{cor}

\begin{proof}

Set
\[
\mathcal{R}:=\mathcal{R}(\mathcal{X}^-,\mathcal{X}^+),
\quad
\mathcal{A}:=\mathcal{A}(\mathcal{X}^-,\mathcal{X}^+).
\]
Since $\mathcal{X}^-$ and $\mathcal{X}^+$ meet, as commented in the Preliminaries each connected component of $\mathcal{R}$ is an open disk. Let $D$ be any such component. Since $\mathcal{R}$ is $f$-invariant, $f$ permutes its connected components. As $f$ is non-wandering on $\mathcal{A}$, the disk $D$ must therefore be periodic and hence contains a periodic point. Let $x_0\in D$ be a periodic point of period $m$, and define $g:=f^m$.

Let $G$ be the lift of $g$ satisfying $\rho_{x_0}(G)=0$.

We claim that there exist two arcs $r^-$ and $r^+$ joining $x_0$ to points
\[
z^-\in\mathcal{X}^-\cap(\mathcal{X}^+)^c
\qquad\text{and}\qquad
z^+\in\mathcal{X}^+\cap(\mathcal{X}^-)^c,
\]
respectively, whose interiors are contained in $D$. Indeed, take a ray starting at $x_0$ and accumulating at $-\infty$, disjoint from $\mathcal{X}^+$, and another ray starting at $x_0$ and accumulating at $+\infty$, disjoint from $\mathcal{X}^-$. Since $x_0\in\mathcal{R}(\mathcal{X}^-,\mathcal{X}^+)$, the first ray must intersect $\mathcal{X}^-$ and the second one must intersect $\mathcal{X}^+$. Truncating each ray at its first intersection with the corresponding circloid gives the desired arcs $r^-$ and $r^+$.

Since $g(D)=D$ and both circloids are $g$-invariant, for every $i\in\mathbb{N}$ the arcs $g^i(r^\pm)$ have interiors contained in $D$ and
\[
g^i(z^\pm)\in\mathcal{X}^\pm\cap(\mathcal{X}^\mp)^c.
\]

Fix $n\in\mathbb{N}$. We may choose sufficiently small neighborhoods $U^\pm$ of $r^\pm$ such that they are topological disks satisfying
\[
U^\pm\cap g(U^\pm)\neq\emptyset
\]
and
\[
\bigcup_{i=0}^n g^i(U^\pm)
\]
is inessential. Moreover, every lift $\widetilde U^\pm$ of $U^\pm$ satisfies
\[
G(\widetilde U^\pm)\cap\widetilde U^\pm\neq\emptyset.
\]

Therefore, applying Theorem~\ref{primeendsrotationthmconservative} to $U^+$ and $U^-$ gives
\[
\rho_{\mathcal{X}^+}(G)\in[-1/n,1/n]
\qquad\text{and}\qquad
\rho_{\mathcal{X}^-}(G)\in[-1/n,1/n].
\]
Since $n$ is arbitrary,
\[
\rho_{\mathcal{X}^+}(G)
=
\rho_{\mathcal{X}^-}(G)
=
0.
\]
Hence,
\[
\rho_{x_0}(G)
=
\rho_{\mathcal{X}^+}(G)
=
\rho_{\mathcal{X}^-}(G)
=
0.
\]

Since $D$ is $g$-invariant and $G$ is the lift for which $\rho_{x_0}(G)=0$, the lift of $D$ containing a lift of $x_0$ is $G$-invariant. Recall that for a point $x$ with bounded orbit, whenever the limit
\[
\lim_{n\to+\infty}\frac{\operatorname{pr}_1(G^n(\tilde{x})-\tilde{x})}{n}
\]
exists, we define its rotation number $\rho_x(G)$ as this limit. In particular, if $x$ is a recurrent point in $D$ whose rotation number exists, then $\rho_x(G)=0$. Indeed, it is enough to take a sequence $n_k\to+\infty$ such that $g^{n_k}(x)\to x$. Since $G$ leaves invariant the lift of $D$ containing $\tilde{x}$, we must have
\[
G^{n_k}(\tilde{x})\to\tilde{x},
\]
and hence $\rho_x(G)=0$.

On the other hand, since the rotation sets of the invariant circloids $\mathcal{X}^+$ and $\mathcal{X}^-$ are both the singleton $\{0\}$, any recurrent point in either of these circloids whose rotation number exists must also have rotation number zero.

We conclude that, since $D$ was an arbitrary connected component of $\mathcal{R}$ and
\[
\mathcal{A}\setminus\mathcal{R}\subseteq\mathcal{X}^-\cup\mathcal{X}^+,
\]
every recurrent point in $\mathcal{A}$ whose rotation number exists has rotation number zero for $G$.

To show that $\rho_{\mathcal{A}}(G)=\{0\}$, assume by contradiction that this rotation set is not a singleton. Since the rotation set of any compact invariant continuum is a compact interval---see, for instance, \cite{passeggi2023weak}---the rotation set of $\mathcal{A}$ is a nondegenerate interval $[a,b]$, with at least one of its endpoints different from $0$, say $b\neq0$. It is shown in Lemma~1(a) of \cite{BoylandAndreHall} that the extremal points of the Misiurewicz--Ziemian rotation set---the one considered in this article---of a compact metric space are realized by recurrent points. In our setting, the compact metric space is $\mathcal{A}$ and we consider the extremal point $b\neq0$. Thus, there exists a recurrent point $z\in\mathcal{A}$ such that
\[
\rho_z(G)
=
\lim_{n\to+\infty}
\frac{\operatorname{pr}_1(G^n(\tilde{z})-\tilde{z})}{n}
=
b
\]
for any lift $\tilde{z}$ of $z$.

But, as observed above, $z$ can belong neither to $\mathcal{R}$ nor to $\mathcal{X}^+\cup\mathcal{X}^-$, leading to a contradiction.

It follows that
\[
\rho_{\mathcal{A}}(G)=\{0\}.
\]

Therefore, for every lift $F$ of $f$, the rotation set
\[
\rho_{\mathcal{A}}(F)
\]
is a singleton. Since it coincides with the rotation number of the periodic point $x_0$, this singleton consists of a rational number.

\end{proof}

\bigskip

We obtain the following two useful lemmas.

\begin{lemma}
	\label{lemacircloidsdisjuntosdosdiscos}
	Let $f\in\mathrm{Homeo}_{0,nw}(\mathbb{A})
	\cup
	\mathrm{Homeo}_{0,\lambda}(\mathbb{A})$, and let $U_0,U_1$ form an $n$-dpd for some $n\geq 3$, with $\rho_f(U_0,U_1)=\rho\in\mathbb{N}^*$.
	Let $\mathcal{X}^-$ and $\mathcal{X}^+$ be two different $f$-invariant circloids such that $\mathcal{X}^-\preccurlyeq\mathcal{X}^+$. If both disks $U_0,U_1$ meet the annular continuum $\mathcal{A}(\mathcal{X}^-,\mathcal{X}^+)$, then
	\[
	\mathcal{X}^-\cap\mathcal{X}^+=\emptyset.
	\]
	In particular, $\mathcal{R}(\mathcal{X}^-,\mathcal{X}^+)$ is a regular annulus.
	
	Furthermore, if $\rho \geq 2$, the result remains valid for $n \geq 2$, and if $\rho \geq 3$, it holds for every $n \in \mathbb{N}$.
\end{lemma}

\begin{proof}
	Assume, by contradiction, that
	$\mathcal{X}^-\cap\mathcal{X}^+\neq\emptyset$. By Corollary~\ref{corolariodiscoentrecircloids}, the annular continuum $\mathcal{A}(\mathcal{X}^-,\mathcal{X}^+)$ has a trivial rotation set.

	Let $F$ be the lift of $f$ for which
	\[
	F(\widetilde{U}_0)\cap\widetilde{U}_0\neq\emptyset
	\]
	for every lift $\widetilde{U}_0$ of $U_0$. Since $U_0$ intersects $\mathcal{A}(\mathcal{X}^-,\mathcal{X}^+)$, either it is contained in a disk component of its interior or it intersects its boundary. In the first case, the periodicity of this disk component implies that the unique rotation number of $\mathcal{A}(\mathcal{X}^-,\mathcal{X}^+)$ is $0$, while in the second case Theorem~\ref{primeendsrotationthmconservative} gives
	\[
	\rho_{\mathcal{A}(\mathcal{X}^-,\mathcal{X}^+)}(F)
	\in[-1/n,1/n].
	\]
	Thus, in either case,
	\[
	\rho_{\mathcal{A}(\mathcal{X}^-,\mathcal{X}^+)}(F)
	\in[-1/n,1/n].
	\]

	The same argument applied to $U_1$ gives
	\[
	\rho_{\mathcal{A}(\mathcal{X}^-,\mathcal{X}^+)}(F)
	\in[\rho-1/n,\rho+1/n],
	\]
	which is impossible since $n\geq3$ and $\rho\in\mathbb{N}^*$. Hence
	\[
	\mathcal{X}^-\cap\mathcal{X}^+=\emptyset.
	\]
	Therefore, $\mathcal{R}(\mathcal{X}^-,\mathcal{X}^+)$ is a regular annulus.

\end{proof}

\begin{lemma}
	\label{lematecnicocircloidmaximal}
	Consider $f \in \mathrm{Homeo}_{\,0,nw}(\mathbb{A})\cup \mathrm{Homeo}_{\,0,\lambda}(\mathbb{A})$ and a closed disk $U$ such that $\operatorname{Int}(U\cap f(U))\neq\emptyset$. Let $F$ be the lift of $f$ for which every lift of $U$ meets its image under $F$.
	Then, for every integer $N>0$, there exists a neighborhood $\mathcal{V}$ of the upper end of $\mathbb{A}$ such that any $f$-invariant circloid $\mathcal{C}$ satisfying $\mathcal{C}\cap\mathcal{V}\neq\emptyset$ and $U\subset\mathcal{U}^+(\mathcal{C})$ verifies
	\[
	\rho_{\mathcal{C}}(F)\in[-1/N,1/N].
	\]
\end{lemma}

\begin{proof}
	
	Fix an integer $N>0$. For each $r\in\mathbb{R}$, let
	\[
	\gamma_r=S^1\times\{r\},
	\]
	and define
	\[
	W_r^+:=\mathcal{U}^+(\gamma_r),\qquad
	W_r^-:=\mathcal{U}^-(\gamma_r).
	\]
	
	Choose $r_0$ such that $U\subset W_{r_0}^-$, and then choose $K$ sufficiently large so that
	\[
	\bigcup_{i=0}^Nf^i(W_{r_0}^-)\subset W_K^-.
	\]
	
	Set $\mathcal{V}:=W_K^+$, and let $\mathcal{C}$ be an $f$-invariant circloid satisfying $\mathcal{C}\cap\mathcal{V}\neq\emptyset$ and $U\subset\mathcal{U}^+(\mathcal{C})$. Since $\bigcup_{i=0}^Nf^i(U)\cap\mathcal{V}=\emptyset$, the connected component of $\operatorname{Int}\bigl(\mathcal{U}^+(\mathcal{C})\cap W^-_K\bigr)$ containing $\bigcup_{i=0}^Nf^i(U)$ is a disk. In particular, $\bigcup_{i=0}^Nf^i(U)$	is inessential.
	
	Next choose $s<r_0$ such that $\gamma_s\cap U\neq\emptyset$, and let $\hat U$ be a sufficiently small neighborhood of the connected component of
	\[
	U\cup\bigl(\gamma_s\cap\mathcal{U}^+(\mathcal{C})\bigr)
	\]
	containing $U$, such that $\bigcup_{i=0}^Nf^i(\hat U)$	is inessential and contained in $W_K^-$.
	
	Since $\gamma_s$ intersects $\mathcal{C}$, the interior of $\hat U$ also intersects $\mathcal{C}$. Applying Theorem~\ref{primeendsrotationthmconservative}, we obtain
	\[
	\rho_{\mathcal{C}}(F)\in[-1/N,1/N].
	\]
	
\end{proof}

\section{Rotational chaos and instability regions}
\label{sec.Instability}

In this section, we prove Theorems~\ref{maintheorem} and~\ref{maintheorembr}, which establish the existence of rotational chaos under suitable conditions on the disks $U_0$ and $U_1$. A key ingredient is Theorem~\ref{primeendsrotationthmconservative}.

One consequence---although not an immediate one---of the existence of two Birkhoff-related disks is that the rotation set of the map must be nontrivial. More precisely, in a general setting, without assuming that the map is non-wandering, the existence of two disjoint disks $U_0$ and $U_1$ with positive rotational difference that are $3$--Birkhoff related implies the existence of two periodic points with different rotation numbers, and hence a nontrivial rotation set.

In the conservative setting, this conclusion can be strengthened: every invariant subannulus intersecting both disks must itself have a nontrivial rotation set and contain at least two periodic points with different rotation numbers. For readers familiar with the subject, this result is perhaps not unexpected. The proof, deferred to Section~\ref{sec.toplem}, is based on ideas from the celebrated article of Franks~\cite{Franks(PoincBirkh)} involving \emph{periodic chains of disks}. The precise statement is given in the proposition below. Moreover, under the assumptions of the main theorems, it establishes the realization of every rational number in a suitable interval by periodic points. This part of the argument relies on results of Le Calvez~\cite{LeCalvez2005IHES}.

Recall that a homeomorphism $f\in\mathrm{Homeo}_{\,0}(\mathbb{A})$ satisfies the \emph{curve intersection property} if every essential simple closed curve $\gamma\subset\mathbb{A}$ satisfies
\[
f(\gamma)\cap\gamma\neq\emptyset.
\]
We also recall that a regular annulus is an open essential subannulus bounded either by circloids or by the ends of $\mathbb{A}$.

\begin{prop}
	\label{propbirkhoffrelated}
Consider $f \in \text{Homeo}_{\,0,nw}(\mathbb{A}) \cup \text{Homeo}_{\,0,\lambda}(\mathbb{A})$. Let $\overline{A} \subset \mathbb{A}$ be the closure of a regular $f$--invariant subannulus $A$, and let $U_0, U_1$ be an $n$--dpd with $n \geq 3$ and $\rho_f(U_0,U_1) = \rho \in \mathbb{N}^*$, such that both $U_0$ and $U_1$ intersect $A$.
Assume that $U_0$ visits $U_1$ and $U_1$ visits $U_0$. If $F$ is the lift of $f$ for which $F(\tilde{U}_0)\cap \tilde{U}_0 \neq \emptyset$, then
\[
[1/n,\, \rho - 1/n] \subseteq \rho_{A}(F),
\]
and $A$ contains at least two periodic points $z_1$ and $z_2$ with different rotation numbers. Furthermore
\begin{itemize}
	\item If $A$ is bounded below or above or,
	\item if $A=\mathbb{A}$ and verifies the curve intersection property
\end{itemize}
Then every rational in $(1/n,\rho-1/n)$ is realized by a periodic point in $A$.
\medskip

Moreover, if $U_0$ and $U_1$ do not meet $\partial A$, the existence of such periodic points does not require the non-wandering or area-preserving assumption, nor the regularity of $A$. More precisely, for any $f\in\mathrm{Homeo}_{\,0}(\mathbb{A})$ and any $f$-invariant open subannulus $A$ containing $U_0$ and $U_1$, we can ensure the existence of periodic points $z_1,z_2\in A$ realizing the rotation numbers $1/n$ and $\rho-1/n$, respectively, for the lift $F$.
\medskip

If $\rho \geq 2$, the result remains valid for $n \geq 2$, and if $\rho \geq 3$, it is valid for $n \geq 1$.
\end{prop}

\begin{rmk}
	Recall that a \emph{Birkhoff instability region} is an $f$-invariant subannulus containing two periodic points with different rotation numbers whose ends are Birkhoff-related, i.e., every pair of disjoint neighborhoods of its ends satisfies that the forward orbit of one intersects the other. The existence of a Birkhoff instability region implies the existence of a rotational horseshoe. Therefore, the proofs of Theorems~\ref{maintheorem} and~\ref{maintheorembr} reduce to establishing the existence of an instability region intersecting both $U_0$ and $U_1$.
\end{rmk}

\begin{rmk}
	\label{remarkCIPinstaregions}
	Observe that every essential simple closed curve $\gamma$ contained in a Birkhoff instability region $A$ satisfies $f(\gamma)\cap\gamma\neq\emptyset$. 
	Indeed, otherwise one end of $A$ cannot be visited from the other, contradicting the fact that they are Birkhoff-related. Therefore, every Birkhoff instability region satisfies the curve intersection property.
\end{rmk}

\bigskip

\subsection{Proof of Theorem~\ref{maintheorem}}

For the reader's convenience, we recall the statement of Theorem~\ref{maintheorem}.

\begin{thmsinnum}{\textbf{A}}
	Let $f \in \mathrm{Homeo}_{0,nw}(\mathbb{A})$ and assume that, for some $n \ge 3$, there exists an $n$-dpd $U_0, U_1$ with nonvanishing rotational difference $\rho \in \mathbb{N}^*$ such that $U_0$ visits $U_1$.
	
	Then $f$ exhibits rotational chaos and the rotational horseshoe for a power of $f$ can be considered included in a Birkhoff instability region $\mathcal{I}$ meeting both $U_0$ and $U_1$. Moreover, for some lift $F$ of $f$,
	\[
	[1/n, \rho - 1/n] \subseteq  \rho_{\overline{\mathcal{I}}}(F)
	\]
	and every rational number in $(1/n,\rho-1/n)$ is realized by a periodic point in $\mathcal{I}$.
	
	If $\rho \ge 2$, the result remains valid for $n \ge 2$; and if $\rho \ge 3$, it remains valid for $n \ge 1$.
\end{thmsinnum}

\bigskip

We will make use of the following result, which appears as Proposition~D in the article by Tal and Le Calvez \cite{fabiopatricehorseshoes} and implies that every Birkhoff instability region carries a topological rotational horseshoe. In the mentioned article the rotation number of a point is generalized to non-escaping points, but in our context we only need the classic rotation number associated to periodic points. We refer the reader to \cite{fabiopatricehorseshoes} for further details.

\begin{prop}[Proposition D in \cite{fabiopatricehorseshoes}]
	\label{propDFabioPatrice}
	Let $f \in \mathrm{Homeo}_{\,0}(\mathbb{A})$ and let $F$ be a lift of $f$. Assume that:
	\begin{itemize}
		\item $F$ has at least two different rotation numbers; and
		\item the ends of $\mathbb{A}$ belong to the same Birkhoff recurrence class.
	\end{itemize}
	Then some iterate of $f$ admits a topological rotational horseshoe.
\end{prop}

Notice that the assumption that the ends of $\mathbb{A}$ belong to the same Birkhoff recurrence class is weaker than requiring them to be Birkhoff related (see \cite{fabiopatricehorseshoes} for the definitions). Consequently, every Birkhoff instability region as defined here carries a rotational horseshoe.

A useful observation, proved below, is that in the non-wandering setting there is an equivalent characterization of Birkhoff instability regions: they are precisely the open subannuli with non-trivial rotation set that contain no $f$-invariant circloid.

\begin{prop}
	\label{equivinstabilityregion}
	Let $f \in \mathrm{Homeo}_{0,nw}(\mathbb{A})$, and let $A$ be an open $f$-invariant subannulus. Then the following are equivalent:
	\begin{enumerate}
		\item The ends of $A$ are Birkhoff-related.
		\item The subannulus $A$ contains no $f$-invariant circloid.
	\end{enumerate}
	In particular, if $A$ contains periodic points with different rotation numbers, then $A$ is a Birkhoff instability region if and only if it contains no $f$-invariant circloid.
\end{prop}

\begin{proof}
	It is immediate that the first condition implies the second, since if $A$ contained an $f$-invariant circloid $\mathcal{C}$, then $\mathcal{C}$ would separate the ends of $A$. It remains to prove that the second condition implies the first.
	
	We identify $A$ with the standard open annulus $\mathbb{A}=S^1\times\mathbb{R}$. Let $V^{\pm}$ be disjoint neighborhoods of the ends $\pm\infty$, respectively. It is enough to prove that the forward orbit of one of these neighborhoods meets the other, since, by the non-wandering hypothesis, the converse relation follows automatically.

	Assume, by contradiction, that the forward orbit of $V^+$ does not intersect $V^-$. Define
	\[
	\mathcal{V}^+:=\cl\left[\operatorname{Fill}\left(\bigcup_{n\in\mathbb{N}} f^n(V^+)\right)\right].
	\]
	By our assumption, the interiors of $\mathcal{V}^+$ and $V^-$ are disjoint. Consider the annular continuum
	\[
	\mathcal{K}:=\operatorname{Fill}(\partial\mathcal{V}^+),
	\]
	and let $\mathcal{X}^+$ be its upper circloid.

	It follows immediately from the definition that $\mathcal{V}^+$ is forward invariant. Since $f$ is non-wandering, this inclusion cannot be strict, and therefore
	\[
	f(\mathcal{V}^+)=\mathcal{V}^+.
	\]
	Hence,
	\[
	f(\mathcal{K})=\mathcal{K}.
	\]
	Since the image of the upper circloid of an annular continuum is the upper circloid of its image, the uniqueness given by Lemma~\ref{lemaunicocircloidborde} implies that
	\[
	f(\mathcal{X}^+)=\mathcal{X}^+.
	\]

	Thus, $\mathcal{X}^+\subset A$ is an $f$-invariant circloid, contradicting assumption~(2). Hence, the ends of $A$ are Birkhoff-related, completing the proof.

\end{proof}

\bigskip

Recall that every instability region contains a rotational horseshoe. Therefore, it remains to prove the existence of an instability region meeting both $U_0$ and $U_1$. The proof of Theorem~\ref{maintheorem} follows.

\medskip

\begin{proof}

		We fix $F$ as the lift of $f$ for which every lift $\widetilde{U}_0$ of $U_0$ satisfies
		\[
		F(\widetilde{U}_0)\cap\widetilde{U}_0\neq\emptyset.
		\]
		We first prove that $U_0$ and $U_1$ are contained in a regular Birkhoff instability region. The estimates on the rotation set and the realization statement follow immediately and will be established at the end of the proof.
		
		\medskip
		
		First assume that no $f$-invariant circloid meets either $U_0$ or $U_1$, and define
		\[
		A:=
		\operatorname{Fill}\left(
		\bigcup_{k\in\mathbb{Z}}
		f^k\bigl(\operatorname{Int}(U_0\cup U_1)\bigr)
		\right).
		\]
		
		By construction, $A$ satisfies the following properties:
		
		\begin{itemize}
			\item $A$ is an essential annulus.
			
			\item Every $f$-invariant circloid contained in $A$ intersects
			\[
			\bigcup_{k\in\mathbb{Z}}
			f^k\bigl(\operatorname{Int}(U_0\cup U_1)\bigr).
			\]
		\end{itemize}
		
		In the present situation, the second property implies that $A$ contains no $f$-invariant circloid. Since the interiors of $U_0$ and $U_1$ are contained in $A$, which is not necessarily a regular annulus, we may apply the final part of Proposition~\ref{propbirkhoffrelated} to obtain two periodic points with different rotation numbers. It follows that $A$ is an instability region, completing the proof in this case.
		
		\medskip
		
		From now on, assume that there exists an $f$-invariant circloid meeting one of the two disks. By Theorem~\ref{primeendsrotationthmconservative}, every $f$-invariant circloid can intersect at most one of the disks $U_0$ and $U_1$. Hence, without loss of generality, assume that there exists an $f$-invariant circloid $\mathcal{C}$ intersecting $U_0$ such that
		\[
		U_1\subset\mathcal{U}^-(\mathcal{C}).
		\]
		
		Consider the nonempty set
		\[
		\Omega_0=
		\left\{
		\mathcal{C}:\mathcal{C}\text{ is an $f$-invariant circloid},
		\mathcal{C}\cap U_0\neq\emptyset,
		U_1\subseteq\mathcal{U}^-(\mathcal{C})
		\right\},
		\]
		endowed with the natural partial order $\preccurlyeq$ introduced in the preliminaries.
		
		Recall that, by Theorem~\ref{primeendsrotationthmconservative}, every $\mathcal{C}\in\Omega_0$ satisfies
		\[
		\rho_{\mathcal{C}}(F)\in[-1/n,1/n].
		\]
		
		\begin{claimsinnum}[1]
			$\Omega_0$ admits a minimal element.
		\end{claimsinnum}
		
		To prove the claim, it is enough to show that $\Omega_0$ is bounded below. Indeed, in this case the minimal element $\mathcal{C}_0$ is the lower circloid of the fill of 
		\[
		\partial \operatorname{Fill}\left(
		\bigcap_{\mathcal{C}\in\Omega_0}
		\overline{\mathcal{U}^-(\mathcal{C})}
		\right).
		\]
		Since every element of $\Omega_0$ intersects the compact disk $U_0$ and contains $U_1$ in its lower complementary component, this limiting circloid also belongs to $\Omega_0$.
		
		Assume, by contradiction, that $\Omega_0$ is not bounded below. Then, for every neighborhood $\mathcal{V}$ of the lower end of the annulus, there exists a circloid $\mathcal{C}\in\Omega_0$ such that $\mathcal{C}\cap\mathcal{V}\neq\emptyset$. By definition of $\Omega_0$, every such circloid also satisfies $U_1\subseteq\mathcal{U}^-(\mathcal{C})$. The analogue of Lemma~\ref{lematecnicocircloidmaximal} for circloids meeting neighborhoods of the lower end, with $N=3$ and adapting its conclusion to the lift $F$, therefore yields a circloid $\mathcal{C}\in\Omega_0$ such that
		\[
		\rho_{\mathcal{C}}(F)\in[\rho-1/3,\rho+1/3],
		\]
		which is incompatible with the previous estimate
		\[
		\rho_{\mathcal{C}}(F)\in[-1/n,1/n],
		\]
		proving the claim.
		
		\bigskip
		
		By definition, the minimal element $\mathcal{C}_0$ satisfies
		\[
		U_1\subseteq \mathcal{U}^-(\mathcal{C}_0).
		\]
		
		Since $\mathcal{C}_0$ cannot separate $\operatorname{Int}(U_0)$ from $U_1$, it also follows that
		\[
		\operatorname{Int}(U_0)\cap\mathcal{U}^-(\mathcal{C}_0)\neq\emptyset.
		\]

		Now consider the set
		\[
		\Omega'=
		\left\{
		\mathcal{C}\in\Omega:
		\mathcal{C}\subset\mathcal{U}^-(\mathcal{C}_0)
		\right\}.
		\]
		
		If $\Omega'$ is empty, then Proposition~\ref{propbirkhoffrelated} implies that the regular subannulus $\mathcal{U}^-(\mathcal{C}_0)$ contains two periodic points with different rotation numbers. Since it contains no $f$-invariant circloid and we are in the non-wandering setting, it is an instability region and we are done.
		
		Assume now that $\Omega'$ is nonempty.
		
		\begin{claimsinnum}[2]
			$\Omega'$ admits a maximal element.
		\end{claimsinnum}
		
		Assume, by contradiction, that it does not.
		
		Define
		\[
		V:=
		\operatorname{Fill}\left(
		\bigcap_{\mathcal{C}\in\Omega'}
		\overline{\mathcal{U}^+(\mathcal{C})}
		\right).
		\]
		
		Let $\mathcal{E}$ be the upper circloid of the annular continuum $\operatorname{Fill}(\partial V)$.
		Since $V$ is $f$-invariant, so is $\operatorname{Fill}(\partial V)$, and therefore $\mathcal{E}$ is $f$-invariant. Observe that, by construction,
		\[
		\mathcal{E}\preccurlyeq\mathcal{C}_0.
		\]
		Moreover, since $\Omega'$ has no maximal element,
		\[
		\mathcal{C}_0\cap\mathcal{E}\neq\emptyset.
		\]
		Indeed, otherwise $\mathcal{E}$ would belong to $\Omega'$ and would be its maximal element.
		
		We first establish the following properties.
		
		\begin{enumerate}
			
			\item[(i)] The circloids $\mathcal{E}$ and $\mathcal{C}_0$ have the same rotation number, which belongs to $[-1/n,1/n]$. Moreover
			\[
			U_1\cap\mathcal{E}=\emptyset.
			\]
			
			Indeed, if the circloids coincide, there is nothing to prove. Otherwise, Corollary~\ref{corolariodiscoentrecircloids} implies, after enlarging $U_1$ slightly if necessary, that the annular continuum delimited by them
			$\mathcal{A}(\mathcal{E},\mathcal{C}_0)$ has trivial rotation set. Hence,
			\[
			\rho_{\mathcal{E}}(F)=\rho_{\mathcal{C}_0}(F)\in[-1/n,1/n].
			\]
			Theorem~\ref{primeendsrotationthmconservative} then implies that $U_1\cap\mathcal{E}=\emptyset$.
			
			\item[(ii)] The disk $U_0$ intersects $\mathcal{E}$. Indeed, assume, by contradiction, that $U_0\cap\mathcal{E}=\emptyset$. Suppose that
			\[
			U_0\subset\mathcal{U}^-(\mathcal{E}).
			\]
			Since $\mathcal{C}_0\cap U_0\neq\emptyset$, this contradicts
			\[
			\mathcal{E}\preccurlyeq\mathcal{C}_0.
			\]
			
			Therefore,
			\[
			U_0\subset\mathcal{U}^+(\mathcal{E}).
			\]
			As $\mathcal{E}$ does not separate $U_0$ from $U_1$, it follows from~(i) that $U_1\subset\mathcal{U}^+(\mathcal{E})$. Hence, $U_1$ is contained in a connected component of $\mathcal{R}(\mathcal{E},\mathcal{C}_0)=\mathcal{U}^+(\mathcal{E})\cap \mathcal{U}^-(\mathcal{C}_0)$, which is an open disk since $\mathcal{E}\cap\mathcal{C}_0\neq\emptyset$. The non-wandering hypothesis then implies the existence of a periodic point in this disk component whose rotation number is forced to be $\rho$, contradicting~(i).
			
			\item[(iii)] Let $\mathcal{C}\prec\mathcal{C}_0$ be an $f$-invariant circloid intersecting $U_0$. Then
			\[
			U_1\subset\mathcal{U}^+(\mathcal{C}).
			\]
			
			Indeed, $\mathcal{C}$ cannot intersect $U_1$ by Theorem~\ref{primeendsrotationthmconservative}. If $U_1\subset\mathcal{U}^-(\mathcal{C})$, then, since $\mathcal{C}$ intersects $U_0$, we would have $\mathcal{C}\in\Omega_0$, contradicting the minimality of $\mathcal{C}_0$.
		
	\end{enumerate}
		
	We now distinguish two cases, as illustrated in Figure~\ref{fig:CasesforcircloidE}.
	
	\medskip
	
	\noindent\textbf{Case 1.} $\mathcal{E}\neq\mathcal{C}_0$.
	
	By (ii), $\mathcal{E}$ intersects $U_0$. Since $\mathcal{E}\prec\mathcal{C}_0$, (iii) implies that
	\[
	U_1\subset\mathcal{U}^+(\mathcal{E}).
	\]
	As $U_1\subset\mathcal{U}^-(\mathcal{C}_0)$, it follows that $U_1$ is contained in a connected component of $\mathcal{R}(\mathcal{E},\mathcal{C}_0)$, which is an open disk since $\mathcal{E}\cap\mathcal{C}_0\neq\emptyset$. The non-wandering hypothesis then implies the existence of a periodic point in this disk component whose rotation number is forced to be $\rho$. This contradicts~(i).
	
	\medskip
	
	\noindent\textbf{Case 2.} $\mathcal{E}=\mathcal{C}_0$.
	
	Since $\mathcal{C}_0$ intersects $U_0$, after enlarging $U_0$ slightly if necessary, for every neighborhood $W_0$ of $\mathcal{C}_0$ there exists a circloid $\mathcal{C}\in \Omega'$ intersecting $\operatorname{Int}(U_0)$ and $W_0$. By (iii),
	\[
	U_1\subset\mathcal{U}^+(\mathcal{C})
	\]
	for every such circloid. Applying Lemma~\ref{lematecnicocircloidmaximal} to the annulus $\mathcal{U}^-(\mathcal{C}_0)$, with $U_1$ and $N=3$, and adapting its conclusion to the lift $F$, we obtain
	\[
	\rho_{\mathcal{C}}(F)\in[\rho-1/3,\rho+1/3]
	\]
	for every such circloid meeting a sufficiently small neighborhood $W_0$ of $\mathcal{C}_0$. On the other hand, since these circloids intersect $\operatorname{Int}(U_0)$, Theorem~\ref{primeendsrotationthmconservative} gives
	\[
	\rho_{\mathcal{C}}(F)\in[-1/n,1/n],
	\]
	yielding the desired contradiction.

		\bigskip

	Once the existence of a maximal element of $\Omega'$ has been established, we denote this maximal element by $\mathcal{C}_1$ and set
	\[
	\mathcal{I}:=\mathcal{R}(\mathcal{C}_0,\mathcal{C}_1),
	\]
	which is a regular annulus. By definition, $\mathcal{I}$ contains no $f$-invariant circloid. Moreover, since $\mathcal{C}_1$ does not separate $U_0$ from $U_1$, the annulus $\mathcal{I}$ intersects both disks. Applying Proposition~\ref{propbirkhoffrelated}, we obtain two periodic points in $\mathcal{I}$ with different rotation numbers. Hence, $\mathcal{I}$ is an instability region.
	
	\begin{figure}[h!]
		\centering
		\includegraphics[width=340 pt]{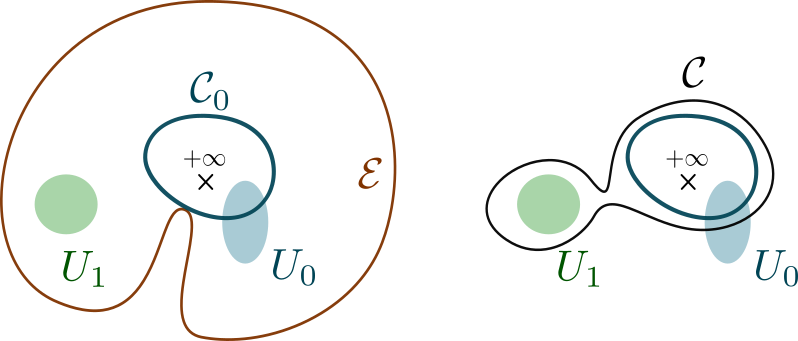}
		
		\caption{Illustration of the two cases in the proof of Claim~2. Left: the circloid $\mathcal{E}$ is different from $\mathcal{C}_0$. Right: $\mathcal{E}=\mathcal{C}_0$, and invariant circloids $\mathcal{C}\prec\mathcal{C}_0$ approach $\mathcal{C}_0$ while intersecting $U_0$.}
		\label{fig:CasesforcircloidE}
	\end{figure}
	
	\medskip
	
	In all cases, we obtain an instability region intersecting both $U_0$ and $U_1$. Finally, every instability region obtained above satisfies the curve intersection property by Remark~\ref{remarkCIPinstaregions}. Therefore, Proposition~\ref{propbirkhoffrelated} implies that every rational number in $(1/n,\rho-1/n)$ is realized by a periodic point in that instability region.
		
\end{proof}

\medskip

Let us remark that the conclusion of Theorem~\ref{maintheorem} remains valid when $U_0$ and $U_1$ intersect a regular $f$-invariant subannulus $\mathcal{A}\subset\mathbb{A}$, without necessarily being contained in it. Indeed, Proposition~\ref{propbirkhoffrelated} implies that this subannulus must contain two periodic points with different rotation number, and the proof carries over without further modifications.

This yields the following slight generalization of Theorem~\ref{maintheorem}.

\begin{thm}
	\label{thmA generalizado}
	Let $f \in \mathrm{Homeo}_{\,0,nw}(\mathbb{A})$ and $\mathcal{A}\subset\mathbb{A}$ be a regular $f$-invariant subannulus.
	Let $U_0$ and $U_1$ be an $n$-dpd, with $n\geq 3$, intersecting $\mathcal{A}$.
	Assume that each disk visits the other and that $\rho_f(U_0,U_1)=\rho\in\mathbb{N}^*$.
	Then $f$ exhibits rotational chaos occurring in an instability region contained in $\mathcal{A}$ that intersects both $U_0$ and $U_1$.
	
	Moreover, if $F$ is the lift of $f$ for which every lift of $U_0$ meets its image under $F$, then
	\[
	\big[ 1/n, \rho-1/n \big]\subseteq \rho_{\overline{\mathcal{I}}}(F)
	\]
	and every rational in the interval $(1/n,\rho-1/n)$ is realized by a periodic point in $\mathcal{I}$.
	
	Furthermore, if $\rho \geq 2$, the result remains valid for $n \geq 2$, and if $\rho \geq 3$, it holds for every $n \in \mathbb{N}$.
\end{thm}

\subsection{Proof of Theorem~\ref{maintheorembr}}

The version of the main result for area-preserving maps requires additional assumptions on the behavior of the ends of the annulus. As before, we first focus on the part of the statement concerning the existence of rotational chaos and an instability region. Notice that Proposition~\ref{propbirkhoffrelated} guarantees the existence of two periodic points with different rotation numbers. When the ends are Birkhoff-related, the entire annulus is a Birkhoff instability region, and the conclusion follows immediately from Proposition~\ref{propDFabioPatrice}. It therefore remains to consider the case where the ends are strongly non-related.

We start by recalling the statement of Theorem~\ref{maintheorembr}.

\begin{thmsinnum}{\textbf{B}}
	Let $f \in \mathrm{Homeo}_{\,0,\lambda}(\mathbb{A})$, and let $U_0, U_1$ form an $n$--dpd for some $n \geq 3$, where $U_0$ visits $U_1$, $U_1$ visits $U_0$, and $\rho_f(U_0, U_1) = \rho \in \mathbb{N}^*$. Assume that either
\begin{enumerate}
	\item the ends of $\mathbb{A}$ are Birkhoff-related, or
	\item the ends of $\mathbb{A}$ are strongly non-related.
\end{enumerate}

Then $f$ exhibits rotational chaos and the rotational horseshoe for a power of $f$ can be considered included in an instability region $\mathcal I$ meeting both $U_0$ and $U_1$. Moreover, for some lift $F$ of $f$, the rotation set of $\mathcal{I}$ satisfies
\[
\bigl[1/n,\, \rho - 1/n\bigr]\subseteq  \rho_{\overline{\mathcal{I}}}(F),
\]
and every rational number in $(1/n,\rho-1/n)$ is realized by a periodic point in $\mathcal{I}$.

Furthermore, if $\rho \geq 2$, the result remains valid for $n \geq 2$, and if $\rho \geq 3$, it holds for every $n \in \mathbb{N}$.
\end{thmsinnum}

\begin{proof}
	
	Case~(1) follows immediately from Proposition~\ref{propbirkhoffrelated}, Remark~\ref{remarkCIPinstaregions}, and Proposition~\ref{propDFabioPatrice}. Assume now that case~(2) holds. If there exists an invariant bounded regular subannulus $\mathcal{A}$ intersecting both $U_0$ and $U_1$, then the restriction $f|_{\mathcal A}$ is non-wandering and Theorem~\ref{thmA generalizado} applies. Hence, throughout the remainder of the proof, we assume that
	
	\begin{center}
		\emph{there is no $f$-invariant bounded regular subannulus intersecting both $U_0$ and $U_1$.}
	\end{center}
	
	Consider the essential sets
	\[
	\mathcal{V}^+ := \operatorname{Fill}\!\left( \bigcup_{n\in \mathbb{N}} f^n(V^+) \right),
	\qquad
	\mathcal{V}^- := \operatorname{Fill}\!\left( \bigcup_{n\in \mathbb{N}} f^n(V^-) \right).
	\]
	
	Since the ends are assumed to be strongly non-related, these sets are disjoint. Let $\mathcal{X}^+$ be the \emph{lower circloid} of the annular continuum $\operatorname{Fill}(\partial\mathcal{V}^+)$, and let $\mathcal{X}^-$ be the \emph{upper circloid} of the annular continuum $\operatorname{Fill}(\partial\mathcal{V}^-)$. With a slight abuse of notation, write
	\[
	\mathcal{U}^+:=\mathcal{U}^+(\mathcal{X}^+),\qquad
	\mathcal{U}^-:=\mathcal{U}^-(\mathcal{X}^-).
	\]
	
	\begin{claimsinnum}[1]
		The circloids $\mathcal{X}^\pm$ are $f$-invariant.
	\end{claimsinnum}
	
	To prove the claim, note that $\mathcal{V}^\pm\subseteq\mathcal{U}^\pm$. Since both $\mathcal{V}^\pm$ are forward invariant, we cannot have
	$f(\mathcal{X}^+)\preccurlyeq\mathcal{X}^+$ unless $f(\mathcal{X}^+)=\mathcal{X}^+$, as this would contradict the fact that $\mathcal{X}^+$ is the lower circloid of $\operatorname{Fill}(\partial\mathcal{V}^+)$. Analogously, we cannot have $\mathcal{X}^-\preccurlyeq f(\mathcal{X}^-)$ unless $f(\mathcal{X}^-)=\mathcal{X}^-$, since $\mathcal{X}^-$ is the upper circloid of $\operatorname{Fill}(\partial\mathcal{V}^-)$. Hence,
	\[
	f(\mathcal{X}^-)\preccurlyeq\mathcal{X}^-,
	\qquad
	\mathcal{X}^+\preccurlyeq f(\mathcal{X}^+).
	\]
	
	If $\mathcal{X}^+=\mathcal{X}^-$, the previous relations immediately imply that this circloid is $f$-invariant.
	
	Suppose now that $\mathcal{X}^+\neq\mathcal{X}^-$. Assume, by contradiction, that at least one of the circloids is not invariant. Then at least one of the previous relations cannot be an equality. Hence, by Lemma~\ref{lemmapreliminaries-regionentrecircloids} and its symmetric version for the lower circloid,
	\[
	\mathcal{R}(\mathcal{X}^-,\mathcal{X}^+)
	\subsetneq
	\mathcal{R}(f(\mathcal{X}^-),f(\mathcal{X}^+))
	=
	f\bigl(\mathcal{R}(\mathcal{X}^-,\mathcal{X}^+)\bigr).
	\]
	Since $\mathcal{R}(\mathcal{X}^-,\mathcal{X}^+)$ is open and bounded, this strict inclusion contradicts the fact that $f$ preserves area. Hence, both $\mathcal{X}^+$ and $\mathcal{X}^-$ are $f$-invariant, proving the claim.
	
	\bigskip
	
	Observe that neither $\mathcal{X}^+$ nor $\mathcal{X}^-$ can separate $U_0$ from $U_1$, and that, by Theorem~\ref{primeendsrotationthmconservative}, $U_0$ and $U_1$ cannot intersect the same invariant circloid.
	
	\begin{claimsinnum}[2]
		One of the sets $\mathcal{U}^\pm$ intersects both $U_0$ and $U_1$.
	\end{claimsinnum}
	
	To prove the claim, consider the annular continuum delimited by the circloids $\mathcal{X}^\pm$
	\[
	\mathcal{A}:=\mathcal{A}(\mathcal{X}^-,\mathcal{X}^+).
	\]
	Since $\mathcal{X}^+$ and $\mathcal{X}^-$ cannot separate $U_0$ from $U_1$, it is enough to prove that $U_0$ and $U_1$ cannot both intersect $\mathcal{A}$. Assume, by contradiction, that they do.
	
	By Lemma~\ref{lemacircloidsdisjuntosdosdiscos},
	\[
	\mathcal{X}^-\cap\mathcal{X}^+=\emptyset.
	\]
	Hence, $\mathcal{R}(\mathcal{X}^-,\mathcal{X}^+)$ is a bounded regular $f$-invariant annulus. Moreover, after enlarging $U_0$ and $U_1$ slightly if necessary, both disks intersect $\mathcal{R}(\mathcal{X}^-,\mathcal{X}^+)$. This contradicts our standing assumption that no bounded regular $f$-invariant subannulus intersects both $U_0$ and $U_1$, proving the claim.
	
	\bigskip
	
	Suppose, without loss of generality, that both $U_0$ and $U_1$ meet $\mathcal{U}^+$. Consider the set
	\[
	\Omega := \bigl\{ \mathcal{C} :\ \mathcal{C} \subset \cl[\mathcal{U}^+] \ \text{is an $f$-invariant circloid} \bigr\},
	\]
	endowed with the natural partial order $\preccurlyeq$ introduced in the preliminaries.
	
	Note that $\mathcal{X}^+\in\Omega$. Given $\mathcal{C}\in\Omega$, the circloid $\mathcal{C}$ can neither meet both $U_0$ and $U_1$ simultaneously nor separate them. Moreover, if neither disk were contained in $\mathcal{U}^+(\mathcal{C})$, then both would intersect $\mathcal{A}(\mathcal{X}^+,\mathcal{C})$. By Lemma~\ref{lemacircloidsdisjuntosdosdiscos}, the circloids $\mathcal{X}^+$ and $\mathcal{C}$ are disjoint. Hence, after enlarging $U_0$ and $U_1$ slightly if necessary, the regular annulus $\mathcal{R}(\mathcal{X}^+,\mathcal{C})$ intersects both disks, contradicting our standing assumption. Therefore, every circloid $\mathcal{C}\in\Omega$ satisfies
	\[
	U_0\subset\mathcal{U}^+(\mathcal{C})
	\qquad\text{or}\qquad
	U_1\subset\mathcal{U}^+(\mathcal{C}).
	\]

	\begin{claimsinnum}[3]
		There exists a maximal element $Y$ of $\Omega$.
	\end{claimsinnum}
	
	To prove the claim, note the following. By Lemma~\ref{lematecnicocircloidmaximal} applied to $U_0$, with $N=3$, there exists a neighborhood $\mathcal{V}_0$ of the upper end such that every $f$-invariant circloid $\mathcal{C}$ satisfying
	\[
	\mathcal{C}\cap\mathcal{V}_0\neq\emptyset
	\quad\text{and}\quad
	U_0\subset\mathcal{U}^+(\mathcal{C})
	\]
	verifies
	\[
	\rho_{\mathcal{C}}(F)\in[-1/3,1/3].
	\]
	
	By the analogue of Lemma~\ref{lematecnicocircloidmaximal} applied to $U_1$, and adapting its conclusion to the lift $F$, we obtain a neighborhood $\mathcal{V}_1$ of the upper end such that every $f$-invariant circloid $\mathcal{C}$ satisfying
	\[
	\mathcal{C}\cap\mathcal{V}_1\neq\emptyset
	\quad\text{and}\quad
	U_1\subset\mathcal{U}^+(\mathcal{C})
	\]
	verifies
	\[
	\rho_{\mathcal{C}}(F)\in[\rho-1/3,\rho+1/3].
	\]
	
	Set
	\[
	\mathcal{V}:=\mathcal{V}_0\cap\mathcal{V}_1.
	\]
	
	Suppose that $\Omega$ has no maximal element. Then every neighborhood of the upper end contained in $\mathcal{V}$ meets an element of $\Omega$. 
	Indeed, otherwise, by considering the fill of the intersection of the closures of the corresponding upper regions, and then the upper circloid of the fill of its boundary, one would obtain a maximal element of $\Omega$, as in the proof of Theorem~\ref{maintheorem}.
	In particular, there exists $\mathcal{C}\in\Omega$ such that
	\[
	\mathcal{C}\cap\mathcal{V}\neq\emptyset.
	\]
	By the previous observation, either $U_0\subset\mathcal{U}^+(\mathcal{C})$ or $U_1\subset\mathcal{U}^+(\mathcal{C})$.
	
	Assume first that
	\[
	U_0\subset\mathcal{U}^+(\mathcal{C}).
	\]
	Then
	\[
	\rho_{\mathcal{C}}(F)\in[-1/3,1/3].
	\]
	
	Since $\mathcal{C}$ does not separate $U_0$ from $U_1$, either $U_1\subset\mathcal{U}^+(\mathcal{C})$ or $\mathcal{C}\cap U_1\neq\emptyset$. In the first case, the second application of the lemma yields
	\[
	\rho_{\mathcal{C}}(F)\in[\rho-1/3,\rho+1/3],
	\]
	whereas in the second case the same estimate follows from Theorem~\ref{primeendsrotationthmconservative}, a contradiction.
	
	The case $U_1\subset\mathcal{U}^+(\mathcal{C})$ is symmetric. Hence, $\Omega$ admits a maximal element.
	
	\bigskip
	
Once the existence of a maximal element of $\Omega$ has been established, we denote it by $Y$ and define
\[
A:=\mathcal{U}^+(Y).
\]
Then both $U_0$ and $U_1$ intersect $A$. We claim that this reduces the proof to case~(1), already established. To prove this, it remains to show that the ends of $A$, namely $Y$ and $+\infty$, are Birkhoff-related. Since $A$ intersects both $U_0$ and $U_1$, Proposition~\ref{propbirkhoffrelated} will then imply that $A\simeq\mathbb{A}$ is a Birkhoff instability region. Observe that $A$ contains no $f$-invariant circloid. Although this characterization of Birkhoff instability regions was established for non-wandering homeomorphisms, it remains valid in the present setting. Let us verify this.

Consider a neighborhood $W\subset A$ of the lower end of $A$, identified with $Y=\partial A$. To prove that the ends of $A$ are Birkhoff-related, it is enough to show that both the forward and backward orbits of $W$ accumulate on $+\infty$, or equivalently, that the following sets are unbounded:
\[
\mathcal{W}^+:=\operatorname{Fill}\left(\bigcup_{n\in\mathbb{N}}f^n(W)\right),
\qquad
\mathcal{W}^-:=\operatorname{Fill}\left(\bigcup_{n\in\mathbb{N}}f^{-n}(W)\right).
\]

Assume, by contradiction, that $\mathcal{W}^+$ is bounded. Let $\mathcal{T}^+$ be the lower circloid of the annular continuum $\operatorname{Fill}(\partial\mathcal{W}^+)$. Since $\mathcal{W}^+$ is forward invariant and $\mathcal{T}^+$ was chosen as the lower circloid, $\mathcal{T}^+$ cannot be mapped strictly below itself. Hence,
\[
\mathcal{T}^+\preccurlyeq f(\mathcal{T}^+).
\]
If this relation is not an equality, Lemma~\ref{lemmapreliminaries-regionentrecircloids} would imply that
\[
\mathcal{R}(Y,\mathcal{T}^+)
\subsetneq
\mathcal{R}(Y,f(\mathcal{T}^+))
=
f\bigl(\mathcal{R}(Y,\mathcal{T}^+)\bigr),
\]
where the equality follows from the $f$-invariance of $Y$. Since $\mathcal{R}(Y,\mathcal{T}^+)$ is open and bounded, this contradicts the fact that $f$ preserves area. Therefore,
\[
f(\mathcal{T}^+)=\mathcal{T}^+.
\]
Since $\mathcal{T}^+\subset A=\mathcal{U}^+(Y)$, we have $Y\preccurlyeq\mathcal{T}^+$ and $Y\neq \mathcal{T}^+$. Thus, $\mathcal{T}^+\in\Omega$, contradicting the maximality of $Y$ in $\Omega$.

An analogous argument applied to $\mathcal{W}^-$ and $f^{-1}$ shows that $\mathcal{W}^-$ is also unbounded.

Finally, the estimate for the interval contained in the rotation set and the realization of every rational rotation number by periodic points follow directly from Proposition~\ref{propbirkhoffrelated}.
	
\end{proof}

\medskip

\section{Technical results}
\label{sec.toplem}

\subsection{Annular continuum delimited by circloids}

\begin{lemmaasterisco}[\ref{lemmapreliminaries-regionentrecircloids}]
		Consider two different circloids $\mathcal{C}^-$ and $\mathcal{C}^+$ with $\mathcal{C}^-\preccurlyeq\mathcal{C}^+$. Let $\mathcal{E}$ be a circloid such that $\mathcal{C}^+\preccurlyeq\mathcal{E}$ and $\mathcal{C}^+\neq\mathcal{E}$. Then
		\[
		\mathcal{R}(\mathcal{C}^-,\mathcal{C}^+)
		\subsetneq
		\mathcal{R}(\mathcal{C}^-,\mathcal{E}).
		\]	
\end{lemmaasterisco}

\begin{proof}
	The relation $\mathcal{C}^+\preccurlyeq\mathcal{E}$ implies
	\[
	\mathcal{U}^-(\mathcal{C}^+)\subseteq\mathcal{U}^-(\mathcal{E}),
	\]
	and hence
	\[
	\mathcal{R}(\mathcal{C}^-,\mathcal{C}^+)
	\subseteq
	\mathcal{R}(\mathcal{C}^-,\mathcal{E}).
	\]

Since $\mathcal{C}^+$ and $\mathcal{E}$ are different, the region $\mathcal{R}(\mathcal{C}^+,\mathcal{E})$ is nonempty, as commented in the Preliminaries. Moreover,
\[
\mathcal{R}(\mathcal{C}^+,\mathcal{E})
\subseteq
\mathcal{R}(\mathcal{C}^-,\mathcal{E}),
\]
while
\[
\mathcal{R}(\mathcal{C}^+,\mathcal{E})
\cap
\mathcal{R}(\mathcal{C}^-,\mathcal{C}^+)
=
\emptyset.
\]
Therefore, the inclusion is strict.

\end{proof}

\subsection{Existence of an ends connector}

We call an \emph{ends--connector} any connected set accumulating on both ends of the annulus $\mathbb{A}$ whose complement consists of an open topological disk admitting an arc that also connects one end to the other (referred to throughout this article as a \emph{vertical line}). As in Section~\ref{seccionRotationsets}, we denote by $\mathcal{W}$ the union of the maximal cross-sections determined by $U\cap A$, where $A$ is a subannulus of $\mathbb{A}$.

\medskip

This section is devoted to the proof of Proposition~\ref{lematecnicoseparacionpancitas}, which asserts the existence of an ends--connector that does not separate the maximal cross-sections in $A$. It is worth noting that, in simple examples, such an ends--connector can be chosen to be a curve. However, this is not always possible, as illustrated in Figure~\ref{fig:endsconnector}.

\begin{figure}[b]
	\centering
	\includegraphics[width=180 pt]{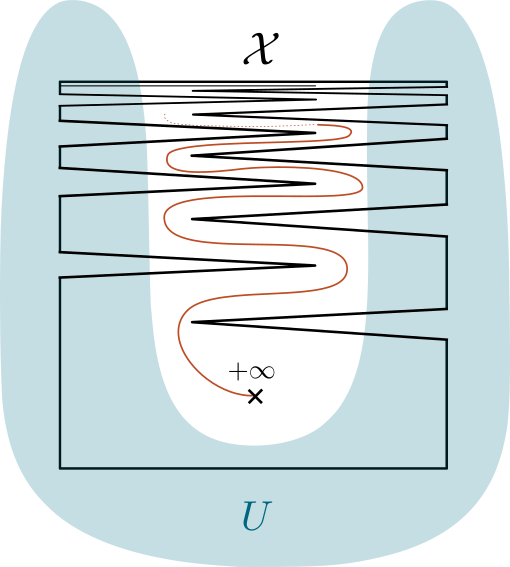}
	
	\caption{Observe that any attempt to construct a curve from $+\infty$ to $-\infty$ that does not separate the maximal crosscuts and avoids intersecting $U$ (represented in blue) must necessarily accumulate on a segment of $\mathcal{X}$.}
	\label{fig:endsconnector}
\end{figure}

\medskip

The construction of such an ends--connector begins with the following lemma, which is almost the desired statement of Proposition~\ref{lematecnicoseparacionpancitas}.

\begin{lemma}
	\label{lemma ends connector}
For a sub-annulus $A\subset \mathbb{A}$ bounded below, $\mathcal{X}=\partial A$ and a closed disk $U$ whose interior intersects $\mathcal{X}$, there exists an ends-connector $K$ such that
\begin{itemize}
    \item $K$ does not meet the interior of $U$
    \item $K\cap (\mathbb{A}\setminus A)$ does not meet $U$
    \item Every pair of maximal crosscuts can be connected by an arc in $\left( A \setminus \mathcal{W}\right) \setminus K$.
    \item $K\cap A$ is connected
\end{itemize}
\end{lemma}

We seek an ends-connector disjoint from $U$, whereas the one provided by the previous lemma may intersect $\partial U$. It is not difficult to adjust $K$ to obtain the desired ends-connector. We recall below the statement of Proposition~\ref{lematecnicoseparacionpancitas}, followed by its proof assuming Lemma~\ref{lemma ends connector}.

\begin{propasterisco}[\ref{lematecnicoseparacionpancitas}]
	For a sub-annulus $A$ bounded below and a closed disk $U$  whose interior meets $\mathcal{X}=\partial A$, there exists an ends-connector set $K$ satisfying the following properties
	\begin{itemize}
		\item is disjoint from $U$
		\item every pair of maximal crosscuts can be connected by an arc in $\left(A \setminus \mathcal{W} \right)\setminus K $
		\item $K\cap A$ is connected
	\end{itemize}
\end{propasterisco}

\begin{proof}
	Consider a slightly larger disk $U'$ containing $U$ which defines a set of maximal cross-sections whose union is denoted by $\mathcal{W}'$. Since maximal crosscuts are characterized as those accessible by a ray from $+\infty$, it follows that every maximal cross-section of $\partial U \cap A$ is contained in a maximal cross-section of $\partial U' \cap A$. In particular $\mathcal{W}\subsetneq \mathcal{W}'$.
	
	Applying Lemma~\ref{lemma ends connector} for $U'$ instead of $U$, we obtain an ends-connector $K$ as in the previous lemma, such that if $K^+$ denotes the connected component of $K \cap \cl[(A \setminus \mathcal{W}')]$ meeting $+\infty$, then $K^+ \cap \mathcal{W} = \emptyset$. Notice that, as $K$ cannot meet the interior of $\mathcal{W}'$, any crosscut of $\partial U\cap A$ can be connected with a crosscut of $\partial U'\cap A$ by an arc contained in $\mathcal{W}' \setminus \mathcal{W}$ without meeting $K$. This shows that any pair of maximal crosscuts of $\partial U\cap A$ can be connected by an arc in $(A\setminus \mathcal{W})\setminus K$. Finally, since $U\subseteq \operatorname{Int}(U')$, we have
	\[
	K \cap U = \emptyset.
	\]
	This yields the desired ends-connector.
\end{proof}

\medskip

Now, the proof of Lemma~\ref{lemma ends connector}

\begin{proof}
	We consider the case in which there is more than one maximal crosscut; otherwise, the result is trivial. Our goal is to escape from $A \setminus \mathcal{W}$ by means of a continuum $K$ that avoids ``covering''	any crosscut, in the sense that all crosscuts are contained in the same connected component of $A \setminus K$.	The main idea is therefore to identify a suitable escape region.

	For an arc $C$ we denote by $\operatorname{int}(C)$ the arc without its endpoints. For each crosscut $C_i$, consider a ray $r_i$ from $+\infty$ landing at a point $c_i \in \operatorname{int}(C_i)$, such that $\operatorname{int}(r_i) \subset A \setminus \mathcal{W}$. We may choose these rays to be pairwise disjoint, with indices $i$ belonging to a countable subset $I \subset \mathbb{S}^1$ inheriting a cyclic order.
	
	Take two different indices $i_1, i_2 \in I$, and denote by $r_1, r_2$ the corresponding rays. Observe that the rays $r_1$ and $r_2$ divide $A \setminus \mathcal{W}$ into two regions.	The following claim states that only one of these regions is an ``escape'' region.
	
	\begin{claimsinnum}[1]
		One of the regions in $A \setminus \mathcal{W}$ bounded by $r_1,r_2$ has the property that every curve connecting a point $x$ in its interior with $-\infty$ without meeting $U$ must intersect at least one of the rays $r_1,r_2$. We call this region $D_{1,2}$. On the other hand, if $x \in (A \setminus \mathcal{W}) \setminus D_{1,2}$, it is possible to connect $x$ to $-\infty$ without meeting $U \cup D_{1,2}$.
	\end{claimsinnum}
	
	By adding $+\infty$, the annulus $\mathbb{A}$ becomes a disk. Setting $\hat{A} = A \cup \{+\infty\}$, we observe that $\hat{A} \setminus \mathcal{W}$ is an open disk.	To prove the claim, note that the endpoints $c_1,c_2$ of the rays $r_1,r_2$ can be connected by one	of the two subarcs of $\partial U$ between them in such a way that the disk bounded by $r_1$, $r_2$, and this subarc, contains no points of $\operatorname{Int}(U)$ in its interior. Let $B_{1,2}$ denote this disk. On the other hand, 
	\[
	\left( \hat{A} \setminus \mathcal{W}\right)\setminus (r_1\cup r_2)
	\]
	 has two connected components, which are topological disks and only one meets $B_{1,2}$. Define $D_{1,2}$ to be the disk component of  $\left( \hat{A} \setminus \mathcal{W}\right)\setminus (r_1\cup r_2)$ intersecting the $B_{1,2}$.  
	
	Now let $x \in D_{1,2}$ and consider a curve joining $x$ to $-\infty$ that avoids $U$ (recall that $U$ is inessential).	By extending the curve inside $D_{1,2}$ if necessary, we may assume that $x$ lies arbitrarily close	to $+\infty$. The curve must then leave $B_{1,2}$ without meeting $U$,	which forces it to exit through either $r_1$ or $r_2$.
	
	For the final assertion of the claim, let
	
	$$
	x \in (A \setminus \mathcal{W}) \setminus D_{1,2}.
	$$
	
	Note that $(A \setminus \mathcal{W}) \setminus D_{1,2}$ must be a disk.
	Since $U$ is inessential, there exists a ray from $-\infty$ to $x$ that is disjoint from $U$. If this ray does not meet $D_{1,2}$, there is nothing to prove. Otherwise, it enters $B_{1,2}$ and,	since it avoids $U$, it must cross the boundary of $B_{1,2}$ through the interior of either $r_1$	or $r_2$. Consequently, the ray intersects $(A \setminus \mathcal{W}) \setminus D_{1,2}$ before	crossing $r_1$ or $r_2$. Cutting the ray at such an intersection point and joining it to $x$ by an arc	contained in $(A \setminus \mathcal{W}) \setminus D_{1,2}$ completes the construction.	In particular, the connected component of this ray inside $(A \setminus \mathcal{W}) \setminus D_{1,2}$ connects $x$ to $\mathcal{X}$.

	 \begin{figure}[h!]
		
		\centering
		\includegraphics[width=250 pt]{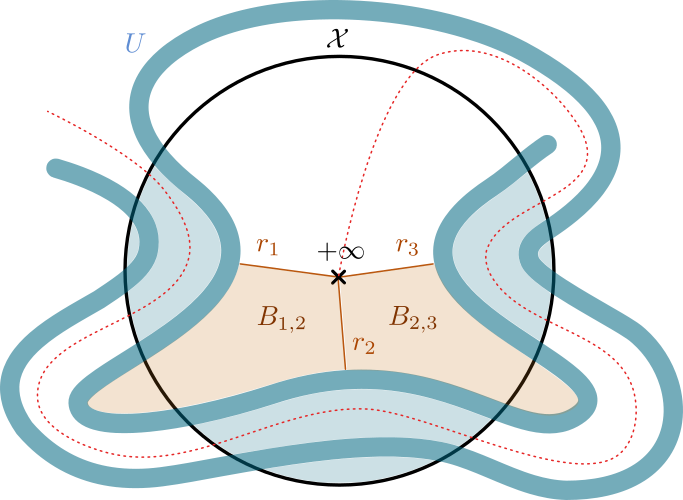}
		
		\caption{A simple case in which the escape region is the complement of two disks in $A\setminus \mathcal{W}$, and an ends-connector can be chosen as a curve, a dashed curve in the figure.}
		
		\label{escaperegion}
	\end{figure}

	\begin{claimsinnum}[2]
		There exists a vertical line $\gamma$ such that it does not intersect $\mathcal{X} \setminus (U \cup D_{1,2})$.
	\end{claimsinnum}
	
	Take $\gamma_1$ as a ray from $+\infty$ inside $D_{1,2}$ landing at an accessible point $a \in \partial U$ different from $c_1,c_2$. Further, take $\gamma_3$ as a ray from $-\infty$ landing at $b \in \partial U$ without meeting $\mathcal{X}$, and let $\gamma_2$ be an arc in $U$ joining $a$ and $b$, where the interior of the arc $\gamma_2$ is contained in $\operatorname{Int}(U)$. Concatenating $\gamma_1, \gamma_2, \gamma_3$ gives the desired vertical line of the claim.
	
	\medskip
	
	Having defined $D_{1,2}$ for the pair $i_1,i_2$, we similarly define $D_{i,j}$ for each pair of rays $r_i,r_j$, with $i,j \in I$. Observe that if $r_3$ is another ray, then one of $D_{1,2}$, $D_{1,3}$, or $D_{2,3}$ contains the other two. Thus
	\[
	\mathcal{D} := \bigcup_{i,j \in I} D_{i,j}
	\]
	can be written as an increasing sequence of open disks. Moreover, its complement
	\[
	K^+:=(\hat{A} \setminus \operatorname{Int}(\mathcal{W})) \setminus \mathcal{D}
	\]
	is the intersection of a decreasing sequence of closed disks in $\hat{A}$, each connecting $\partial \mathcal{X}$ to $+\infty$.

	\begin{claimsinnum}[3]
		The set $K^+$ is inessential.
	\end{claimsinnum}
	
	By definition, $K^+$ meets neither $\operatorname{Int}(U)$ nor $D_{1,2}$. As $K^+ \subset \cl[A]$, it cannot meet the vertical $\gamma$ from the previous claim. Therefore, $K^+$ is inessential.
	
	\medskip
	
	Take a vertical line $v \colon (-\infty,+\infty) \to \mathbb{A} \setminus (U \cup \gamma)$, i.e., a curve joining one end to the other. If $r_i, r_j$ are two rays and $D_{i,j}$ is the associated disk, then whenever $v$ meets $D_{i,j}$ it must first cross one of the rays $r_i$ or $r_j$, which implies that it also meets $(A \setminus \mathcal{W}) \setminus D_{i,j}$. Hence, the vertical $v$ meets every escape region. Therefore, for each $i,j$ there exists $t_{i,j}$ such that $v(t_{i,j}) \in (A \setminus \mathcal{W}) \setminus D_{i,j}$ and the restriction $v|_{(-\infty,t_{i,j})}$ does not intersect $D_{i,j}$.
	
	Writing $\mathcal{D}$ as an increasing sequence of disks $(D_n)_{n \in \mathbb{N}}$, choose $t_n$ such that $t_{n+1} \leq t_n$, $v(t_n) \in (A \setminus \mathcal{W}) \setminus D_n$, and
	\[
	v|_{(-\infty,t_n)} \ \text{does not intersect } D_n.
	\]
	
	If $t_\infty = \lim_{n \to \infty} t_n$, then $v(t_\infty) \in K^+$ and $v|_{(-\infty,t_\infty)}$ avoids all $D_n$. By truncating $v$ if necessary, we may assume that it meets $K^+$ for the first time at $v(t_\infty)$. Define
	\[
	K = K^+ \cup v|_{(-\infty,t_\infty]}.
	\]
	
 	This set is connected, inessential (since it is disjoint from $\gamma$), and accumulates on $-\infty$ and $+\infty$. Moreover, since $K^+$ is obtained as the intersection of a decreasing sequence of closed disks, its complement is an increasing nested union of open disks and is therefore simply connected. Attaching the ray $v|_{(-\infty,t_\infty]}$, which meets $K^+$ only at its endpoint, does not create any new complementary component. Hence, the complement of $K$ is connected and simply connected. Since the complement also contains the vertical line $\gamma$, it is an open disk containing a vertical line. Therefore, $K$ is an ends-connector.
	
	Moreover, $K \cap A$ does not separate maximal crosscuts in $A \setminus \mathcal{W}$. Indeed, if $C_1, C_2$ are two different maximal crosscuts, there exist rays $r_1, r_2$ landing on them and defining a disk $D_{1,2}$. The set $K$ is disjoint from this disk, and within it one can join $C_1$ and $C_2$ by an arc, which completes the proof.
\end{proof}

\subsection{Disjoint pair of disks and non-trivial rotation set}

Let us first recall the following result by Franks, which can be found as Proposition 1.3 in \cite{Franks(PoincBirkh)}.

\begin{prop}[\cite{Franks(PoincBirkh)}]
	\label{proposicionFranks}
	Let $H:\mathbb{R}^2 \to \mathbb{R}^2$ be an orientation-preserving homeomorphism that possesses a periodic disk chain. Then there exists a simple closed curve $\gamma$ such that the index of $\gamma$ with respect to $H$ is $1$. In particular, $H$ has a fixed point.
\end{prop}

Recall that for a homeomorphism $H$ of a surface, a \textit{chain of disks} is a finite collection of pairwise disjoint embedded open disks $V_1,\dots,V_n$ satisfying:
\begin{itemize}
	\item For each $i=1,\dots,n$, $H(V_i)\cap V_i = \emptyset$;
	\item For each $i=1,\dots,n-1$, there exists $m_i\in \mathbb{N}$ such that $H^{m_i}(V_i)\cap V_{i+1}\neq \emptyset$.
\end{itemize}

A \textit{periodic chain of disks} is a chain of disks in which the last disk eventually meets the first one; explicitly, there exists $m_n$ such that
\[
H^{m_n}(V_n)\cap V_1 \neq \emptyset.
\]

The positive rotational difference, the $3$-dpd condition, and the fact that one disk visits the other allow us, for two suitable lifts of $f^3$, to find a periodic chain in the plane.

Also, we need the following proposition due to Le Calvez, which summarizes Theorem~9.1 and Corollary~9.2 of the article~\cite{LeCalvez2005IHES}. At this point, it is important to recall that for a lift $F$ of a map $f \in \text{Homeo}_{\,0}(\mathbb{A})$ although for a given lift $F$ of $f$ the limit
\[
\lim_{n\to+\infty}\frac{\mathrm{pr}_1(F^n(x)-x)}{n}
\]
may exist for a point $x\in\mathbb{R}^2$, the existence of this limit is not preserved under conjugacy. The difficulty arises when the orbit of $\pi(x)$ is unbounded. Nevertheless, if $\pi(x)$ is positively recurrent, its rotation number can always be defined by considering subsequences of the orbit returning arbitrarily close to $\pi(x)$. More generally, if $\pi(x)$ is positively non-escaping---that is, its $\omega$-limit set is not reduced to one of the ends---then its rotation number is well defined, and its existence is preserved under conjugacy. For further details, we refer the reader to \cite{LeCalvezInfiniteannulus,fabiopatricehorseshoes}.

We also say that $f \in \text{Homeo}_{\,0}(\mathbb{A})$ satisfies the \emph{curve intersection property} if every essential simple curve $\gamma \subset \mathbb{A}$ satisfies $f(\gamma)\cap \gamma \neq \emptyset$.

\begin{prop}[\cite{LeCalvez2005IHES}]
	\label{realizacionPatrice}
	Let $f\in\text{Homeo}_{\,0}(\mathbb{A})$, let $F$ be a lift of $f$, and let $z_0,z_1$ be positively recurrent points with $\rho_i:=\rho_{z_i}(F)$ for $i=0,1$. Assume that $\rho_0<\rho_1$. If some rational number in $(\rho_0,\rho_1)$ is not realized by a periodic point of $f$, then there exists an essential simple closed curve $\gamma$ separating $z_0$ from $z_1$ such that $f(\gamma)\cap\gamma=\emptyset$. Consequently, if $f$ satisfies the curve intersection property, then every rational number in $(\rho_0,\rho_1)$ is realized by a periodic point.
\end{prop}

As observed in Remark~\ref{remarkCIPinstaregions}, any $f \in \text{Homeo}_{\,0}(\mathbb{A})$ restricted to a Birkhoff instability region $A$ satisfies the curve intersection property in $A$. As an immediate consequence of the proposition above, although this particular result will not be used in the article, we can isolate the following property.

\begin{cor}	
 Let $f \in \text{Homeo}_{\,0}(\mathbb{A})$, and let $A$ be an open invariant annulus whose ends are Birkhoff-related. If there exist two periodic points $z_0,z_1\in A$ with different rotation numbers $\rho_0$ and $\rho_1$, then every rational number between $\rho_0$ and $\rho_1$ is realized by a periodic point in $A$.
\end{cor}

It is possible to generalize Proposition~\ref{realizacionPatrice} by allowing the prime-ends rotation number of an invariant circloid to play the role of the rotation number of one of the recurrent points $z_0,z_1$.

\begin{lemma}
	\label{lemmacorolariorealizacionPatrice}
	Let $f\in\mathrm{Homeo}_{\,0}(\mathbb{A})$, let $F$ be a lift of $f$, let $z_0$ be a positively recurrent point, and let $C_1$ be an $f$-invariant circloid. Set $\rho_0:=\rho_{z_0}(F)$ and $\rho_1:=\rho^+_{C_1}(F)$, where $\rho^+_{C_1}(F)$ denotes the prime-ends rotation number associated to $\mathcal{U}^+(C_1)$,
	and assume that $\rho_0<\rho_1$ and that $z_0\in\mathcal{U}^+(C_1)$.
	
	If some rational number in $(\rho_0,\rho_1)$ is not realized by a periodic point in $\mathcal{U}^+(C_1)$, then there exists an essential simple closed curve $\gamma\subset\mathcal{U}^+(C_1)$
	separating $z_0$ from $C_1$ such that $f(\gamma)\cap\gamma=\emptyset.$
	In particular, if $f|_{\mathcal{U}^+(C_1)}$ satisfies the curve intersection property, then every rational number in $(\rho_0,\rho_1)$ is realized by a periodic point in $\mathcal{U}^+(C_1)$.
\end{lemma}

\begin{proof}
Consider the prime-ends compactification of $\mathcal{U}^+(C_1)$, where $\mathcal{U}^+(C_1)$ is identified via a homeomorphism with the interior of
\[
A=S^1\times[0,+\infty).
\]

As explained in the Preliminaries, under this identification the map $f$ induces a homeomorphism which extends to $A$ and is conjugate to the original one in $S^1\times(0,+\infty)$.

Once a lift $F$ of $f$ is fixed, the dynamics on the boundary $S^1\times\{0\}$, together with a suitable choice of the lift of the corresponding map on $A$, gives rise to the definition of the prime-ends rotation number $\rho_{C_1}^+(F)$. With an abuse of notation, here we also write $f$ for the corresponding map on $A$, since, up to this identification, it satisfies the same hypotheses as $f$ in the interior of $\mathcal{U}^+(C_1)$.

In particular, note that under this identification neighborhoods of the end of $\mathcal{U}^+(C_1)$ corresponding to $C_1$ correspond to neighborhoods of $S^1\times\{0\}$. For further details, see the Preliminaries or the proof of Proposition~\ref{primeendsrotationproposition}.

If, for the point corresponding to $z_0$ under this identification---which we still denote by $z_0$---we obtain a simple closed curve $\gamma\subset S^1\times(0,+\infty)$ separating $z_0$ from $S^1\times\{0\}$, then, under the same identification, we obtain the conclusion of the lemma.

	Extend $f$ to a homeomorphism
	\[
	g:S^1\times\mathbb{R}\longrightarrow S^1\times\mathbb{R}
	\]
	defined by
	\[
	\begin{array}{ll}
		g(x,y)=f(x,y), & \text{if } y\geq0,\\[2pt]
		g(x,y)=(f(x,0),y), & \text{if } y<0.
	\end{array}
	\]
	
	Let $G$ be the lift of $g$ that coincides with $F$ on $\mathbb{R}\times[0,+\infty)$. If $z_1$ is a recurrent point in $S^1\times\{0\}$, then
	\[
	\rho_{z_0}(G)=\rho_0, \qquad \rho_{z_1}(G)=\rho_1
	\]
	
	Observe that every circle $S^1\times\{r\}$ with $r<0$ is $g$-invariant. Moreover, every recurrent point in $S^1\times(-\infty,0]$ has rotation number exactly $\rho_1$. Consequently, any periodic point realizing a rotation number in $(\rho_0,\rho_1)$ must belong to
	\[
	S^1\times(0,+\infty)\subset A.
	\]
	Assume that some rational number in $(\rho_0,\rho_1)$ is not realized by a periodic point in $A=S^1\times(0,+\infty)$. As observed above, it is therefore not realized by any periodic point in $S^1\times\mathbb{R}$ either. By Proposition~\ref{realizacionPatrice}, there exists an essential simple closed curve $\gamma$ satisfying
	\[
	g(\gamma)\cap\gamma=\emptyset
	\]
	and separating $z_0$ from $z_1$.
	
	It is therefore enough to prove that $\gamma\subset A$
	since this implies that $\gamma$ separates $z_0$ from $C_1$.
	Suppose, by contradiction, that
	\[
	\gamma\cap\bigl(S^1\times(-\infty,0]\bigr)\neq\emptyset.
	\]
	
	Choose
	\[
	p\in\gamma\cap\bigl(S^1\times(-\infty,0]\bigr)
	\]
	with minimal second coordinate. Since $g$ preserves the second coordinate on $S^1\times(-\infty,0]$, the point $g(p)$ also has minimal second coordinate on $g(\gamma)$. It follows that both $p$ and $g(p)$ are accessible from $-\infty$ in
	\[
	\bigl(S^1\times(-\infty,0]\bigr)\setminus\bigl(\gamma\cup g(\gamma)\bigr).
	\]
	
	In particular, $\gamma$ does not separate $g(\gamma)$ from $-\infty$, and $g(\gamma)$ does not separate $\gamma$ from $-\infty$. However, since $\gamma$ and $g(\gamma)$ are disjoint essential simple closed curves, either
	$g(\gamma)\subset\mathcal{U}^-(\gamma)$ or $g(\gamma)\subset\mathcal{U}^+(\gamma)$. In the first case, $\gamma$ separates $g(\gamma)$ from $-\infty$, contradicting the first observation. In the second case, $g(\gamma)$ separates $\gamma$ from $-\infty$, contradicting the second one. Thus, both possibilities lead to a contradiction.

\end{proof}

We obtain, as a corollary of Proposition~\ref{realizacionPatrice} and Lemma~\ref{lemmacorolariorealizacionPatrice}, a realization result for non-wandering or area-preserving maps restricted to regular annuli under a sort of twist condition. Recall that, by Proposition~\ref{propcircloidsrotaciontrivial}, any $f$-invariant circloid in this setting has a trivial rotation set, which coincides with its prime-ends rotation number.

\begin{cor}
	\label{corolariorealizacionPatrice}
	Consider $f \in \text{Homeo}_{\,0,nw}(\mathbb{A}) \cup \text{Homeo}_{\,0,\lambda}(\mathbb{A})$ and $A$ a regular $f$-invariant annulus. Assume that one of the following conditions holds:
	\begin{enumerate}
		\item $A$ contains two positively recurrent points $z_0$ and $z_1$ such that $\rho_0:=\rho_{z_0}(F)\neq \rho_{z_1}(F)=:\rho_1$ and $f|_A$ satisfies the curve intersection property;
		
		\item $A$ is bounded by two circloids $C_0$ and $C_1$ such that $\rho_0:=\rho_{C_0}(F)\neq \rho_{C_1}(F)=:\rho_1$;
		
		\item $A$ contains a positively recurrent point $z_0$ and has lower boundary $C_1$, where $C_1$ is a circloid such that $\rho_0:=\rho_{z_0}(F)\neq \rho_{C_1}(F)=:\rho_1$.
	\end{enumerate}
	Then every rational number between $\rho_0$ and $\rho_1$ is realized by a periodic point in $A$.
\end{cor}

\begin{proof}
	
	In the non-wandering case, $f$ automatically satisfies the curve intersection property. Moreover, all three cases are classical and can be proved using the strategy of Franks, constructing a periodic chain of disks for a suitable lift. For further details, we refer the reader to Theorem~2.1 of \cite{Franks(PoincBirkh)}.
	
	Assume now that $f\in\text{Homeo}_{\,0,\lambda}(\mathbb{A})$. Under the first assumption, the result follows directly from Proposition~\ref{realizacionPatrice}. Under the second assumption, the restriction of $f$ to $A$, the subannulus bounded by the two circloids, is non-wandering, so the same argument based on Franks' strategy applies. Finally, under the third assumption, we may identify $A$ with $\mathcal{U}^+(C_1)$. Since $f$ is area-preserving and $A$ is bounded below, $f$ satisfies the curve intersection property in $A$. Indeed, since $C_1$ is invariant, if $f(\gamma)\cap\gamma=\emptyset$, then $f(\gamma)$ lies strictly above or below $\gamma$. Hence $\mathcal{R}(C_1,f(\gamma))=f(\mathcal{R}(C_1,\gamma))$ would have strictly larger or smaller area than $\mathcal{R}(C_1,\gamma)$, contradicting the fact that $f$ preserves area. Therefore, the result follows directly from Lemma~\ref{lemmacorolariorealizacionPatrice}.
	
\end{proof}

We are ready to prove the main proposition of this subsection.

\begin{propasterisco}[\ref{propbirkhoffrelated}]
Consider $f \in \text{Homeo}_{\,0,nw}(\mathbb{A}) \cup \text{Homeo}_{\,0,\lambda}(\mathbb{A})$. Let $\overline{A} \subset \mathbb{A}$ be the closure of a regular $f$--invariant subannulus $A$, and let $U_0, U_1$ be an $n$--dpd with $n \geq 3$ and $\rho_f(U_0,U_1) = \rho \in \mathbb{N}^*$, such that both $U_0$ and $U_1$ intersect $A$.
Assume that $U_0$ visits $U_1$ and $U_1$ visits $U_0$. If $F$ is the lift of $f$ for which $F(\tilde{U}_0)\cap \tilde{U}_0 \neq \emptyset$, then
\[
[1/n,\, \rho - 1/n] \subseteq \rho_{A}(F),
\]
and $A$ contains at least two periodic points $z_1$ and $z_2$ with different rotation numbers. Furthermore
\begin{itemize}
	\item If $A$ is bounded below or above or,
	\item if $A$ satisfies the curve intersection property
\end{itemize}
then every rational in $(1/n,\rho-1/n)$ is realized by a periodic point in $A$.
\medskip

Moreover, if $U_0$ and $U_1$ do not meet $\partial A$, the existence of such periodic points does not require the non-wandering or area-preserving assumption, nor the regularity of $A$. More precisely, for any $f\in\mathrm{Homeo}_{\,0}(\mathbb{A})$ and any $f$-invariant open subannulus $A$ containing $U_0$ and $U_1$, we can ensure the existence of periodic points $z_1,z_2\in A$ realizing the rotation numbers $1/n$ and $\rho-1/n$, respectively, for the lift $F$.
\medskip

If $\rho \geq 2$, the result remains valid for $n \geq 2$, and if $\rho \geq 3$, it is valid for $n \geq 1$.

\end{propasterisco}

\begin{proof}
Let us first deal with the case where $\rho=1$ and the disks form a $3$-dpd. The remaining cases follow by analogous arguments, and the details are given at the end of the proof. We denote by $A$ the interior of $\overline{A}$.
	
If $U_0$ and $U_1$ meet the boundary $\partial A$, then by Proposition~\ref{primeendsrotationthmconservative} they must intersect two different boundary components of $\partial A$, each corresponding to a different rotation. In particular, $A$ is bounded, $f$ restricted to $A$ is non-wandering, and the rotational constraints on the circloids allow us to ensure that the rotation set $\rho_F(\overline{A})$ contains the interval $[1/3,\,2/3]$. The conclusions in this case are a direct application of Corollary~\ref{corolariorealizacionPatrice}, where the circloids delimit a subannulus with a twist condition.
	
	Then, from now on, we consider the other case where one of the two disks is strictly contained in $A$; without loss of generality, let us assume that $U_0 \subseteq A$.
	
	Note that the definition of $n-dpd$ for $n\geq 2$ allows that $U_0 \cap f^2(U_0) = \emptyset$. In order to construct a periodic chain of disks, it will be useful to ensure a nontrivial intersection between $U_0$ and its first three iterates under $f$. To this end, we perform a small perturbation to create a fixed point in $U_0$. The next claim states that this perturbation does not create any new periodic points with rotation number $1/3$, so we may find such a periodic point for the perturbed map instead of the original one.
	
	\begin{claimsinnum}[1]
		Assume that, for a small perturbation $g$ of $f$ supported inside $U_0$, there exists a periodic point $z_1$ such that $\rho_G(z_1)=1/3$, where $G$ is the lift of $g$ satisfying $G(\tilde{U}_0)\cap \tilde{U}_0 \neq \emptyset$. Then $z_1$ must lie outside $U_0$, and hence it is also a periodic point for $f$ with $\rho_F(z_1)=1/3$.
		
		An analogous statement holds for $U_1$ in the case $U_1 \subset A$, yielding a periodic point $z_2$ with rotation $2/3$.
	\end{claimsinnum}
	
Take $p\in U_0$, let $\tilde{U}_0$ be a lift of $U_0$, and $\tilde{p}\in \tilde{U}_0$ a lift of $p$. Then $G^3(\tilde{p})\in \bigcup_{i=1}^3 G^i\left(\tilde{U}_0\right)$. If $\tilde{z}_1$ is a lift of $z_1$, we have $G^3(\tilde{z}_1)=\tilde{z}_1+(1,0)\in \bigcup_{i=1}^3 G^i\left(\tilde{U}_0\right)+(1,0)$. Since these two sets are disjoint by the $3$-dpd property, $z_1$ cannot belong to $U_0$, proving the claim.
		
	\bigskip
		
	Then, as mentioned above, for the proof of this proposition we will consider a perturbation $g$ of $f$ supported in $U_0$ --- or in $U_0 \cup U_1$ when $U_1 \subset A$ --- and find fixed points in the plane for the maps $G^3 - (1,0)$ and $G^3 - (n\rho-1,0)=G^3-(2,0)$. 
	
	To obtain these fixed points for $g$ which are $3$-periodic for $f$ we consider two different cases:
	
	\smallskip
	\[
	 \mathrm{Case}\ (1):\ U_1 \subset A, \qquad  \mathrm{Case}\ (2):\ U_1 \cap \partial A \neq \emptyset
	\]
	\smallskip
	
	Since each disk visits the other, there exist $x_0 \in U_0$, $x_1 \in U_1 \cap A$, and $n_0, n_1 \in \mathbb{N}$ such that $f^{n_0}(x_0) \in U_1 \cap A$ and $f^{n_1}(x_1) \in U_0$. From here, we proceed by discussing the two cases mentioned above.

		
	\bigskip	
		
\textbf{Case (1):} $U_1\subset A$.

	One may try to perform two small perturbations supported in $U_0\cup U_1$ in order to create two fixed points, one with rotation number $0$ and the other with rotation number $1$. If, in addition, we can guarantee the curve intersection property for the perturbed map, then the result follows from Claim~(1) and Corollary~\ref{corolariorealizacionPatrice}. The difficulty is that this strategy fails without additional assumptions, or even in the non-wandering setting, since the perturbation may destroy the non-wandering property. We shall see, however, that in this case no non-wandering or area-preserving assumption is actually needed to guarantee the existence of two periodic points with different rotation numbers. Indeed, the hypotheses on $U_0$ and $U_1$ allow us to construct suitable periodic chains of disks for different lifts of $f$.
	
	\medskip
	
	Consider a homeomorphism $\varphi$ that coincides with the identity outside two closed disks $S_0 \subset U_0$ and $S_1 \subset U_1$, with $f^n(x_0)$ and $f^n(x_1)$ outside of $(S_0 \cup S_1)$ for $n = 0, \dots, \max\{n_0, n_1\}+2$, such that the new map $g := \varphi \circ f \in \mathrm{Homeo}_{\,0}(\mathbb{A})$ satisfies:
	\begin{itemize}
		\item There exist fixed points $z_0, z_3$ of $g$ in $U_0, U_1$, respectively, with
		\[
		\rho_g(z_0, z_3) = 1.
		\]
		\item $U_0$ and $U_1$ form a $3$-dpd and each disk visits the other for the map $g$.
	\end{itemize}
	
	\medskip
	
	Observe that necessarily
	\[
	\bigcup_{i=1}^3 f^i(U_0)=\bigcup_{i=1}^3 g^i(U_0), \qquad
	\bigcup_{i=1}^3 f^i(U_1)=\bigcup_{i=1}^3 g^i(U_1).
	\]
	
	In this case, define
	\[
	h:=g^3.
	\]
	Let $G$ be the lift of $g$ fixing every lift of $z_0$. In particular, $G$ coincides with $F$ outside the lifted support of the perturbation. Let $H$ be the lift of $h$ defined by
	\[
	H:=G^3-(1,0).
	\]
	
	In particular, $\rho_H(z_0)=-1$ and $\rho_H(z_3)=2$. The next step is to show that $H$ admits a periodic chain of disks.
	
	Since the support of the perturbation does not meet the truncated orbits of $x_0$ and $x_1$, we can extract from the truncated orbits $f^n(x_0)$ for $n=0,\dots,n_0+2$ and $f^n(x_1)$ for $n=0,\dots,n_1+2$ two truncated orbits of $h$ such that, for certain $m_0,m_1\in\mathbb{N}$,
	\[
	h^{m_0}(x_0)\in\bigcup_{i=0}^2g^i(U_1), \qquad
	h^{m_1}(x_1)\in\bigcup_{i=0}^2g^i(U_0).
	\]
	
	Let $V_1$ be the element of $\{U_1,g(U_1),g^2(U_1)\}$ containing $h^{m_0}(x_0)$, and let $V_0$ be the element of $\{U_0,g(U_0),g^2(U_0)\}$ containing $h^{m_1}(x_1)$. For the construction of the periodic chain of disks, it will be convenient to assume that $x_1$ belongs to the same disk $V_1$ as $h^{m_0}(x_0)$. Recall that the $g$-orbit of $x_1$ splits into three $h=g^3$-orbits, corresponding to iterates of the form $3k$, $3k+1$, and $3k+2$. Thus, up to replacing $x_1$ by $g(x_1)$ or $g^2(x_1)$, we may assume that both $h^{m_0}(x_0)$ and $x_1$ belong to $V_1$.
	
	Note that the perturbations used to create fixed points inside $U_0$ and $U_1$ ensure that
	$h(V_0)\cap (U_0\cap V_0)\neq\emptyset$ and
	$h(V_1)\cap (U_1\cap V_1)\neq\emptyset$, respectively. These intersections will be used to close the periodic chain of disks constructed below.
	
	\medskip
	
	If $\tilde{U}_0$ is a lift of $U_0$, denote by $\tilde{V}_0$ the lift of $V_0$ contained in $\bigcup_{i=0}^2G^i(\tilde{U}_0)$. Let $\tilde{V}_1$ be the lift of $V_1$ such that
	\[
	H^{m_0}(\tilde{U}_0)\cap\tilde{V}_1\neq\emptyset,
	\]
	and let $\tilde{U}_1$ be the lift of $U_1$ that intersects $\tilde{V}_1$.
	
Because of the particular lift $H$ of $h$ we chose and the $3$-dpd property, we have
\begin{equation}
	\begin{array}{c}
		H(\tilde{U}_0)\cap \tilde{U}_0=\emptyset, \\[2pt]
		H(\tilde{V}_0)\cap\tilde{V}_0=\emptyset, \\[2pt]
		H(\tilde{V}_1)\cap\tilde{V}_1=\emptyset.
	\end{array}
	\label{eq:freedisks}
\end{equation}

Moreover, the fixed points $z_0 \in U_0\cap V_0$ and $z_3 \in U_1\cap V_1$ ensure that
\begin{equation}
	\begin{array}{c}
		H(\tilde{V}_1)\cap\big((\tilde{V}_1\cap\tilde{U}_1)+(2,0)\big)\neq\emptyset, \\[2pt]
		H(\tilde{V}_0)\cap\big((\tilde{V}_0\cap\tilde{U}_0)-(1,0)\big)\neq\emptyset.
	\end{array}
	\label{eq:nonemptyintersectiondisks}
\end{equation}
	
	Moreover, there exists $k\in\mathbb{Z}$ such that
	\[
	H^{m_1}(\tilde{V}_1)\cap\big(\tilde{V}_0+(k,0)\big)\neq\emptyset.
	\]
	
	We are now ready to define a collection of disks in order to obtain a periodic chain of disks, following Franks' construction in \cite{Franks(PoincBirkh)} and distinguishing two cases.
	
	\medskip
	
	Case $k\geq1$. Consider:
	\begin{itemize}
		\item $B_0=\tilde{U}_0$,
		\item $B_1=\tilde{V}_1$,
		\item $B_i=\tilde{V}_0+(k+2-i,0)$ for $i=2,\dots,k+1$.
	\end{itemize}
	
	\medskip
	
	Case $k<1$. Consider:
	\begin{itemize}
		\item $B_0=\tilde{U}_0$,
		\item $B_1=\tilde{V}_1$,
		\item $B_i=\tilde{V}_1+(2(i-1),0)$ for $i=2,\dots,-k+2$,
		\item $B_{-k+2+j}=\tilde{V}_0+(-k+3-j,0)$ for $j=1,\dots,-k+2$.
	\end{itemize}
	
	\medskip
	
	Because of the intersection properties in \eqref{eq:freedisks} and \eqref{eq:nonemptyintersectiondisks}, we obtain a periodic chain of disks in both cases. More precisely, if $k\geq 1$, setting $N=k+1$, we have $B_N=\tilde{V}_0+(1,0)$, and the collection $B_0,\dots,B_N$ is a periodic chain of disks. The same holds in the case $k<1$, with $N=-2k+4$. See Figure~\ref{fig:periodicchaincase1} for an illustration of the case $k=-1$.
	
	\medskip
	
	\begin{figure}[h!]
		\centering
		\includegraphics[width=380 pt]{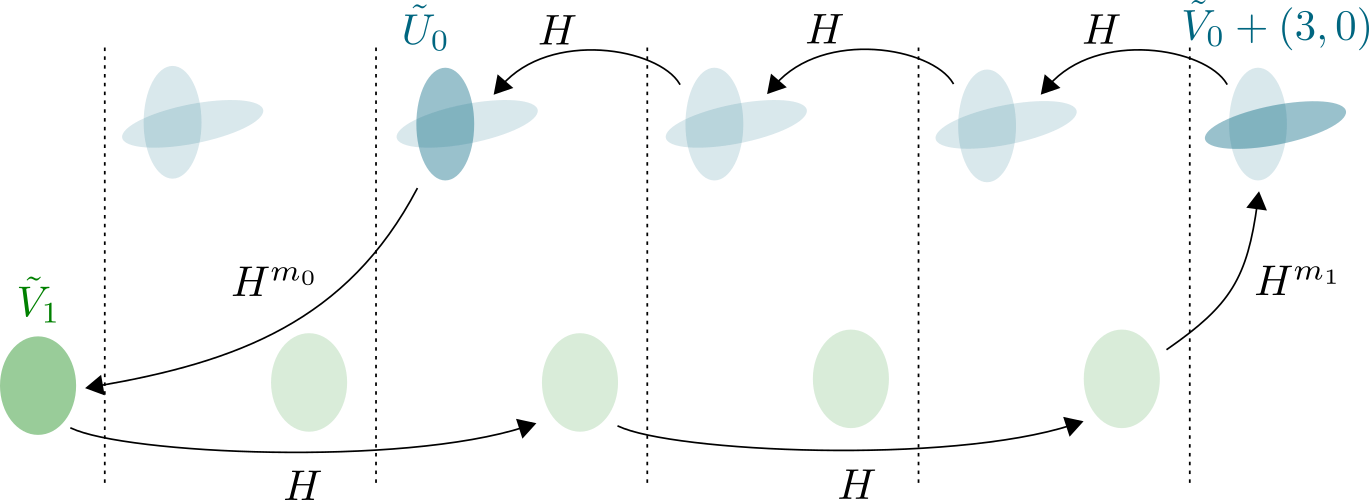}
		\caption{\small{Case (1): construction of a periodic chain of disks when $k=-1$.}}
		\label{fig:periodicchaincase1}
	\end{figure}
	
	\medskip
	
	Applying Proposition~\ref{proposicionFranks}, we obtain a fixed point $z_1$ for $H$. In particular,
	\[
	\rho_{z_1}(H)=0,
	\]
	which implies that $\rho_{z_1}(G^3)=1$. Therefore, $z_1$ is a $3$-periodic point of $g$ satisfying
	\[
	\rho_{z_1}(G)=\frac13.
	\]
	
	As observed before, $z_1$ is also a periodic point of $f$, and
	\[
	\rho_{z_1}(F)=\frac13.
	\]
	
	By the same argument, considering instead
	\[
	H=G^3-(2,0),
	\]
	we obtain another $3$-periodic point $z_2$ satisfying
	\[
	\rho_{z_2}(G)=\rho_{z_2}(F)=\frac23.
	\]
	
	Hence, the rotation set of $A$ for $F$ contains the interval $[1/3,2/3]$ and contains two periodic points realizing the boundary points of this interval. Finally, if $A$ satisfies the curve intersection property or if it is bounded below or above, we can apply Corollary~\ref{corolariorealizacionPatrice} in order to guarantee that any rational point in $(1/3,2/3)$ is realized by a point in $A$. This proves the proposition for $n=3$, $\rho=1$, in Case~(1).
	
	\bigskip


\textbf{Case (2):} $U_1\cap \partial A \neq \emptyset$.

As mentioned above, one can try to perform a perturbation supported in $U_0$ to create a fixed point and then apply Corollary~\ref{corolariorealizacionPatrice}. The difficulty arises again in the non-wandering setting, since after performing a perturbation we cannot guarantee that the perturbed map still satisfies the curve intersection property. Therefore, in order to apply Corollary~\ref{corolariorealizacionPatrice}, we need to find a periodic point of the original map $f$.

	\medskip
	Consider a homeomorphism $\varphi$ which is equal to the identity outside a closed disk $S_0 \subset U_0$, with $f^n(x_0)$ and $f^n(x_1)$ disjoint from $S_0$ for $n=0,\dots,\max\{n_0,n_1\}$, such that the new map $g := \varphi \circ f \in \mathrm{Homeo}_{\,0}(\mathbb{A})$ has a fixed point $z_0\in U_0$. The new map $g$ satisfies:
	\begin{itemize}
		\item if $G$ is the lift of $g$ which coincides with $F$ outside the lifts of $S_0$, then
		\[
		\rho_G(z_0)=0;
		\]
		\item $U_0$ and $U_1$ form a $3$-dpd, and each disk visits the other for the map $g$;
		\item the union of the first iterates of $U_0$ and $U_1$ by $g$ remains inessential and, moreover,
		\[
		\bigcup_{i=1}^3 f^i(U_0)=\bigcup_{i=1}^3 g^i(U_0), \qquad
		\bigcup_{i=1}^3 f^i(U_1)=\bigcup_{i=1}^3 g^i(U_1).
		\]
	\end{itemize}
	
	\medskip
	
	Let $\mathcal{X}$ be the boundary component of $A$ meeting $U_1$, which must be a circloid. By Proposition~\ref{primeendsrotationthmconservative}, it has a singleton rotation set satisfying
	\[
		\rho_\mathcal{X}(G)=\rho_\mathcal{X}(F)\in[2/3,\,4/3].
	\]
	\medskip
	
	In order to apply Corollary~\ref{corolariorealizacionPatrice} it is enough to show that $f$ presents a periodic point whose rotation number for $F$ is $1/3$. 
	
	\medskip
	
	In view of Lemma~\ref{lematecnicoimagenpancitas}, there exists an ends-connector $K$ such that, if $K^+=K\cap\overline{A}$ and $\tilde{A}_K$ is a lift of $\overline{A}\setminus K$, then the family of maximal cross-sections defined by $\partial U_1\cap A$, namely $\mathcal{W}$, can be lifted so as to be contained in $\tilde{A}_K+(\ell,0)$ for a unique integer $\ell$. Up to relabeling the lifts of $K$, we may assume that $\ell=0$.

	By Corollary~\ref{corodisjointK} and Corollary~\ref{coropancitasencerradas}, it is possible to choose such an ends-connector $K$ so that, simultaneously, each of the sets $U_1$, $g(U_1)$, and $g^2(U_1)$ defines a family of maximal cross-sections whose lifts are all contained in a single copy of $\tilde{A}_K$. Recall that $g(\mathcal{W})$ corresponds to the family of maximal cross-sections of $g(U_1)$. By Proposition~\ref{proppancitasincluidas}, we know that, for $i=0,1,2,3$,
	\[
	G^i(\widetilde{\mathcal{W}})\subseteq\tilde{A}_K+(i,0),
	\]
	where $G$ is the lift of $g$ fixing every lift of $z_0$. Define
	\[
	h:=g^3.
	\]
	
The support of the perturbation $\varphi$ does not meet the truncated orbits of $x_0$ and $x_1$. Hence, we can again extract from the truncated orbits $f^n(x_0)$ for $n=0,\dots,n_0+2$ and $f^n(x_1)$ for $n=0,\dots,n_1+2$ two truncated orbits of $h$ such that, for certain $m_0,m_1\in\mathbb{N}$,
\[
h^{m_0}(x_0)\in\bigcup_{i=0}^2 g^i(U_1 \cap A), \qquad
h^{m_1}(x_1)\in\bigcup_{i=0}^2 g^i(U_0).
\]

Similarly to Case~(1), up to replacing $x_0$ by $g(x_0)$ or $g^2(x_0)$, we may assume that the orbit of $x_0$ under $h$ passes through $U_1$, namely,
\[
h^{m_0}(x_0)\in U_1.
\]

In this context, we denote by $V_0^0$ and $V_0^1$ the elements of $\{U_0,g(U_0),g^2(U_0)\}$ containing $x_0$ and $h^{m_1}(x_1)$, respectively.

Note that $U_1\cap A$ need not be connected, which introduces additional difficulties compared with Case~(1). In particular, the points $h^{m_0}(x_0)$ and $x_1$ may lie in different connected components of $U_1\cap A$. Let $W_0$ and $W_1$ be the maximal cross-sections in $\mathcal{W}$ such that
\[
h^{m_0}(x_0)\in W_0,\qquad
x_1\in W_1.
\]
	
	Once the cross-sections involved in the trajectories of $x_0$ and $x_1$, namely $W_0$ and $W_1$, have been identified, we can improve the choice of the ends-connector $K$. Consider $\mathcal{W}'\subseteq\mathcal{W}$ given by the union of $g^i(W_0)$ and $g^j(W_1)$ for $i,j=0,1,2,3$. Choose a ray $r$ in $\mathcal{U}^+(\mathcal{X})$ from $+\infty$ landing on $\mathcal{X}$ such that $r$ lies in a small neighborhood of $K$ and satisfies
	\begin{itemize}
		\item $r\cap A$ neither intersects nor separates the cross-sections $g^i(W_0)$ and $g^j(W_1)$ for $i,j=0,1,2,3$, and
		\item if $\tilde{A}_r$ is the lift of $A\setminus r$ corresponding to $\tilde{A}_K$, and $\widetilde{\mathcal{W}}'$ is the lift of $\mathcal{W}'$ in $\tilde{A}_r$, then
		\[
		G^i(\widetilde{\mathcal{W}}')\subseteq\tilde{A}_r+(i,0).
		\]
		
	\end{itemize}
	
	We write $A_r:= \pi(\tilde{A}_r)$ and $\tilde{W}_0,\tilde{W}_1$ the respective lifts of $W_0,W_1$ in $\widetilde{\mathcal{W}}'$. 
	
	The following claim provides the free disk that will play the role of the disk translated with displacement $(2,0)$ in the construction of the periodic chain of disks. See Figure~\ref{fig:setBclaim} for an illustration of the construction.
	
	\begin{claimsinnum}[2]
		For any neighborhood of $\mathcal{X}\setminus r$ intersected with $A_r$, there exists an open disk $B$ included in that neighborhood such that
		\begin{itemize}
			\item $B\cap g^i(W_s)\neq\emptyset$ for $i=0,1,2,3$ and $s=0,1$;
			\item the interior of $B\cup W_0\cup W_1$ is an open topological disk;
			\item $B\cup W_0\cup W_1$ is disjoint from $U_0\cup g(U_0)\cup g^2(U_0)\cup g^3(U_0)$;
			\item if $\tilde{B}$ is a lift of $B$ included in $\tilde{A}_r$, then
			\[
			G^3(\tilde{B}\cup \tilde{W}_0 \cup \tilde{W}_1) \cap (\tilde{B}\cup \tilde{W}_0 \cup \tilde{W}_1) = \emptyset.
			\]
		\end{itemize}
	\end{claimsinnum}
	
	To prove the claim, we can assume that $A=\mathcal{U}^+(\mathcal{X})$. As done previously, we consider the prime-ends compactification of $A$, where it is identified with $S^1\times (0,+\infty)$. Recall that under the prime-ends compactification, the interiors of the annuli are identified, crosscuts of $\partial U_1 \cap A$ correspond to crosscuts in $S^1\times (0,+\infty)$, rays landing on $\mathcal{X}$ correspond to rays landing on $S^1\times\{0\}$, and neighborhoods of the end of $A$ associated with $\mathcal{X}$ correspond to neighborhoods of $S^1\times\{0\}$ in $S^1\times [0,+\infty)$. The map obtained on $S^1\times (0,+\infty)$ is conjugate to $g$ on $A$ and extends to the boundary. As we have done this construction before, for instance in Proposition~\ref{primeendsrotationproposition}, Lemma~\ref{lemmacorolariorealizacionPatrice}, or the Preliminaries, we do not present here the entire formal construction and identifications. As commented in the mentioned proofs, with an abuse of notation supported by the corresponding identification, we may assume $A=S^1\times (0,+\infty)$.
	
	For simplicity, also assume that the lifts $\tilde r$ and $\tilde r+(1,0)$ of $r$, which bound $\tilde A_r$, land at $(0,0)$ and $(1,0)$, respectively.
	
	Since the lift of $\partial\mathcal{W}'\cap\partial A$ is contained in $[0,1]\times\{0\}$, the condition
	\[
	G^i(\widetilde{\mathcal{W}}')\subseteq \tilde{A}_r+(i,0)
	\]
	implies that
	\[
	G^3\big([0,1]\times\{0\}\big)\cap\big([3,4]\times\{0\}\big)\neq\emptyset.
	\]
	
	Notice that $G^3\big([0,1]\times\{0\}\big)$ is an interval of length $1$. In particular, the image of one of the boundary points under $G^3$ must belong to $[3,4)\times\{0\}$. Therefore, it holds that
	\[
	G^3\big([0,1]\times\{0\}\big) \subseteq (2,5)\times \{0\}.
	\]

	\begin{figure}[h!]
		\centering
		\includegraphics[width=140 pt]{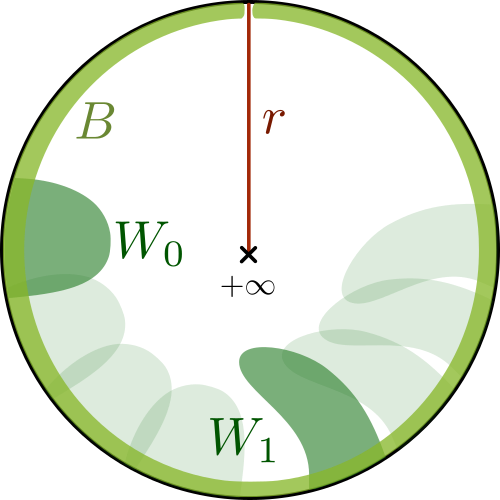}
		\caption{\centering{\small{Construction of the set $B$. \\
					The iterates of the cross-sections $W_0$ and $W_1$ are shown in light green.}}}
		\label{fig:setBclaim}
	\end{figure}

	Define the following topological disk:
	\[
	B:=\left(S^1\times (0,\epsilon)\right) \cap A_r \subset S^1\times [0,+\infty).
	\]
	
	Let $\tilde{B}$ be the lift of $B$ included in $\tilde{A}_r$. Notice that $[0,1]\times \{0\} \subset \partial \tilde{B}$. Since $\tilde{W}_0$,$\tilde{W}_1$ and their first three iterates have their endpoints in this interval, it follows that $\tilde{B}$ meets the first three iterates of $\tilde{W}_0$ and $\tilde{W}_1$.
	
	Given any neighborhood of $S^1\times \{0\}$, there exists $\epsilon>0$ small enough such that $B$ is included in it and satisfies
	\[
	G^3(\tilde{B}) \subseteq (2,5)\times [0,+\infty), \qquad G^3(\tilde{B})\cap \tilde{B}= \emptyset.
	\]
	
	Together with the fact that $G^3(\tilde{W}_0)$ and $G^3(\tilde{W}_1)$ do not intersect $\tilde{A}_r$, we obtain
	\[
	G^3(\tilde{B}\cup \tilde{W}_0 \cup \tilde{W}_1) \cap (\tilde{B}\cup \tilde{W}_0 \cup \tilde{W}_1) = \emptyset.
	\]

	Finally, note that $\epsilon$ can be taken small enough to guarantee that the interior of $B\cup W_0\cup W_1$ is also an open topological disk and that $B$ does not meet $U_0\cup g(U_0)\cup g^2(U_0)\cup g^3(U_0)$.
	This completes the proof of the claim.
\bigskip

Now, using the set $B$ given by the claim above, we are going to build a chain of disks for suitable lifts of $g^3$ in order to obtain a periodic orbit of rotation number $1/3$. It only remains to complete the chain of disks by choosing disks playing the role of the lifts of $V_1$ in case~(1).

Let $H$ be the lift of $h$ defined by
\[
H:=G^3-(1,0).
\]

Fix a lift $\tilde{U}_0$ of $U_0$, and let $\tilde{V}_0^0$ and $\tilde{V}_0^1$ denote the lifts of $V_0^0$ and $V_0^1$, respectively, that meet $\tilde{U}_0$.

We denote by $\tilde{U}_1$ the lift of $U_1$ such that
\[
H^{m_0}(\tilde{V}_0^0)\cap\tilde{U}_1\neq\emptyset.
\]
Let $\tilde{W}_0$ and $\tilde{W}_1$ be the lifts of $W_0$ and $W_1$, respectively, contained in $\widetilde{\mathcal{W}}'$. 
Also, let $\tilde{B}$ be the lift of the set $B$ contained in $\tilde{A}_r$ given by Claim~(2) above. We set
\[
M:=\tilde{B}\cup \mathrm{Int}\left(\tilde{W}_0\cup\tilde{W}_1\right) \subset \tilde{A}_r.
\]

Claim~(2) ensures that $M$ is a topological open disk, $\pi(M)$ is disjoint from $U_0 \cup g(U_0) \cup g^2(U_0)$ and, for the lift $H$, we have
\begin{equation}
	\begin{array}{c}
		H(M) \cap M = \emptyset, \\[2pt]
		H(M) \cap \left(M+(2,0)\right) \neq \emptyset.
	\end{array}
	\label{eq:Mcase2}
\end{equation}
\smallskip

By the choice of the lifts and of the disk $M$, we have
\begin{equation}
	\begin{array}{c}
		H^{m_0}(\tilde{V}_0^0)\cap M\neq\emptyset, \\[2pt]
		H^{m_1}(M)\cap\big(\tilde{V}_0^1+(k,0)\big)\neq\emptyset
		\quad\text{for some } k\in\mathbb{Z}.
	\end{array}
	\label{eq:connectionscase2}
\end{equation}

By the $3$-dpd property and since $U_0\cap V_0^1$ contains the fixed point $z_0$, we have
\begin{equation}
	\begin{array}{c}
		H(\tilde{V}_0^1)\cap\tilde{V}_0^1=\emptyset, \\[2pt]
		H(\tilde{V}_0^1)\cap
		\big((\tilde{V}_0^1\cap\tilde{U}_0)-(1,0)\big)
		\neq\emptyset.
	\end{array}
	\label{eq:propertiesV0case2}
\end{equation}

With these definitions in place, we are ready to construct the periodic chain of disks, again distinguishing the cases $k\geq1$ and $k<1$.

\medskip

Case $k\geq1$. Consider:
\begin{itemize}
	\item $B_0=\tilde{V}_0^0$,
	\item $B_1=M$,
	\item $B_i=\tilde{V}_0^1+(k+2-i,0)$ for $i=2,\dots,k+1$.
\end{itemize}

\medskip

Case $k<1$. Consider:
\begin{itemize}
	\item $B_0=\tilde{V}_0^0$,
	\item $B_1=M$,
	\item $B_i=M+(2(i-1),0)$ for $i=2,\dots,-k+2$,
	\item $B_{-k+2+j}=\tilde{V}_0^1+(-k+3-j,0)$ for $j=1,\dots,-k+2$.
\end{itemize}

\medskip

Because of the intersection properties in \eqref{eq:connectionscase2}, \eqref{eq:Mcase2}, and \eqref{eq:propertiesV0case2}, we obtain a periodic chain of disks in both cases. More precisely, if $k\geq1$, setting $N=k+1$, we have $B_N=\tilde{V}_0^1+(1,0)$, and the collection $B_0,\dots,B_N$ is a periodic chain of disks. The same holds in the case $k<1$, with $N=-2k+4$. See Figure~\ref{chaindiskscase2} for an illustration of the case $k=-1$.
		
	\begin{figure}[h!]
			
			\centering
			\includegraphics[width=380 pt]{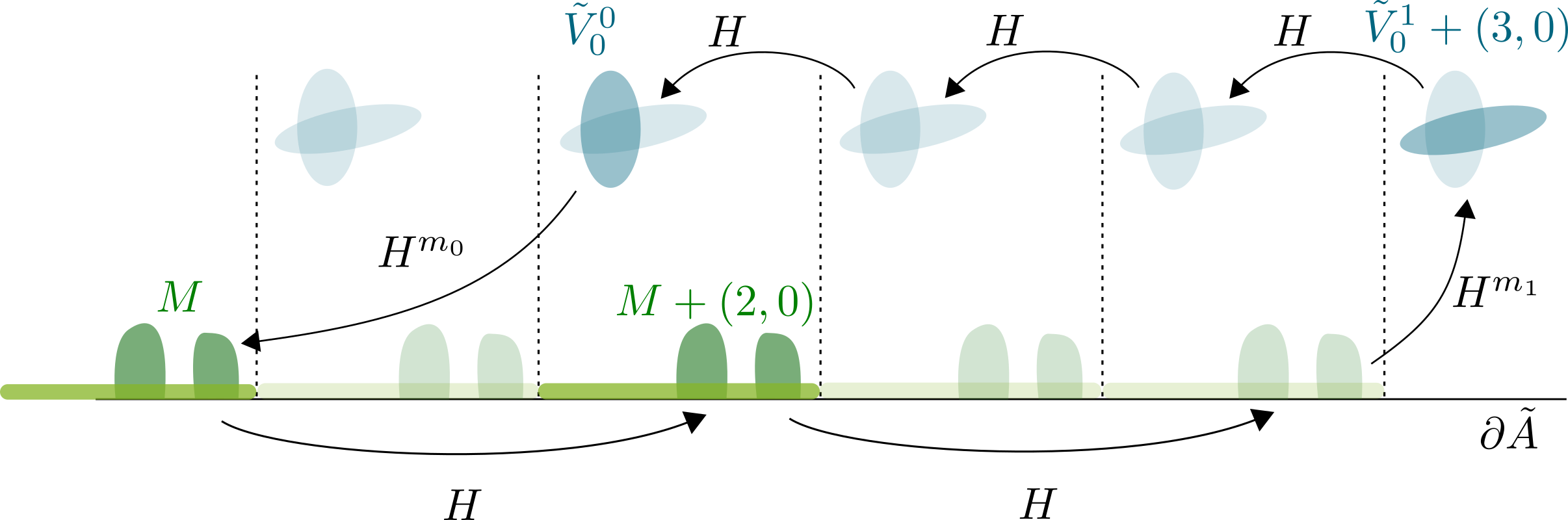}
			
			\caption{\small{Case (2): construction of a periodic chain of disks when $k=-1$}}
			\label{chaindiskscase2}
	\end{figure}

	\medskip
		
Applying Proposition~\ref{proposicionFranks}, which applies in the same way to $\mathbb{H}^2$ as it does to $\mathbb{R}^2$, we obtain a fixed point $z_1$ of $H$ such that $\rho_{z_1}(H)=0$. This implies that $\rho_{z_1}(G^3)=1$. Therefore, $z_1$ is a $3$--periodic point of $g$ satisfying
\[
\rho_{z_1}(G)=\tfrac{1}{3}.
\]

Consequently, by Claim~(1), $z_1$ is also a periodic point of $f$ with
\[
\rho_{z_1}(F)=\tfrac{1}{3}.
\]

\bigskip

Observe that we cannot apply the same construction to the lift $G^3-(2,0)$. The problem is that we cannot ensure that the disk $M$ is free for the map $G^3-(2,0)$. Nevertheless, we know that
\[
\rho_{\mathcal{X}}(F)\in [2/3,\,4/3].
\]

Hence, we are in the third situation of Corollary~\ref{corolariorealizacionPatrice}. Applying the corollary, we conclude that every rational number in $(1/3,2/3)$ is realized in $A$, thus completing the proof of the proposition for $n=3$, $\rho=1$, in Case~(2).

\bigskip

If $U_0$ and $U_1$ form an $n$--dpd with $\rho>1$ and $n\geq3$, then, when $U_1\cap\mathcal{X}=\emptyset$, we argue exactly as in Case~(1), using the lifts $G^n-(1,0)$ and $G^n-(n\rho-1,0)$, to obtain two $n$--periodic points with rotation numbers
\[
\rho_{z_1}(F)=\tfrac{1}{n}, \qquad
\rho_{z_2}(F)=\rho-\tfrac{1}{n}.
\]
If instead $U_1$ meets $\mathcal{X}$, then
\[
\rho_{\mathcal{X}}(F)\in\left(\rho-\tfrac{1}{n},\rho+\tfrac{1}{n}\right),
\]
and we conclude exactly as in Case~(2).

\medskip

The same constructions cover the remaining cases in the statement. If $n=2$ and $\rho\geq2$, they yield two periodic points with rotation numbers
\[
\frac{1}{2}
\qquad\text{and}\qquad
\rho-\frac{1}{2}.
\]
If $n=1$ and $\rho\geq3$, they yield two periodic points with rotation numbers
\[
1
\qquad\text{and}\qquad
\rho-1.
\]
Therefore, in all cases, the rotation set of $A$ contains
\[
\left[\frac{1}{n},\rho-\frac{1}{n}\right],
\]
with both endpoints realized by periodic points.

\end{proof}

\begin{rmk}
In Case (2) of the previous proposition, we use the assumption that $f$ is non-wandering or area-preserving to guarantee that the rotation set of $\mathcal{X}$ is a singleton, which can then be estimated using Theorem~\ref{primeendsrotationthmconservative}, and to control the behavior of the maximal crosscuts of $U_1 \cap A$.
\end{rmk}

\bibliography{biblio}
\bibliographystyle{plain}

\end{document}